\documentclass[reqno]{amsart}
\usepackage[utf8]{inputenc}
\usepackage{amssymb}
\usepackage{mdframed}
\usepackage{bm,bbm}
\usepackage{xcolor}
\usepackage{stmaryrd}
\usepackage{tikz-cd}\usetikzlibrary{decorations.pathmorphing}
\usepackage{comment}
\usepackage[all,cmtip]{xy}
\usepackage{microtype}
\usepackage{graphicx}
\usepackage{amsthm,etoolbox}
\usepackage{eucal, mathbbol}

\usepackage{hyperref}
\hypersetup{hidelinks, colorlinks}
\hypersetup{
    citecolor = {teal},
    linkcolor = {black},
}
\usepackage[capitalise]{cleveref}

\numberwithin{equation}{section}

\calclayout

\usepackage[left= 3.4 cm, right= 3.4 cm, top= 2.8 cm, bottom=2.7 cm, foot= 1.2 cm, marginparwidth=2.7 cm, marginparsep=0.5 cm]{geometry}

\numberwithin{equation}{section}
\usepackage{soul}

\usepackage{hyperref}
\hypersetup{colorlinks}
\definecolor{darkred}{rgb}{0.5,0,0}
\definecolor{darkgreen}{rgb}{0,0.5,0}
\definecolor{darkblue}{rgb}{0,0,0.5}
\hypersetup{colorlinks, linkcolor=darkblue, filecolor=darkgreen, urlcolor=darkred, citecolor=darkblue}
\makeatletter 
\@addtoreset{equation}{section}
\makeatother  

\numberwithin{equation}{section}

\newtheorem{thm}{Theorem}[section]
\newtheorem{cor}[thm]{Corollary}

\newtheorem{conj}[thm]{Conjecture}

\newtheorem{prop}[thm]{Proposition}

\newtheorem{lemma}[thm]{Lemma}
\newtheorem{def-lemma}[thm]{Definition-Lemma}
\theoremstyle{definition}
\newtheorem{defn}[thm]{Definition}
\theoremstyle{remark}

\theoremstyle{remark}
\newtheorem{rem}[thm]{Remark}
\AtBeginEnvironment{rem}{%
    \small
  \pushQED{\qed}%
}
\AtEndEnvironment{rem}{\popQED}

\newcommand{\beq}{\begin{equation}}
\newcommand{\eeq}{\end{equation}}
\newcommand{\beqn}{\begin{equation*}}
\newcommand{\eeqn}{\end{equation*}}
\newcommand{\ov}{\overline}

\newcommand{\calH}{\mathcal{H}}

\makeatletter
\newcommand{\make@circled}[2]{%
  \ooalign{$\m@th#1\smallbigcirc{#1}$\cr\hidewidth$\m@th#1#2$\hidewidth\cr}%
}
\newcommand{\smallbigcirc}[1]{%
  \vcenter{\hbox{\scalebox{0.77778}{$\m@th#1\bigcirc$}}}%
}
\makeatother

\makeatletter
\newcommand{\colim@}[2]{%
  \vtop{\m@th\ialign{##\cr
    \hfil$#1\operator@font colim$\hfil\cr
    \noalign{\nointerlineskip\kern1.5\ex@}#2\cr
    \noalign{\nointerlineskip\kern-\ex@}\cr}}%
}
\newcommand{\colim}{%
  \mathop{\mathpalette\colim@{\rightarrowfill@\textstyle}}\nmlimits@
}
\makeatother

\title[]{Arithmetic geometry of quantum connections on Calabi--Yau $3$-folds}

\author{Shaoyun Bai}
\address{Department of Mathematics, MIT, Boston, MA, 02139, USA}
\email{shaoyunb@mit.edu}

\author{Jae Hee Lee}
\address{Department of Mathematics, Stanford University, Stanford, CA, 94305, USA}
\email{jaeheelee@stanford.edu}

\author{Daniel Pomerleano}
\address{University of Massachusetts, Boston, 100 William T, Morrissey Blvd, Boston, MA 02125, USA}
\email{Daniel.Pomerleano@umb.edu}

\thanks{The first-named author is supported by NSF DMS-2404843. The third-named author is supported by NSF DMS-2306204.}

\begin{document}

\maketitle

\begin{abstract}
Fix a prime $p > 3$. Working over $\mathbb{Z}_p$, we show that the quantum connection of any closed Calabi--Yau threefold gives rise to a Fontaine--Laffaile module when restricted to the even degree and torsion-free part of $p$-adic quantum cohomology, whose associated Frobenius endomorphism has leading order term prescribed by the $p$-adic Gamma class. After reducing mod $p$, the divided Frobenius endomorphism defines an analogue of the inverse Cartier operator on mod $p$ quantum cohomology. We establish an $A$-model analogue of a classical result due to Katz: the conjugation of the $p$-curvature of the mod $p$ quantum connection by the inverse Cartier operator is equal to the Frobenius pullback of the quantum product, the $A$-model counterpart of the Kodaira--Spencer class. Moreover, we identify the quantum Steenrod operation with the $p$-curvature of the mod $p$ quantum connection in this setting for any prime $p$.

We propose several conjectures concerning how these arithmetic structures may extend to quantum connections on more general semi-positive symplectic manifolds.
\end{abstract}

\setcounter{tocdepth}{1}
\tableofcontents
\section{Introduction}
\subsection{Fontaine--Laffaille modules}
Let $p$ be an odd prime number. Fontaine--Laffaille modules were introduced in \cite{Fontaine_Laffaille} to capture the linear algebraic structures that exist on the de Rham cohomology of a smooth variety over $\mathbb{Z}_p$.  A (torsion-free) Fontaine--Laffaille module over $\mathbb{Z}_p$ is a triple 
\begin{align} \label{eq:strongdivisibility}
(\calH, \mathrm{Fil}^i, \Phi) \end{align}
consisting of a finite and free $\mathbb{Z}_p$-module $\calH$, a decreasing filtration $\mathrm{Fil}^i$ such that $\calH/\mathrm{Fil}^i$ is torsion-free for all $i$, and a map $\Phi$ called the Frobenius. The map $\Phi$ must satisfy the ``strong $p$-divisibility" property with respect to the filtration:  
    \[
    \Phi(\mathrm{Fil}^i) \subset p^i \calH, \quad \sum_{i} \frac{1}{p^i} \Phi(\mathrm{Fil}^i) = \calH.
    \] Fontaine and Laffaille define a fully faithful functor $T_{cris}$ from the category of Fontaine-Laffaille modules to the category of integral $p$-adic representations of the Galois group $\mathrm{Gal}(\overline{\mathbb{Q}}_p/\mathbb{Q}_p).$ 

In \cite[II.2.7]{FontaineMessing}, Fontaine and Messing establish the existence of a Fontaine-Laffaille module structure on the de Rham cohomology of a variety $\mathcal{Y}$ over $\mathbb{Z}_p$ of dimension $\operatorname{dim}(\mathcal{Y})< p$ whose de Rham cohomology is torsion free. In this setting, the filtration $\mathrm{Fil}^\bullet$ corresponds to the Hodge filtration and the Frobenius map $\Phi$ corresponds to the crystalline Frobenius under the comparison isomorphism between de Rham and crystalline cohomologies. There is a canonical isomorphism (\cite{deligneillusie,Kato}) in the derived category of Zariski sheaves on $\mathcal{Y}$:
\begin{align} \label{eq:Katzisomorphism}
(\hat{\Omega}^\bullet_{\mathcal{Y}}, p \cdot d_{\mathrm{dR}}) \xrightarrow{\sim} (\hat{\Omega}^\bullet_{\mathcal{Y}}, d_{\mathrm{dR}})
\end{align}
whose composition with the natural inclusion of complexes $(\hat{\Omega}^\bullet_{\mathcal{Y}}, d_{\mathrm{dR}}) \hookrightarrow (\hat{\Omega}^\bullet_{\mathcal{Y}}, p \cdot d_{\mathrm{dR}})$, defined in degree $i$ by the map $p^i \cdot \mathrm{Id}: \hat{\Omega}^i_{\mathcal{Y}} \to \hat{\Omega}^i_{\mathcal{Y}}$, recovers the crystalline Frobenius morphism. The existence of the isomorphism \eqref{eq:Katzisomorphism} implies the strong divisibility property of the crystalline Frobenius. When $\operatorname{dim}(\mathcal{Y})< p-1$, Fontaine and Messing (\cite[III.6.4]{FontaineMessing}) go on to show that the functor $T_{cris}$ identifies this module with the Galois representation on the \'etale cohomology groups $H_{et}^*(\mathcal{Y} \times_{\mathbb{Z}_p}\bar{\mathbb{Q}}_p,\mathbb{Z}_p)$.

In \cite{Faltings1989}, Faltings establishes direct analogues of the Fontaine--Messing results for families as well as generalizations to the logarithmic setting. Let $R$ be a ring over $\mathbb{Z}_p$ with a fixed absolute Frobenius lifting \( F \).\footnote{Typically, $R$ is the $p$-adic completion of the algebra of functions on a smooth affine scheme over $\mathbb{Z}_p$.} A relative Fontaine-Laffaille module over $R$ is now a quadruple $(\calH, \nabla, \mathrm{Fil}, \Phi)$ where: \begin{itemize} \item \( (\calH, \mathrm{Fil}) \) is a filtered free $R$-module equipped with a connection \( \nabla \) satisfying a ``Griffiths transversality" property with respect the filtration (see \S \ref{subsection:FL definitions} for precise definitions). \vskip 5 pt 
     \item The (relative) Frobenius is an $R$-linear morphism 
    \[
    \Phi : F^{*} \calH \longrightarrow \calH \] which satisfies the strong divisibility property \eqref{eq:strongdivisibility} and is horizontal with respect to the
    connections \( F^{*}\nabla \) on \( F^{*} \calH \) and \( \nabla \) on \( \calH \). \end{itemize} 
 
\subsubsection{The A-model Fontaine-Laffaille module}
 In the discussion below, $(X,\omega)$ will denote a symplectic Calabi--Yau $3$-fold. By this we mean a $6$-dimensional, simply connected, closed symplectic manifold $(X,\omega)$ with integral symplectic class $[\omega] \in H^2(X,\mathbb{Z})$ and such that $c_1(TX)=0$. Our first goal is to build a logarithmic Fontaine-Laffaille structure on the quantum cohomology of a Calabi-Yau 3-fold. Let us begin by describing the Novikov ring. Fix a ground commutative ring $\mathbb{K}$ and let $\bar{H}_2(X;\mathbb{Z})$ denote the quotient of $H_2(X;\mathbb{Z})$ by its torsion subgroup, $$\bar{H}_2(X;\mathbb{Z}):=\frac{H_2(X;\mathbb{Z})}{H_2(X;\mathbb{Z})_{tor}}.$$ The Novikov ring $\Lambda_{\mathbb{K}}$ will denote the ring whose elements are series:
\begin{equation}
\sum_{\beta \in \bar{H}_2(X; \mathbb{Z}), \int_\beta \omega \geq 0} c_\beta q^\beta, \quad c_\beta \in \mathbb{K}
\end{equation}
such that for any fixed $E \geq 0$, there are only finitely many $c_\beta \neq 0$ for all $\beta \in \bar{H}_2(X; \mathbb{Z})$ satisfying $\int_\beta \omega \leq E$. We next introduce a formal variable $u$ of degree 2 and consider the even part of quantum cohomology which is defined to be: \begin{align} QH^*(X;\Lambda_{\mathbb{K}})[u^{\pm}]:=H^{even}(X; \Lambda_\mathbb{K})[u,u^{-1}]. \end{align} 
Quantum multiplication by a class $b \in H^2(X; \mathbb{K})$ induces an endomorphism:  \begin{equation}
\mathbb{M}_b : H^{even}(X; \Lambda_\mathbb{K}) \to H^{even}(X; \Lambda_\mathbb{K}), \quad (\mathbb{M}_b(a_1) \cdot a_2) = \sum_{\beta} q^\beta \langle a_1, a_2, b \rangle_{0,3,\beta},
\end{equation}
where $\langle - \rangle_{0,3,\beta}$ denotes the 3-pointed Gromov–Witten correlator (with extension of scalars to $\mathbb{K}$). The small quantum connection is defined by:  
\begin{equation}
\nabla_b := \partial_b + u^{-1} \mathbb{M}_b : QH^*(X;\Lambda_{\mathbb{K}})[u^{\pm}] \to QH^*(X;\Lambda_{\mathbb{K}})[u^{\pm}],
\end{equation}
where $\partial_b$ denotes the derivation of  $\Lambda_{\mathbb{K}}$ associated to $b$. This defines a logarithmic connection with respect to the natural log structure on $\operatorname{Spf}(\Lambda_\mathbb{K})$ (see \S \ref{subsection:FL definitions}). The filtration on $QH^*(X;\Lambda_{\mathbb{K}})[u^{\pm}]$ is defined by: \begin{align}\operatorname{Fil}^i := \bigoplus_{3+k \geq i} u^kH^{even}(X; \Lambda_\mathbb{K}) \subset H^{even}(X; \Lambda_\mathbb{K})[u,u^{-1}]. \end{align}

At this stage we specialize $\mathbb{K}=\mathbb{Z}_p.$ We consider the ``standard" lift of Frobenius \begin{align} F: \Lambda_{\mathbb{Z}_p} \to \Lambda_{\mathbb{Z}_p}, \quad q^{\beta} \to q^{p\beta}. \end{align}

The A-model Frobenius intertwiner is determined by a characteristic class \begin{align} \Gamma_p(TX) \in H^*(X; \mathbb{Q}_p)[u,u^{-1}]\end{align} built out of Morita's $p$-adic Gamma function $\Gamma_p(z)$ (see \cite[\S 4]{bai2025p} or \S \ref{sect:FLonQH} below). In the case of Calabi-Yau 3-folds, the class $\Gamma_p(TX)$ simplifies to:  
\begin{align} \label{eq:Gammap}
\Gamma_p(TX) 
= 1 - u^{-3}\zeta_p(3) \cdot c_3(TX),\quad \zeta_p(3)= -\frac{1}{2}(\Gamma_p'''(0)-\Gamma_p'(0)^3).
\end{align}

We use $\Gamma_p(TX)$ to introduce a grading preserving $\mathbb{Z}_p$-linear map on cohomologies: \begin{align} \label{eq:b-frobenius} 
\Phi_0: H^*(X;\mathbb{Z}_p)[u,u^{-1}] \to H^*(X;\mathbb{Q}_p)[u,u^{-1}], \quad \Phi_0(u^kx) = p^{3+k} u^k\cdot(\Gamma_p(TM) \smile x). \end{align} 

Our first result is the following:

\begin{thm}\label{thm:integrality} Let $p$ be a prime for which $H^{even}(X;\mathbb{Z}_p)$ contains no torsion. Then there is a unique Frobenius map: \begin{align}\label{eq:Amodelcrystalline} \Phi: F^*QH^*(X; \Lambda_{\mathbb{Z}_p})[u^{\pm}] \to QH^*(X; \Lambda_{\mathbb{Z}_p})[u^{\pm},1/p] \end{align} such that: \begin{enumerate} \item Let $\Lambda_{+} \subset \Lambda_{\mathbb{Z}_{p}}$ denote the ideal of elements with strictly positive $\omega$-valuation. The ``leading order term" (with respect to the $\omega$-valuation) of $\Phi$ is given by \eqref{eq:b-frobenius}: \begin{align} \Phi = \Phi_0  \quad \operatorname{mod}\ \Lambda_{+}.\end{align}  \item Let $\mathcal{H}^{2k}$ denote the subspace of $QH^*(X; \Lambda_{\mathbb{Z}_p})[u^{\pm}]$ of degree $2k$. The map $\Phi$ satisfies:   \begin{align} \label{eq:keyintegralitycondition} \Phi(F^*\mathcal{H}^{2k}) \subset p^k\mathcal{H}^{2k},\quad k \in \mathbb{Z}.\end{align} \end{enumerate} \end{thm}

 We note that the elementary theory of $p$-adic differential equations (see \cite[\S 3]{bai2025p} or Lemma \ref{lem:initialtermdeterminesFrob} below) promises that there is at most one Frobenius intertwiner of the form \eqref{eq:Amodelcrystalline} whose leading order term is \eqref{eq:b-frobenius}. The main content of Theorem \ref{thm:integrality} is thus to construct such an intertwiner. In fact, ``most" connections do not admit compatible Frobenius maps (\cite[Remark 3.2]{Katz73}, \cite[Remark 2.3.11]{Kedlaya}). The construction of the intertwiner in this case relies on subtle integrality features of Gromov-Witten invariants on Calab-Yau 3-folds (see \S \ref{sect:Gopakumar}). 

 \begin{cor} \label{cor:FLclassical} Suppose $p>3$ and let $\mathcal{H}:=\mathcal{H}^0$ denote the degree zero piece of  $QH^*(X;\Lambda_{\mathbb{Z}_p})[u^{\pm}].$ The data $$(\mathcal{H},\nabla,\operatorname{Fil}^i,\Phi)$$ defines a logarithmic Fontaine--Laffaille module in the sense of Definition \ref{defn:fontaine}.\end{cor}

 To understand Corollary \ref{cor:FLclassical} it is helpful to notice that, because $u$ has degree 2, there are canonical isomorphisms: \begin{align} \mathcal{H}^0\otimes_{\mathbb{Z}_p} \mathbb{Z}_p[u,u^{-1}] \cong QH^*(X;\Lambda_{{\mathbb{Z}_p}})[u^{\pm}], \quad \mathcal{H}^0 \cong \bigoplus_{k=0}^3 u^{-k}H^{2k}(X;\Lambda_{{\mathbb{Z}_p}}).  \end{align} 

The paper \cite{smirnov24} uses a leading order term similar to \eqref{eq:b-frobenius} to construct Frobenius intertwiners for the quantum differential equation of $T^*\mathbb{P}^n$ and the paper \cite{bai2025p} also uses a very similar leading term to to construct overconvergent Frobenius structures on the quantum connection of certain monotone symplectic manifolds. The leading order term \eqref{eq:b-frobenius} is also motivated by an algebro-geometric conjecture due to Candelas--de la Ossa--van Straten \cite[\S 4.4]{candelas}.  Namely, suppose we are given a semi-stable family of Calabi-Yau 3-folds $$ \mathcal{Y}_\mathbb{Z} \to \operatorname{Spec}(\mathbb{Z}[1/N][\![t]\!]),$$
whose base change to $\mathbb{C}[\![t]\!]$ gives rise to a maximally unipotent family. We choose a prime $p \nmid N$ and consider the  family $\mathcal{Y} \to \mathbb{Z}_p[\![t]\!]$ given by base change to $\mathbb{Z}_p[\![t]\!]$. The logarithmic Fontaine-Laffaille theory of \cite{Faltings1989} gives rise to a limiting Fontaine--Laffaille module at $t=0$. Candelas--de la Ossa--van Straten have conjectured  that the limiting Frobenius endomorphism $\Phi_0$ on the logarithmic de Rham cohomology $H^3_{log}(\mathcal{Y}/\mathbb{Z}_p[\![t]\!])$ has a very constrained structure. For example, when $\operatorname{rank}_{\mathbb{Z}_p[\![t]\!]}(H^3_{log}(\mathcal{Y}/\mathbb{Z}_p[\![t]\!]))=4$ (the smallest possible rank), they conjecture that in a suitable basis, the limiting Frobenius matrix $\Phi_0$ is of the form:

\[
\Phi_0 = \begin{pmatrix} 
p^3 & 0 & 0 & 0 \\ 
0 & p^2 & 0 & 0 \\ 
0 & 0 & p & 0 \\ 
c \cdot \zeta_p(3)  & 0 & 0 & 1
\end{pmatrix}
\]

where $c \in \mathbb{Q}$ is a rational number and $\zeta_p(3)$ is the constant appearing in \eqref{eq:Gammap}. This conjecture has been verified for the mirror quintic 3-fold by Shapiro in \cite{Schapiro1, Schapiro2}. Theorem \ref{thm:integrality} provides, through the lens of mirror symmetry, more evidence for this algebro-geometric conjecture. In the opposite direction, the Candelas--de la Ossa--van Straten conjecture suggests a way to characterize the leading order term $\Phi_0$ term through the Fukaya category and noncommutative geometry; see Conjecture \ref{conj:open-closed} below.

\subsection{The reduction modulo $p$}

\subsubsection{Inverse Cartier, p-curvature, and Katz' formula}
Our remaining results concern the quantum connection in characteristic $p$, obtained as the mod $p$ reduction of the $p$-adic quantum connection. To motivate the next result we return to the setting a smooth and proper variety $\mathcal{Y}$ over $\mathbb{Z}_p$. The reduction mod $p$ of the divided Frobenius operator \eqref{eq:Katzisomorphism} gives rise to the the inverse Cartier isomorphism from \cite{deligneillusie}:  \begin{align} \label{eq:Cartierisomorphism}
\mathcal{C}^{-1}:\bigoplus_i\Omega^i_{\mathcal{Y}_{\mathbb{F}_p}} \xrightarrow{\sim} (\Omega^\bullet_{\mathcal{Y}_{\mathbb{F}_p}}, d_{\mathrm{dR}}).
\end{align} 

In fact, the mod $p$ reduction of any Fontaine--Laffaille module gives rise to divided Frobenius operators --- in this way we obtain what is known as a \emph{$p$-torsion Fontaine--Laffaille module} (see Definition \ref{defn:ptorsion_fontaine}). Assuming $p>3$, the divided Frobenius operators on the mod $p$ reduction of the Fontaine-Laffaille module from Theorem \ref{thm:integrality} assemble to give an A-model version of the Cartier operator:
\begin{align} \label{eq:AmodelCartier} \mathcal{C}^{-1}: F^*QH^*(X;\Lambda_{\mathbb{F}_p})[u,u^{-1}] \to QH^*(X;\Lambda_{\mathbb{F}_p})[u,u^{-1}]. \end{align}

\begin{rem} For a general $p$-torsion Fontaine-Laffaille module, the domain of the inverse Cartier map \eqref{eq:inversecartiergeneral} is an associated graded with respect to the filtration $\bar{\operatorname{Fil}}$. In quantum cohomology, the filtration is canonically split, and so we may identify the associated graded with $QH^*$ in \eqref{eq:AmodelCartier}.  \end{rem}

Over a field of characteristic $p$, the following fundamental operator known as the \emph{$p$-curvature} of $\nabla$ along $b$, 
\begin{equation}
        \psi_b := \nabla_b^p - \nabla_b : QH^*(X;\Lambda_{\mathbb{F}_p})[u]\langle \theta \rangle \to QH^*(X;\Lambda_{\mathbb{F}_p})[u]\langle \theta \rangle,
\end{equation}
is $\Lambda$-linear (cf. \cite[Section 5]{Kat70}) (perhaps surprisingly, since $\psi_b$ involves iterative compositions of the connection $\nabla_b$). In the algebro-geometric setting, a celebrated formula of Katz \cite[Theorem 3.2]{Kat72} relates the $p$-curvature of the Gauss-Manin connection on de Rham cohomology to cup product by the Kodaira-Spencer class of the family suitably conjugated by the inverse-Cartier operator. 

The following is a direct analogue of Katz' formula in quantum cohomology: 

\begin{thm} \label{thm:Cartierconjugation}  The following diagram commutes:
\begin{equation}\label{diagram:katz}
        \begin{tikzcd}
             F^*QH^*(X;\Lambda_{\mathbb{F}_p})[u^{\pm}]  \arrow[rr, "-u^{-1}F^*\mathbb{M}_b"] \dar["\mathcal{C}^{-1}"] &&  F^*QH^*(X;\Lambda_{\mathbb{F}_p})[u^{\pm}]  \dar["\mathcal{C}^{-1}"] \\
             QH^*(X;\Lambda_{\mathbb{F}_p})[u^{\pm}] \arrow[rr, "\psi_b"] && QH^*(X;\Lambda_{\mathbb{F}_p})[u^{\pm}]
        \end{tikzcd}.
    \end{equation}
 \end{thm}


\subsubsection{Quantum Steenrod and p-curvature} 
The (inverse) Cartier operator obtained from the Fontaine--Laffaille structure provides a way to ``twist'' the quantum product to get nontrivial linear endomorphisms of the quantum $D$-module, namely the $p$-curvature.  Initially proposed by Fukaya \cite{Fuk97} and studied in detail by Seidel--Wilkins \cite{Wil20, seidel-wilkins}, the quantum $D$-module $QH^*(X;\Lambda_{\mathbb{F}_p})[u]\langle \theta \rangle$ carries another set of distinguished endomorphisms, the \emph{quantum Steenrod operations}
\begin{equation}
    Q\Sigma_b: QH^*(X;\Lambda_{\mathbb{F}_p})[u]\langle \theta \rangle \to QH^*(X;\Lambda_{\mathbb{F}_p})[u]\langle \theta \rangle
\end{equation}
which are defined for any $b \in H^*(X; \Lambda)$ using $\mu_p$-equivariant Gromov--Witten type invariants. The operator $Q\Sigma_b$ is a quantum deformation of the total Steenrod operation, in the sense that $Q\Sigma_b |_{q=0} = St_p(b) \smile$ where $St_p(b)$ is the $p$-th total Steenrod power applied to $b$. A detailed review of the construction of $Q\Sigma_b$ is provided in Section \ref{subsec:qst}.

The final theorem of this paper identifies the $p$-curvature of quantum connection and the quantum Steenrod operation for Calabi--Yau threefolds, verifying a conjecture proposed by the second-named author in this setting.

\begin{thm}\label{thm:main}
    Let $(X,\omega)$ be a symplectic Calabi--Yau manifold with $\dim_{\mathbb{R}} X = 6$, and $b \in \mathrm{im}(H^2(X;\mathbb{Z}) \to H^2(X;\mathbb{F}_p))$ a degree $2$ cohomology class. Then
    \begin{equation}
        Q\Sigma_b = u^p \psi_b,
    \end{equation}
    that is the quantum Steenrod operation for $b$ and the $p$-curvature endomorphism for $b$ are equal.
\end{thm}

\begin{rem}
    In contrast with the assumption in Theorem \ref{thm:integrality}, for Theorem \ref{thm:main}, we do not have to restrict to the mod $p$ reduction of the torsion-free part of integral cohomology nor assume $p>3$. On the other hand, the assumption $b \in \mathrm{im}(H^2(X;\mathbb{Z}) \to H^2(X;\mathbb{F}_p))$ is crucial, as the classical limit of the $p$-curvature endomorphism $\psi_b$ does not see the Bockstein, which is part of the total Steenrod operation. See the recent work \cite{chen2026getzler} for a related discussion in the context of categorical enumerative invariants.
\end{rem}

The two theorems above can be summarized by the following conceptual diagram:

\begin{figure}[ht]
  \centering
  \begin{tikzcd}
    &F^*\mathbb{M}_b \arrow[ld, "\mbox{power operations}"'] \arrow[rd, "\mbox{torsion Fontaine--Laffaille}"] &\\ Q\Sigma_b \arrow[rr, equal] & & \psi_b
  \end{tikzcd}

\end{figure}
Namely, there are two ways of obtaining operations on the mod $p$ quantum $D$-module from the (Frobenius-twisted) quantum product $\mathbb{M}_b$ on usual quantum cohomology. One is to consider the Cartier isomorphism induced from the Fontaine--Laffaille structure on the $p$-adic quantum connection; the other is to consider the ``quantum power operations'' defined from counts of $\mu_p$-equivariant curves which arise as $p$-fold covers of the $3$-pointed spheres relevant for the quantum product.

\subsection{Proof method: Arithmetics from Gopakumar--Vafa integrality}\label{sect:Gopakumar}
In \cite{KSV06}, Kontsevich--Schwarz--Vologodsky argued that there is a fundamental connection between the integrality of B-model ``BPS" (or ``instanton") numbers and arithmetic structures on de Rham cohomology.  Our arguments give a concrete realization of their ideas on the A-side of mirror symmetry. Namely, our proof of the main theorems build on salient geometric features of curve counting in Calabi--Yau threefolds, both in terms of the integrality prediction of the Gopakumar--Vafa invariants of Calabi--Yau $3$-folds, and the proof method discovered in Ionel--Parker's remarkable work \cite{IP-GV}, the cluster formalism\footnote{We refer the reader to  \cite{doan2021gopakumar} and \cite{pardon-MNOP} for other important developments in this circle of ideas. }. 

Theorem \ref{thm:integrality} is deduced directly from the integrality of genus-zero Gopakumar--Vafa invariants. Namely, we write down an explicit formula for the Frobenius intertwiner $\Phi$ in terms of derivatives of a certain generating series $\eta(q)$, the ``$p$-adically integral" version of the Gromov-Witten potential. This explicit formula is inspired by certain B-model computations in \cite{KSV06}. The integrality of Gopakumar-Vafa invariants implies that the series  $\eta(q)$ lies in $\Lambda_{\mathbb{Z}_p}.$ This in turn implies that the entire Frobenius map $\Phi$ satisfies the key condition \eqref{eq:keyintegralitycondition}. Reducing mod $p$, we obtain the Cartier operator \eqref{eq:AmodelCartier} from the Frobenius map $\Phi$ by dividing suitable powers of $p$.

The proofs of Theorems \ref{thm:Cartierconjugation} and  \ref{thm:main} require more technical work. We upgrade the cluster decomposition and isotopy method from \cite{IP-GV} so that we can study Gromov--Witten correspondences to account for the quantum product, and quantum Steenrod operators, which involve equivariant counts of genus $0$ $J$-holomorphic curves from \emph{inhomogeneous perturbations}. The upshot is that we can bootstrap from the case of the \emph{elementary clusters} where the neighborhood of a curve is that of a \emph{local $\mathbb{P}^1$-geometry}, that is to the case where $X$ is the Calabi--Yau 3-fold given as the total space of the bundle $\mathcal{O}(-1)^{\oplus 2} \to \mathbb{P}^1$. In this setting, the computation of the quantum connection is a classical result of Voision \cite[Theorem 1.1]{Voi96}. The computation of the quantum Steenrod operations is more recent and due to the second-named author \cite[Section 3]{Lee23a}; see \cref{thm:localP1-qst-is-pcurv} for an exposition.

An interesting consequence of our proof is that we obtain an explicit formula for the $p$-curvature $\psi_b$, equivalently, the quantum Steenrod operation $Q\Sigma_b$ for $b \in \mathrm{im}(H^2(X;\mathbb{Z}) \to H^2(X;\mathbb{F}_p))$, in terms of genus $0$ Gromov--Witten invariants, which can further be described using genus $0$ Gopakumar--Vafa invariants and the Yukawa coupling, i.e., incidence coefficients of the quantum product. See \eqref{eqn:p-curv-formula} for the concrete formula. 

\subsection{Outlook}\label{sect:outlook}
The results in this article give a detailed picture of the arithmetic structure of the quantum connection on Calabi--Yau $3$-folds. Even though the proofs are built on specific features of the enumerative geometry of $3$-folds, most notably the Gopakumar--Vafa integrality and cluster decomposition of moduli spaces of $J$-holomorphic curves, we expect the results to serve as sample statements of general structural aspects of arithmetic properties of quantum connections. Closing the introduction, we formulate conjectures concerning potential generalizations.


\subsubsection{Higher dimensional generalization of the Fontaine--Laffaile structure}
Our first sequence of conjectures is exactly the generalization of our results on Fontaine--Laffaile modules in higher dimensions, which we formulate as follows. We borrow the notations from above.

\begin{conj}\label{conj:p-adic-gamma}
    Let $(X, \omega)$ be a closed symplectic Calabi--Yau manifold. Let $p$ be a prime for which $H^{even}(X;\mathbb{Z}_p)$ contains no torsion and $\mathcal{H}^{2k} \subset QH^*(X; \Lambda_{\mathbb{Z}_p})[u^{\pm}]$ the subspace of degree $2k$ elements.
    \begin{enumerate}
    \item (\textbf{$p$-adic $\Gamma$-conjecture})   There is a unique Frobenius map
    \begin{align} \Phi: F^*QH^*(X; \Lambda_{\mathbb{Z}_p})[u^{\pm}] \to QH^*(X; \Lambda_{\mathbb{Z}_p})[u^{\pm},1/p] \end{align} 
    with leading order term 
    \begin{equation}
        \Phi_0(u^kx) = p^{\dim(X)+k} u^k\cdot(\Gamma_p(TX) \smile x), \quad x \in H^*(X;\mathbb{Z}_p)
    \end{equation}
   such that: \begin{itemize} \item the map $\Phi$ satisfies:   \begin{align} \label{eq:keyintegralitycondition2} \Phi(F^*\mathcal{H}^{2k}) \subset p^k\mathcal{H}^{2k},\quad k \in \mathbb{Z}.\end{align} 
    \item for $p>\operatorname{dim}_{\mathbb{C}}(X)$, the quadruple: \begin{equation}(\mathcal{H}^0,\nabla,\operatorname{Fil}^i,\Phi)
    \end{equation}
    defines a Fontaine--Laffaile module. \end{itemize}
    \item (\textbf{Inverse Cartier isomorphism}) The mod $p$ reduction of the divided Frobenius defines a Cartier operator 
    \begin{align} 
    \mathcal{C}^{-1}: F^*QH^*(X;\Lambda_{\mathbb{F}_p})[u^{\pm}] \to QH^*(X;\Lambda_{\mathbb{F}_p})[u^{\pm}]
    \end{align}
    such that the diagram \eqref{diagram:katz} is commutative.
    \item (\textbf{Katz's formula with quantum power operation}) Given any $b \in \mathrm{im}(H^2(X;\mathbb{Z}) \to H^2(X;\mathbb{F}_p))$, denote by $Q\Sigma_b$ the quantum Steenrod operation along $b$. The the diagram
    \begin{equation}
        \begin{tikzcd}
             F^*QH^*(X;\Lambda_{\mathbb{F}_p})[u^{\pm}]  \arrow[rr, "-u^{-1}F^*\mathbb{M}_b"] \dar["\mathcal{C}^{-1}"] &&  F^*QH^*(X;\Lambda_{\mathbb{F}_p})[u^{\pm}]  \dar["\mathcal{C}^{-1}"] \\
             QH^*(X;\Lambda_{\mathbb{F}_p})[u^{\pm}] \arrow[rr, "u^{-p}Q\Sigma_b"] && QH^*(X;\Lambda_{\mathbb{F}_p})[u^{\pm}]
        \end{tikzcd}
    \end{equation}
    is commutative.
    \end{enumerate}
\end{conj}

As noted above, there is a closely related $p$-adic Gamma conjecture in the Fano/monotone setting in \cite{bai2025p}. The true content of our Conjecture \ref{conj:p-adic-gamma} (1) in the \emph{$p$-adic integrality} of the coefficients of the Frobenius because the target consists of series with ``bounded $p$-divisibility in the denominator." In contrast, in the monotone setting, the key assertion of the $p$-adic Gamma conjecture is the overconvergence property of the Frobenius intertwiner. As for Conjecture (2) and (3), although it is stated as being obtained from mod $p$ reduction of the divided Frobenius over $\mathbb{Z}_p$, it would be interesting to see if there is a direct geometric construction of $\mathcal{C}^{-1}$ using Floer-theoretic moduli spaces.

\subsubsection{Non-commutative geometry}
 Given an odd prime $p$, let $\mathcal{A}$ be a DG-category over $\mathbb{Z}_p$ whose morphism spaces are flat. Denote by $\mathcal{A}_0 := \mathcal{A} \otimes_{\mathbb{Z}_p} \mathbb{F}_p$ the mod $p$ reduction. By \cite[Theorem 1]{petrov2019periodic} or \cite[Theorem 0.5.1]{raksitsanath}, we have an equivalence of spectra
\begin{equation}\label{eqn:nc-crys-dR}
   \sigma: \mathrm{TP}(\mathcal{A}_0)_{\hat{p}} \xrightarrow{\sim} \mathrm{HP}(\mathcal{A})_{\hat{p}}, 
\end{equation}
where we take the $p$-completion of topological periodic cyclic homology and periodic cyclic homology respectively. This is the non-commutative analogue of the comparison between crystalline cohomology and de Rham cohomology for a variety over $\mathbb{F}_p$ which is liftable to $\mathbb{Z}_p$. By Nikolaus--Scholze \cite{nikolaus-scholze}, if $\mathcal{A}_0$ is smooth and proper over $\mathbb{F}_p$, $\mathrm{TP}(\mathcal{A}_0)_{\hat{p}}[1/p]$ comes with a Frobenius endomorphism, which in turn gives rise to a Frobenius $\varphi_p$ on $\mathrm{HP}(\mathcal{A})_{\hat{p}}[1/p]$. When $\mathcal{A}=D^b\operatorname{Coh}(\mathcal{Y})$ for a smooth variety over $\mathbb{Z}_p$, this recovers, up to Tate twists, the crystalline Frobenius \cite[\S 6]{tabuadaweil}. 

If we consider DG categories defined over suitable $p$-adically integral rings $R$ (e.g. $R=\mathbb{Z}_p((q))$), the method of \cite{raksitsanath} extends to give a version $\sigma_R$ of \eqref{eqn:nc-crys-dR} relating a relative version of $\mathrm{TP}$ with relative periodic cyclic homology, $\mathrm{HP}(\mathcal{A}/R)_{\hat{p}}.$ We expect that induced cyclotomic Frobenius on $\mathrm{HP}(\mathcal{A}/R)_{\hat{p}}[1/p]$ is horizontal with respect to the Getzler-Gauss-Manin connection $\nabla^{\mathrm{GGM}}$ which is defined on $\mathrm{HP}(\mathcal{A}/R)$. Finally, the Getzler--Gauss--Manin connection $\nabla^{\mathrm{GGM}}$ satisfies Griffiths transversality with respect to the non-commutative analogue of the Hodge filtration, $\operatorname{Fil}^{\bullet}$. The upshot is that non-commutative geometry gives rise to a quadruple:
\begin{equation}
    (\mathrm{HP}(\mathcal{A}/R)_{\hat{p}}[1/p], \nabla^{\mathrm{GGM}}, \operatorname{Fil}^{\bullet}, \varphi_p).
\end{equation}

For suitable Calabi-Yau manifolds $X$,  Perutz-Sheridan \cite{rel-Fuk} have defined constructed a version of the Fukaya category $\operatorname{Fuk}(X)$ which is defined over $p$-adically integral Novikov rings $R$. The paper \cite{ganatra2025cyclic} defines an ``open-closed" map in this setting:  \begin{align}\label{eq:OCmap} \mathcal{OC}_{S^1}: \mathrm{HP}_*(\operatorname{Fuk}(X)/R) \to QH^*(X;R) \end{align}
which intertwines the Getzler-Gauss-Manin connection with the quantum connection: \begin{align} \mathcal{OC}_{S^1} \circ \nabla^{\mathrm{GGM}}= \nabla^{\mathrm{QH}} \circ \mathcal{OC}_{S^1}. \end{align}
\begin{conj}\label{conj:open-closed}
    After taking $p$-completion and inverting $p$, the cyclic open-closed map \eqref{eq:OCmap} intertwines the Frobenius maps : \begin{align} \mathcal{OC}_{S^1} \circ p^{\lfloor \operatorname{dim}_{\mathbb{C}}(X)/2 \rfloor} \varphi_p= \Phi \circ \mathcal{OC}_{S^1}. \end{align}
\end{conj}

In view of our $p$-adic Gamma conjecture in the Calabi--Yau setting, Conjecture \ref{conj:open-closed} asserts that the cyclotomic Frobenius $\varphi_p$ ``sees" the $p$-adic Gamma class $\Gamma_p(TX)$. This is consistent with the Candelas--de la Ossa--van Straten conjecture discussed above and mirror symmetry expectations, but seeing the geometric origin of such a conjectural statement is a very interesting problem.

This non-commutative discussion also applies to the families of $A_\infty$-categories arising from the wrapped Fukaya category of a smooth anti-canonical divisor complement in a monotone symplectic manifold $M$ (\cite{PS23}). In \emph{loc.cit.}, the authors combine the $p$-torsion Fontaine-Laffaille theory from \cite{PVV} with geometric considerations to study the Fourier transform of the quantum connection on $M$. The above discussion should apply to this situation to upgrade these constructions to the $p$-adic setting, which should have interesting consequences for studying quantum connections on monotone symplectic manifolds using the arithmetic Fourier--Laplace transform (cf. \cite{kedlaya-fourier}).

\subsubsection{Quantum power operation and p-curvature}
Theorem \ref{thm:main} proves a conjecture due to the second-named author \cite{Lee23b} in a broad class of examples. Previous work for other (non-compact) Calabi-Yau manifolds include \cite{Lee23a, bai-lee}. Away from the Calabi-Yau cases, Chen established a weaker version of \cref{thm:main} for monotone symplectic manifolds as a key step in his work on the exponential type conjecture for quantum connections of Fano varieties \cite{chen2024exponential}. Unpublished work of Pomerleano--Seidel elaborates on this argument to prove the analogue of \cref{thm:main} for monotone symplectic manifolds. The arguments in the monotone case are completely different and rely on algebraic considerations which are specific to the quantum connection on monotone symplectic manifolds. In a broader context, the work \cite{bai2025quantum} identifies the quantum Adams operation, i.e., the quantum power operation in the $K$-theory setting, with the analogue of the $p$-curvature of the $q$-difference connections defined using quasimap counts.

Assembling all of these examples together, we expect that variants of $p$-curvature of (generalizations of) the quantum connection should be equal to the quantum power operation whenever they can be defined (see \cite{Bai_Xu_2026} for how to define quantum product and quantum Steenrod operations on general symplectic manifold with mod $p$ coefficients). It is thus tempting to speculate that there should a moduli-theoretic approach to constructing an analog of the Cartier operator $\mathcal{C}^{-1}$ which conjugates the quantum Steenrod operation $Q\Sigma_b$ to the Frobenius-twisted quantum product $F^*\mathbb{M}_b$, and such an operator should also be able to relate the $p$-curvature $\psi_b$ with $F^*\mathbb{M}_b$ as in the setting of $p$-torsion Fontaine--Laffaile modules (cf. \cite{katz-nilpotent} and \cite{PVV}) even though we have to work with $\mathbb{Z}/2$-grading in general. It is likely that this should come from a piece of the cyclotomic structure on Hamiltonian Floer theory (\cite{rezchikov2024cyclotomic}).

\begin{rem}
    S. Rezchikov has proposed an analogous picture in noncommutative geometry. It should be possible to relate this picture to our discussion via a mod $p$ version of Conjecture \ref{conj:open-closed}.
\end{rem}

\subsection{Organization of the paper}
In Section 2, we recall the constructions of quantum cohomology and quantum connection, with an emphasis on the special properties of Gromov--Witten theory for Calabi--Yau threefolds. In Section 3, we recall the general notion of logarithmic Fontaine--Laffaille modules, and equip the $p$-adic quantum connection with that structure. In Section 4, we begin discussing the mod $p$ quantum connection, in particular we introduce the inverse Cartier operator obtained from the $p$-adic Frobenius intertwiner and its relationship to the quantum operations intrinsic to the mod $p$ quantum cohomology. In Section 5, we prove the equivalence between quantum Steenrod operations and the $p$-curvature via the cluster formalism, and we also prove the $A$-model analogue of Katz's formula for the $p$-curvature.

\subsection*{Acknowledgments} We would like to thank Paul Seidel for extensive collaboration (\cite{bai2025p, PS23}) and discussions on many topics closely related to the present paper. Concretely, he suggested the idea of using the cluster formalism to extend the computations from \cite{Lee23a} to more general examples in 2023. This proved to be a key inspiration for the present paper. We thank Eleny Ionel and John Pardon for very helpful discussions regarding the cluster formalism, Sanath Devalapurkar for discussions on topological cyclic homology, and Alexander Petrov and Daniel (Dongryul) Kim for discussions on $p$-adic Hodge theory.

\section{Gromov--Witten theory of Calabi--Yau 3-folds}

In this section, we recall the constructions of quantum cohomology and the quantum connection. We define the operations using pseudocycles to represent incidence constraints. The last subsection is focused on the special properties of genus $0$ Gromov--Witten theory for Calabi--Yau threefolds.

\subsection{Pseudocycles}\label{ssec:pseudocycles}
The purpose of this section is to recall the notion of (integral) pseudocycles, which geometrically represent the Poincar\'e dual of classes in $H^*(X;\mathbb{Z})$ for a smooth oriented manifold $X$. The intersection theory of pseudocycles will be used below to define the Gromov--Witten invariants.
\begin{defn}
    A smooth map $f: W \to X$ from a $k$-dimensional oriented manifold $W$ is called an (oriented $k$-dimensional) \emph{pseudocycle} if it satisfies the following properties.
    \begin{enumerate}
        \item $f(W)$ is precompact in $X$.
        \item The dimension of the $\Omega$-set of $f$,
        \begin{equation}
            \Omega_f := \bigcap_{K \subseteq W \mathrm{compact}} \ov{f(W\setminus K)},
        \end{equation}
        is at most $k-2$.
    \end{enumerate}
\end{defn}
A more detailed introduction to the notion of an (integral) pseudocycle can be found in \cite[Section 6.5]{MS12}, \cite{Zin08}.

\begin{defn}
    Two (oriented $k$-dimensional) pseudocycles $f_0: W_0 \to X$ and $f_1 : W_1 \to X$ are \emph{cobordant} if there exists a smooth map $\hat{f}: \hat{W} \to X$ from a $(k+1)$-dimensional oriented manifold with boundary $\hat{W}$ such that
    \begin{enumerate}
        \item $\hat{f}(\hat{W})$ is precompact in $X$.
        \item $\dim \Omega_{\hat{f}} \le k-1$.
        \item There is an oriented diffeomorphism $\partial \hat{W} \cong - W_0 \sqcup W_1$ such that $\hat{f}|_{W_i} = f_i |_{W_i}$.
    \end{enumerate}
\end{defn}
The set of bordism classes of $p$-pseudocycles are denoted by $\mathfrak{H}_k(X)$. Using disjoint union as addition and orientation reversal as inverses, $\mathfrak{H}_*(X) = \bigoplus_k \mathfrak{H}_k(X)$ forms a graded abelian group, where the grading comes from dimension.

If we denote $\mathfrak{H}_*(X)$ the abelian group of (integral) pseudocycle bordisms, Zinger \cite{Zin08} constructs a natural isomorphism
\begin{equation}
    \Psi_* : H_*(X;\mathbb{Z}) \cong \mathfrak{H}_*(X).
\end{equation}
 
\begin{prop}
    There is a well-defined intersection pairing
    \begin{equation}
        \cap : \mathfrak{H}_*(X) \otimes \mathfrak{H}_*(X) \to \mathbb{Z}.
    \end{equation}
\end{prop}

Note that this pairing factors through the quotient of $\mathfrak{H}_*(X)$ by its torsion subgroup.

Fix a ground coefficient commutative ring $\mathbb{K}$. Consider the quotient of cohomology by the torsion subgroup, so that  
\begin{equation}    
\bar{H}^*(X;\mathbb{Z}) \cong \bigoplus_k H^k(X;\mathbb{Z})/H^k(X;\mathbb{Z})_{tor}
\end{equation}
is the torsion-free part of $H^*(X;\mathbb{Z})$. Similarly, let
$$\bar{H}_*(X;\mathbb{Z}):=H_*(X;\mathbb{Z})/H_*(X;\mathbb{Z})_{tor}$$ 
denote the torsion-free part of $H_*(X;\mathbb{Z})$. Then Poincar\'e duality induces a map
\begin{equation}
    \mathrm{PD}: \bar{H}^*(X;\mathbb{Z}) \otimes \mathbb{K} \to \bar{H}_*(X;\mathbb{Z}) \otimes \mathbb{K}.
\end{equation}

For convenience, we introduce the following notation:

\begin{defn}
    Let $X$ be a closed oriented manifold. Let $a \in \bar{H}^*(X;\mathbb{Z})\otimes \mathbb{K}$, and denote its Poincar\'e dual homology class by $\mathrm{PD}[a] \in \bar{H}_*(X;\mathbb{Z}) \otimes \mathbb{K}$. Then
    \begin{equation}
        \Psi(a) := \Psi_*(\mathrm{PD}[a]) \in \mathfrak{H}_*(X)/\mathfrak{H}_*(X)_{tor} \otimes \mathbb{K}
    \end{equation}
    is the bordism class of a pseudocycle representing the Poincar\'e dual of the cohomology class $a \in \bar{H}^*(X;\mathbb{Z})\otimes \mathbb{K}$.
\end{defn}
By a slight abuse of notation, we may denote by $\Psi(a)$ some pseudocycle representing the bordism class $\Psi(a)$.

\subsection{Quantum cohomology}\label{ssec:qh}

In this section, we recall our conventions for the quantum products. We fix $(X, \omega)$ to be a closed Calabi--Yau symplectic manifold, i.e., $c_1(TX) = 0$, and we additionally require that $\pi_1(X) = 1$. We further assume that the symplectic class $[\omega] \in H^2(X;\mathbb{Z})$ is integral, which can always be achieved by perturbing and rescaling. 

Let us begin by describing the Novikov ring over a fixed commutative ring $\mathbb{K}$. We assume that $$\bar{H}^*(X;\mathbb{Z}) \otimes \mathbb{K} \cong H^*(X;\mathbb{K});$$ if $\mathbb{K} = \mathbb{F}_p$, this amounts to the assumption that $H^*(X;\mathbb{Z})$ has no $p$-torsion.

\begin{defn}
The \emph{Novikov ring} $\Lambda:= \Lambda_{\mathbb{K}}$ over $\mathbb{K}$ is the ring with elements given by formal power series:
\begin{equation}
\sum_{\beta \in \bar{H}_2(X; \mathbb{Z}), \int_\beta \omega \geq 0} c_\beta q^\beta, \quad c_\beta \in \mathbb{K}
\end{equation}
such that for any fixed $E \geq 0$, there are only finitely many $c_\beta \neq 0$ for all $\beta \in \bar{H}_2(X; \mathbb{Z})$ satisfying $\int_\beta \omega \leq E$. We denote by  $\Lambda_{+} \subset \Lambda_{\mathbb{K}}$ denote the ideal of elements with strictly positive $\omega$-valuation. 
\end{defn}

 Let $\bar{H}_2(X,\mathbb{Z})_{\geq 0} \subset \bar{H}_2(X,\mathbb{Z})$ denote the monoid of classes $\beta$ with $\omega([\beta]) \geq 0.$ The Novikov ring carries a natural collection of logarithmic derivations indexed by $b \in H^2(X;\mathbb{K})$. Namely, any $b \in H^2(X;\mathbb{K}) = \bar{H}^2(X;\mathbb{K})$ induces a logarithmic derivation of $\Lambda$:
\begin{align}
\partial_b : \Lambda_\mathbb{K} \to \Lambda_\mathbb{K}, \quad
\sum_\beta c_\beta q^\beta \to \sum_\beta c_\beta (b \cdot \beta) q^\beta \nonumber
\end{align}
where $(b \cdot \beta)$ denotes the natural pairing between $H^2(X)$ and $\bar{H}_2(X)$.

\begin{defn}
    The \emph{quantum cohomology} is additively the $\Lambda_{\mathbb{K}}$-module
    \begin{equation}
        QH^*(X;\Lambda_{\mathbb{K}}) := H^{even}(X;\Lambda_{\mathbb{K}}) = \bigoplus_{k \in 2 \mathbb{Z}} H^k (X;\Lambda_{\mathbb{K}}).
    \end{equation}
\end{defn}

The cohomology module with Novikov coefficients $H^{even}(X;\Lambda_{\mathbb{K}})$ admits the (small) quantum product $\star_q$ using the three-pointed Gromov--Witten invariants and the Poincar\'e pairing $(\cdot, \cdot) : H^*(X;\mathbb{K}) \otimes H^*(X;\mathbb{K}) \to \mathbb{K}$. One can extend the Poincar\'e pairing to $H^*(X;\Lambda)$ by first extending $q^\beta$-linearly for a finite truncation $\Lambda_{\le E} \subseteq \Lambda$ given a bound $\int_\beta \omega \le E$ as a map $(\cdot, \cdot) : H^*(X;\Lambda_{\le E})^{\otimes 2} \to \Lambda_{\le 2E}$ and taking the limit.

The three-pointed Gromov--Witten invariants and the quantum product is defined as follows. Let $J$ be an $\omega$-compatible almost complex structure. Fix the curve $ \mathbb{P}^1$ with its fundamental class $[\mathbb{P}^1] \in H_2(\mathbb{P}^1;\mathbb{Z})$, and consider the three-pointed genus $0$ Gromov--Witten moduli space of maps satisfying the $J$-holomorphic curve equation (with source a parametrized copy of $\mathbb{P}^1$)
\begin{equation}
    \mathcal{M}_\beta = \mathcal{M}_{0,3}(X, J,  \beta)  := \{ u : \mathbb{P}^1 \to X: \overline{\partial}_J u = 0,  u_*[\mathbb{P}^1] = \beta \}.
\end{equation}
For a $C^\infty$-generic choice of $J$, it can be arranged that all \emph{simple} maps (those that do not factor through a nontrivial branched covering $\mathbb{P}^1 \to \mathbb{P}^1$) are transversely cut out. Hence, the moduli space of simple maps is a manifold of dimension $\dim_{\mathbb{R}} \mathcal{M}_\beta = \dim_{\mathbb{R}} X + 2c_1(\beta) = \dim_{\mathbb{R}} X$, which is equipped with a natural orientation as the linearized $\ov{\partial}_J$-operator can be deformed to be $\mathbb{C}$-linear.

The moduli space admits the stable map compactification $\ov{\mathcal{M}}_\beta$, and since the source curve of the maps $u : \mathbb{P}^1 \to X$ is parametrized, there are evaluation maps
\begin{align}
    \mathrm{ev} = \mathrm{ev}_0 \times \mathrm{ev}_\infty \times \mathrm{ev}_1 : \ov{\mathcal{M}}_\beta &\to X \times X \times X \\
    u &\mapsto (u(0), u(\infty), u(1)).
\end{align}
It is standard that, under our Calabi--Yau assumption, the evaluation map $\mathrm{ev}$ restricted to the simple maps in $\mathcal{M}_\beta$ form an integral pseudocycle in $X \times X \times X$, whose bordism class is denoted by
\begin{equation}
    \mathrm{ev}_* [\ov{\mathcal{M}}_\beta] \in \mathcal{H}_*(X \times X \times X) \cong H_*(X \times X \times X ; \mathbb{Z}).
\end{equation}
The bordism class is independent of the choice of $J$ by a standard cobordism argument.

\begin{defn}
    Let $a_1, a_2, b \in H^*(X;\mathbb{K})$. The \emph{quantum product correlator} is the $3$-pointed Gromov-Witten invariant
    \begin{equation}
        \langle a_1, a_2, b \rangle_\beta := \mathrm{ev}_*[\overline{\mathcal{M}}_\beta] \cap \Psi(a_1 \otimes a_2 \otimes b) \in \mathbb{K}
    \end{equation}
    where $\cap$ denotes the intersection pairing of pseudocycles in $X \times X \times X$.
\end{defn}

One can easily extend the correlator $q^\beta$-multi-linearly to insert classes from $H^*(X;\Lambda_{\mathbb{K}})$. We define the quantum product by using the quantum product correlators as structure constants:

\begin{defn}
    On $H^*(X;\Lambda)$, the (small) quantum product $$\star := \star_q : H^*(X;\Lambda_{\mathbb{K}}) \otimes H^*(X;\Lambda_{\mathbb{K}})\to H^*(X;\Lambda_{\mathbb{K}})$$ is the bilinear operation defined by
    \begin{equation}
        \left( b \star_q a_1, a_2 \right) = \sum_{\beta} \langle a_1, a_2, b \rangle_\beta \  q^\beta \in \Lambda_{\mathbb{K}},
    \end{equation}
    where $(\cdot, \cdot)$ denotes the Poincar\'e pairing. It is well-defined (in the sense that the right hand side is genuinely an element of $\Lambda$) by Gromov compactness. The operation $\star$ defines a graded-commutative and associative product on $H^*(X;\Lambda)$ \cite[Chapter 9]{MS12}, giving rise to the ring structure on quantum cohomology.
\end{defn}

It is useful to consider the quantum cohomology not just as a ring, but as a collection of $\Lambda$-linear operators indexed by cohomology classes $b \in H^*(X;\mathbb{K})$:
\begin{equation}
    \mathbb{M}_b := b \  \star_q \in \mathrm{End}(H^*(X;\Lambda_{\mathbb{K}})).
\end{equation}
For later use, we describe the following alternative characterization of this operator via correspondences. Let $b \in H^*(X;\mathbb{K})$ be a cohomology class whose Poincar\'e dual is represented by a pseudocycle $\Psi(b): B \to X$. Take the fiber product 
    \begin{equation}
        \begin{tikzcd}
{\overline{\mathcal{M}}_{0,3}(X, J, \beta)_b} \arrow[d] \arrow[r]  & {\overline{\mathcal{M}}_{0,3}(X, J, \beta)} \arrow[d, "\mathrm{ev}_1"] \\
B \arrow[r, "\Psi(b)"']                                       & X            
\end{tikzcd}
    \end{equation}

and consider the evaluation map
\begin{equation}
      \mathrm{ev}_0 \times \mathrm{ev}_\infty: {\overline{\mathcal{M}}_{0,3}(X, J, \beta)_b}  \to X \times X
\end{equation}
whose restriction to the simple locus defines an integral pseudocycle if the pseudocycle representative for $b$ is chosen generically, hence a homology class in $H_*(X \times X ; \mathbb{K}) \cong \mathcal{H}_*(X\times X) \otimes \mathbb{K}$. Denote this homology class by $\mathrm{ev}_*[\ov{\mathcal{M}}_{A,b}]$.

\begin{lemma}\label[lemma]{lemma:quantum-product-correspondence}
    The operator $\mathbb{M}_b$ is induced by the correspondences $\mathrm{ev}_0 \times \mathrm{ev}_\infty: {\overline{\mathcal{M}}_{0,3}(X, J, \beta)_b}  \to X \times X$.
\end{lemma}
\begin{proof}
    Denote the two projections from $X\times X$ by $\pi_1, \pi_2 : X\times X \to X$, and similarly let $\pi_{12} : X\times X \times X \to X\times X$ to be the projection to first two factors. It suffices to show that
    \begin{equation}
        (\pi_2)_* \left( \pi_1^* a_1 \cap \mathrm{ev}_*[\ov{\mathcal{M}}_{\beta,b}] \right) \cap a_2 = \langle a_1, a_2, b \rangle_\beta.
    \end{equation}
    The desired result follows from 
    \begin{align}
        (\pi_2)_* \left( \pi_1^* a_1 \cap \mathrm{ev}_*[\ov{\mathcal{M}}_{\beta,b}] \right) \cap a_2 &= (\pi_2)_* \left( \pi_1^* a_1 \cap \pi_2^* a_2 \cap \mathrm{ev}_*[\ov{\mathcal{M}}_{\beta,b}] \right) \\
        &= (\pi_2)_* \left( (a_1 \otimes a_2) \cap \mathrm{ev}_*[\ov{\mathcal{M}}_{\beta,b}] \right) \\
        &= (\pi_2)_* \left((\pi_{12})_* (a_1 \otimes a_2 \otimes b \cap \mathrm{ev}_*[\ov{\mathcal{M}_\beta}])  \right) \\
        &= (\pi_{2} \circ \pi_{12})_* \langle a_1, a_2, b \rangle_\beta \\ &= 
        \langle a_1, a_2, b \rangle_\beta  .
    \end{align}
\end{proof}


\subsection{Quantum connections}

The \emph{quantum connection} or Dubrovin--Givental connection \cite{Dub98, Giv95, seidel-connections} is the $\mathbb{G}_m$-equivariant lift of the operators $\mathbb{M}_b$, where $\mathbb{G}_m = \mathbb{G}_m(\mathbb{C})$ denotes the multiplicative group of nonzero complex numbers.

\begin{lemma}\label{lem:group-cohomology-G_m}
    The group cohomology of $\mathbb{G}_m$ with $\mathbb{Z}$-coefficients
    \begin{equation}
        H^*(B\mathbb{G}_m ; \mathbb{Z}) \cong \mathbb{Z} [u]
    \end{equation}
    is a graded polynomial algebra on the generator of degree $|u|=2$.
\end{lemma}

\begin{defn}
    The \emph{loop-equivariant quantum cohomology} is the graded $\Lambda$-module defined as
    \begin{equation}
        QH^*(X;\Lambda_{\mathbb{K}})[u] := H^{even}(X;\Lambda_{\mathbb{K}})[u]
    \end{equation}
    where $|u| = 2$. We also consider the version where $u$ is inverted, i.e.
\begin{align} 
QH^*(X;\Lambda_{\mathbb{K}})[u^\pm]:=H^{even}(X; \Lambda_\mathbb{K})[u,u^{-1}]. 
\end{align}
\end{defn}
Here, we remark that there is no $\mathbb{G}_m$-action on the target symplectic manifold $X$; the action of $\mathbb{G}_m$ should be considered as a \emph{loop rotation} on the source curve $C$ for $J$-holomorphic maps $u: C \to X$. 

Note that because $u$ has degree 2, the loop-equivariant quantum cohomology is ``2-periodic" once $u$ is inverted. Namely, the degree zero component of $QH^*(X;\Lambda_{\mathbb{K}})[u^\pm]$ has the decomposition
\begin{align} 
\bigoplus_{k=0}^{\frac{1}{2}\dim_{\mathbb{R}}X} u^{-k}H^{2k}(X;\Lambda_{{\mathbb{K}}}),
\end{align}  
which becomes isomorphic to $QH^*(X;\Lambda_{\mathbb{K}})[u^\pm]$ after taking tensor product with $\mathbb{Z}[u, u^{-1}]$.

There is no natural ring structure on $QH^*(X;\Lambda_{\mathbb{K}})[u^\pm]$. However, it still carries an action of certain connection-type operators. Recall that any $b \in H^2(X;\mathbb{K})$ induces a logarithmic derivation of $\Lambda$:
\begin{align}
\partial_b : \Lambda_\mathbb{K} \to \Lambda_\mathbb{K}, \quad
\sum_\beta c_\beta q^\beta \mapsto\sum_\beta c_\beta (b \cdot \beta) q^\beta \nonumber,
\end{align}
which can be extended $u$-linearly to $\Lambda[u]$.

\begin{defn}
Let $b \in H^2(X;\mathbb{K})$. The \emph{(small) quantum connection} is the $\mathbb{K}[u^\pm]$-linear endomorphism on loop-equivariant quantum cohomology, indexed by $ b \in H^2(X;\mathbb{K})$ defined by:  
\begin{equation}
\nabla_b := \partial_b + u^{-1} \mathbb{M}_b : QH^*(X;\Lambda_{\mathbb{K}})[u^\pm] \to QH^*(X;\Lambda_{\mathbb{K}})[u^\pm],
\end{equation}
where $\mathbb{M}_b$ is the $u$-linear extension of the quantum product by $b$ considered as a $\Lambda_{\mathbb{K}}$-linear endomorphism of $QH^*(X;\Lambda_{\mathbb{K}})$.
\end{defn}

The quantum connection is reminiscent of \emph{$\lambda$-connections} of Deligne from nonabelian Hodge theory (see e.g. \cite{simpson-hodge-nonabelian}), where the role of $\lambda$ is played in this setting by the equivariant parameter $u$. In particular, the quantum connection lifts the quantum product to loop-equivariant quantum cohomology in the sense that the following diagram commutes:
\begin{center}
    \begin{tikzcd}
        QH^*(X;\Lambda_{\mathbb{K}})[u] \rar["u\nabla_b"] \dar & QH^*(X;\Lambda_{\mathbb{K}})[u] \dar\\ QH^*(X;\Lambda_{\mathbb{K}}) \rar["\mathbb{M}_b"] & QH^*(X;\Lambda_{\mathbb{K}})
    \end{tikzcd}
\end{center}
where the vertical maps specialize $u=0$, forgetting the $\mathbb{G}_m$-equivariance (note that $u\nabla_b$ can be defined without poles in $u$).

\subsection{Quantum products of CY3}\label{ssec:qhCY3}
In this section, we describe the special properties of the quantum cohomology of Calabi--Yau threefolds, i.e., the special case of $\dim_{\mathbb{R}} X = 6$. We retain our assumptions that $c_1(X) = 0$, $\pi_1(X) = 1$, and $[\omega] \in H^2(X;\mathbb{Z})$. 

\subsubsection{Preferred basis of cohomology}
Our notation for the (torsion-free part of) even cohomology
\begin{equation}
    H^{even}(X) :=  \bar{H}^0(X;\mathbb{Z}) \oplus \bar{H}^2(X;\mathbb{Z}) \oplus \bar{H}^4(X;\mathbb{Z}) \oplus \bar{H}^6 (X;\mathbb{Z}),
\end{equation}
where $\bar{H}^k(X;\mathbb{Z}) = H^k(X;\mathbb{Z})/H^k(X;\mathbb{Z})_{tor}$ is the quotient by the torsion subgroup, is as follows. We denote the Poincar\'e dual of the fundamental class by $ 1 \in H^0(X;\mathbb{Z})$, and the Poincar\'e dual of a point by $[\mathrm{pt}] \in H^6(X;\mathbb{Z})$. Since $\pi_1(X) = 1$, by universal coefficient theorem, $H^2(X;\mathbb{Z})$ is torsion-free. We choose an integral basis 
\begin{equation}
    e_1, \dots, e_r \in H^2(X;\mathbb{Z}),
\end{equation} and assume that $e_r$ is a positive primitive multiple of the class of the symplectic form $[\omega]$. The corresponding linear dual basis is denoted by
\begin{equation}
    \beta_1, \dots, \beta_r \in \bar{H}_2(X;\mathbb{Z}) = H_2(X;\mathbb{Z})/H_2(X;\mathbb{Z})_{tor},
\end{equation}
and their Poincar\'e duals are denoted by
\begin{equation}
    L_1, \dots, L_r \in \bar{H}^4(X;\mathbb{Z}).
\end{equation}

\subsubsection{Quantum connection of CY3}
Using the dual basis $\beta_1, \dots, \beta_r \in \bar{H}_2(X;\mathbb{Z})$, we may identify the terms $q^\beta \in \Lambda$ with monomials in the variables $q_1 := q^{\beta_1}, \dots, q_r = q^{\beta_r}$, so that in these coordinates we may write
\begin{equation}
    \Lambda_{\mathbb{K}} = \mathbb{K} [q_1^\pm, \dots, q_{r-1}^\pm] [\![q_r]\!],
\end{equation}
and the log derivations corresponding to $b = e_k$ ($1 \le k \le r$) can be alternatively written as
\begin{equation}
    \partial_{e_k} = q_k \partial_{q_k}
\end{equation}
in these coordinates, so that $\nabla_{e_k} = q_k \partial_{q_k} + u^{-1} \mathbb{M}_{e_k}$. To explicitly determine the connection, it suffices to compute the quantum product operators $\mathbb{M}_b$ for each $b  = e_k \in H^2(X;\mathbb{Z})$. In our notation, the relevant structure constants (the quantum product correlators) are given as follows:

\begin{defn}\label{defn:yukawa}
    Let $1 \le i, j, k \le r$. The \emph{Yukawa coupling} is the element of the Novikov ring
    \begin{equation}
        Y_{ijk}(q) := \sum_{\beta, \omega(\beta) \ge 0} \langle e_i, e_j, e_k \rangle_{\beta}  \ q^\beta \in \Lambda_{\mathbb{K}}.
    \end{equation}
\end{defn}

\begin{lemma}\label{lem:qproduct-computation}
In the preferred homogeneous basis given by $1 \in H^0(X), e_i \in H^2(X), L_i \in H^4(X), [\mathrm{pt}] \in H^6(X)$, the quantum product $\mathbb{M}_{e_k}$ for $b = e_k \in H^2(X;\mathbb{Z})$, $1 \le k \le r$, is given by
\begin{align} \label{eq:matricesquantumproduct}
\mathbb{M}_{e_k} (1) &= e_k, \\
\mathbb{M}_{e_k} (e_i) &=  \sum_{j} Y_{ijk}(q) L_j ,\nonumber \\
\mathbb{M}_{e_k}(L_i) &= \delta_{ik} [\mathrm{pt}], \\
\mathbb{M}_{e_k} ([\mathrm{pt}])&= 0 \nonumber  .
\end{align}
\end{lemma}
\begin{proof}
Based on Definition \ref{defn:yukawa}, the computation follows from the fact that $\mathbb{M}_{e_k}$ is an operation of cohomological degree $2$, and the Poincar\'e pairing $(e_k, L_i) = \delta_{ik}$. 
\end{proof}

\subsubsection{Gopakumar--Vafa integrality}
The Yukawa coupling may also be described in terms of generating functions of $0$-pointed Gromov--Witten invariants,
\begin{equation}
    \langle \rangle_\beta := [\mathcal{M}_{\beta, 0, 0}] \in \mathbb{Q}.
\end{equation}
Note that $0$-pointed genus $0$ Gromov--Witten invariants are defined using moduli spaces within the stable range (meaning at least three marked points in the genus $0$ setting), hence are not defined over the integers. Therefore, for the following definition, we restrict to $\mathbb{K} = \mathbb{Q}$.

\begin{defn}\label{defn:gw-potential}
    The \emph{genus 0 Gromov--Witten potential} is the element of the rational Novikov ring
    \begin{equation}
        G(q) := \sum_{\beta \neq 0, \omega(\beta) \ge 0} \langle \rangle_{\beta} \ q^\beta \in \Lambda_{\mathbb{Q}}.
    \end{equation}
\end{defn}

By the divisor axiom of Gromov--Witten theory, we have the following.
\begin{lemma}\label{lem:yukawa-as-third-derivative}
    Let $\delta_k := \partial_{e_k} = q_k \partial_{q_k}$ be the Novikov derivation in the direction of $e_k \in H^2(X;\mathbb{Z})$. Then
    \begin{equation}
        Y_{ijk}(q) = Y_{ijk}(0) + \delta_i \delta_j \delta_k  G(q) \in \Lambda_{\mathbb{K}}.
    \end{equation}
\end{lemma}
Note that \emph{a priori}, the coefficients of $\delta_i \delta_j \delta_k  G(q)$ are rational numbers. However, the divisor axiom equates the coeffcients with certain genus $0$ Gromov--Witten invariants defined using three-pointed curves, which are integral using the intersection pairing of pseudocycles. Therefore, $\delta_i \delta_j \delta_k  G(q) \in \Lambda_{\mathbb{K}}$ via the homomorphism $\Lambda_{\mathbb{Z}} \to \Lambda_{\mathbb{K}}$ induced by the inclusion of the identity $\mathbb{Z} \to \mathbb{K}$.

The following is a deep theorem of Ionel--Parker \cite{IP-GV} (restricted to the case of genus $0$ counts), implying strong structural properties of the Gromov--Witten invariants $\langle \rangle_\beta$. For a class $\beta$ with $\omega(\beta) \ge 0$, consider the following logarithmic integral series,
\begin{equation}
    \mathrm{Li}_3(q^\beta) = \sum_{d \ge 1} \frac{q^{d \beta}}{d^3} \in \mathbb{Q}[\![q^\beta]\!],
\end{equation}
an element of the rational power series ring.

\begin{thm}[{Gopakumar--Vafa integrality, \cite{IP-GV}}]
    \label{thm:gv-integrality}There exist integers $n_\beta \in \mathbb{Z}$ such that
    \begin{equation}
        G(q) = \sum_{\beta \neq 0, \omega(\beta) \ge 0 } n_\beta \ \mathrm{Li}_3 (q^\beta) \in \Lambda_{\mathbb{Q}}.
    \end{equation}
\end{thm}
The integrality of $n_\beta$ is consistent with the observation that the Yukawa couplings are defined integrally; indeed, the denominators in $\mathrm{Li}_3(q^\beta)$ are cancelled upon taking logarithmic derivatives three times:
\begin{equation}
    Y_{ijk}(q) - Y_{ijk}(0) = \delta_i \delta_j \delta_k G(q) = \sum_{\beta \neq 0 } n_\beta \sum_{ d \ge 1} (d \cdot \beta(e_i) ) (d \cdot \beta(e_j)) (d \cdot \beta(e_k))  \frac{q^{d \beta}}{d^3}
\end{equation}
is integral; we have
\begin{equation}
   \sum_{\beta \neq 0}  \langle e_i, e_j, e_k \rangle_\beta  \ q^\beta =  \sum_{\beta \neq 0} n_\beta \sum_{d \ge 1} \beta(e_i) \beta(e_j) \beta(e_k) q^{d \beta} \in \Lambda_{\mathbb{Z}}. 
\end{equation}

The coefficients $n_\beta$ are called the \emph{genus $0$ Gopakumar--Vafa invariants} or the \emph{instanton numbers}.

\subsection{Local $\mathbb{P}^1$ Calabi--Yau 3-fold}\label{ssec:local-p1-geometry}

Ionel--Parker's proof of Gopakumar--Vafa integrality is obtained by establishing a structure theorem for Gromov--Witten invariants, dubbed the \emph{cluster formalism}. The cluster formalism is reviewed in detail in \cref{ssec:clusters} below. 

Roughly speaking, the cluster formalism guarantees that the Gromov--Witten theory of Calabi--Yau threefolds is governed by the standard local model given by the local $\mathbb{P}^1$; In this section, we review the geometry of local $\mathbb{P}^1$ and describe its quantum connection.

\subsubsection{Local $\mathbb{P}^1$ and its small quantum connection}
We consider the geometry of a local $\mathbb{P}^1$, which is the noncompact Calabi--Yau $3$-fold
\begin{equation}
    Y = \mathrm{Tot} \left( \mathcal{O}(-1) \oplus \mathcal{O}(-1) \to \mathbb{P}^1 \right),
\end{equation}
the total space of the rank 2 vector bundle $\mathcal{O}_{\mathbb{P}^1}(-1)^{\oplus 2}$, together with its integrable complex structure and the K\"ahler form given by its description as a small resolution of the quadric cone $\{x^2+y^2+z^2+w^2 = 0\} \subseteq \mathbb{C}^4$.  In particular, the zero section $\mathbb{P}^1$ is a K\"ahler submanifold of $Y$. 

By maximum principle, all non-constant $J$-holomorphic spheres in $Y$ must map into the zero section $\mathbb{P}^1$. Hence, as a set, the moduli space of $J$-holomorphic spheres into $Y$ is isomorphic to the moduli space of $J$-holomorphic spheres into $\mathbb{P}^1$; however, their obstruction theories differ (the former carries an obstruction bundle arising from the normal directions, see \cite{Voi96, Lee23a}).

\subsubsection{Quantum operations on non-compact targets}
Since $Y$ is non-compact, the discussion of the quantum product as Novikov-linear endomorphisms of $H^*(Y) := \bar{H}^*(Y;\mathbb{Z}) = H^*(Y;\mathbb{Z})/H^*(Y;\mathbb{Z})_{tor}$ requires some modification. Recall that the (small) quantum product is defined using $J$-holomorphic genus $0$ spheres with $3$-marked points together with its evaluation maps,
\begin{equation}
    \mathrm{ev}: \overline{M}_{\beta} = \overline{\mathcal{M}}_{0,3} (Y,J,\beta) \to Y \times Y \times Y
\end{equation}
which defines a pseudocycle $\mathrm{ev}_*[\overline{\mathcal{M}}_{\beta}] \in H_*(Y^3;\mathbb{Z})$ assuming sufficient positivity of the target $Y$. The corresponding 3-pointed Gromov--Witten invariants
\begin{equation}
    \langle a_1, a_2, b \rangle_{0, 3, \beta} = \mathrm{ev}_*[\overline{\mathcal{M}}_{\beta}] \cap (a_1 \otimes a_2 \otimes b ) \in H_0(Y^3;\mathbb{Z}) \cong \mathbb{Z}
\end{equation}
for cohomology classes $a_1, a_2, b \in H^*(Y)$ of degree $|a_1| + |a_2| + |b| = \mathrm{vdim}_{\mathbb{R}} \overline{\mathcal{M}}_{\beta}$ are then used to define the structure constants of the small quantum product as
\begin{equation}
    b \star_q  a_1 = \sum_{\beta, \omega(\beta) \ge 0} \sum_{a_2} \langle a_1, a_2, b \rangle_{0, 3, \beta} \ q^\beta \ a_2^\vee
\end{equation}
where the sum $\sum_{a_2}$ ranges over a chosen basis $\{a_2\}$ of $H^*(Y)$. Here $a_2^\vee$ is the dual cohomology class under the unimodular Poincar\'e pairing $(\cdot, \cdot) : H^*(Y) \otimes H^*(Y) \to \mathbb{Z}$, inducing the dualizing map $a_2^\vee \in H^*(Y)^\vee \cong H^*(Y)$. In other words, the operation $b \star_q -$ is most naturally defined as a map $H^*(Y) \to H^*(Y)^\vee$, and we post-compose with the linear isomorphism $H^*(Y)^\vee \cong H^*(Y)$ induced by the Poincar\'e pairing to consider the quantum product as an endomorphism of $H^*(Y)$.

In the case of non-compact target $Y$, we modify this using the compactly supported cohomology $H^*(Y)^\vee \cong H^*_c(Y)$, and the (purely quantum part of) an operation defined by Gromov--Witten theory is naturally considered as a map $H^*(Y) \to H^*_c(Y)$.

\subsubsection{Quantum connection of local $\mathbb{P}^1$}\label{ssec:qconn-localP1}
In light of the discussion above, we consider both the cohomology and the compactly supported cohomology of local $\mathbb{P}^1$, $Y = \mathrm{Tot} \left( \mathcal{O}(-1) \oplus \mathcal{O}(-1) \to \mathbb{P}^1 \right).$

The cohomology of $Y$ is rank $2$; indeed, $H^*(Y) \cong H^*(\mathbb{P}^1)$ by the deformation retract to the zero section. It is generated by two classes $1 =[Y]\in H^0(Y)$ and $e = [\mathrm{fiber}] \in H^2(Y)$, Poincar\'e dual to the fundamental class (codimension $0$) and the fiber of $Y \to \mathbb{P}^1$ (codimension $2$) respectively.

The compactly supported cohomology of $Y$ is again rank $2$; by integration along fibers, $H^{*-4}_c(Y) \cong H^*(\mathbb{P}^1)$. It is generated by two classes $L = [\mathbb{P}^1] \in H^4_c(Y)$ and $[\mathrm{pt}] \in H^6_c(Y)$, Poincar\'e dual to the cycle of the zero section (codimension $4$) and a point in $Y$ (codimension $6$) respectively.

The degrees of $J$-holomorphic spheres in $Y$ are indexed by $H_2(Y;\mathbb{Z}) \cong H_2(\mathbb{P}^1;\mathbb{Z}) \cong \mathbb{Z}$, and the Novikov ring has the form    $\Lambda_{\mathbb{K}} = \mathbb{K}[\![q]\!].$

The expected dimension of the corresponding moduli space of genus $0$, $3$-pointed $J$-holomorphic spheres of  degree $\beta$ is
\begin{equation}
    \mathrm{vdim}_{\mathbb{R}} \overline{\mathcal{M}}_{0, 3} (Y,J,\beta) = \dim_{\mathbb{R}} Y + 2 c_1(\beta) = 6
\end{equation}
as $c_1  = 0$. Hence, Gromov--Witten invariants are only non-trivial for three insertions given by $H^2$-classes. The computation of the corresponding quantum product is due to Voisin:

\begin{thm}[\cite{Voi96}]\label{thm:localP1-qconn}
    Let $Y$ be the local $\mathbb{P}^1$ Calabi--Yau $3$-fold and choose the basis $\langle 1, e, L, [\mathrm{pt}] \rangle \in H^*(Y) \oplus H^*_c(Y)$. The matrix of the quantum product by $e \in H^2(Y)$, the Poincar\'e dual of the fiber class, is given in this basis by
    \begin{equation}
        e  \star_q - = \begin{pmatrix} 0 & 0 & 0 & 0 \\ 1 & 0 & 0 & 0 \\ 0 & \frac{q}{1-q} & 0 & 0 \\ 0 & 0& 1& 0 \end{pmatrix} : H^*(Y; \Lambda_{\mathbb{K}}) \oplus H^*_c(Y; \Lambda_{\mathbb{K}}) \to H^*(Y; \Lambda_{\mathbb{K}}) \oplus H^*_c(Y; \Lambda_{\mathbb{K}}).
    \end{equation}
\end{thm}
\begin{proof}
    The quantum product by $e \in H^2(Y)$ is an operation which raises cohomological degree by $|e|=2$, where the degree is well-defined as a $\mathbb{Z}$-grading as $c_1(Y) = 0$. As our basis elements are chosen to be in degrees $0, 2, 4, 6$, the corresponding matrix can only have non-zero entries on the leading subdiagonal. The classical terms ($q=0$) follow from the fact that (i) the fiber and the whole space intersects along the fiber, and (ii) the fiber and the zero section intersects at a single point. The quantum terms follow from the fundamental computation of Voisin \cite[Theorem 1.1]{Voi96} that $\langle e, e, e\rangle_{0, 3, \beta} = 1$ for all $\beta = k [\mathbb{P}^1]$, $k \ge 1$, leading to the sum 
    \begin{equation}
    \sum_{k \ge 1} \langle e, e, e \rangle_{0, 3, k[\mathbb{P}^1]} \   q^k = q + q^2 +  q^3 + \cdots = \frac{q}{1-q}.
    \end{equation}
    Note that this computation is equivalent to the Aspinwall--Morrison formula $\langle  \rangle_{k [\mathbb{P}^1]} = \frac{1}{k^3}$ by the divisor axiom. 
\end{proof}

\begin{cor}\label[cor]{cor:localP1-qconn-nilp}
    The matrix $\mathbb{M}_e := \mathbb{M}= e \  \star_q $ of the quantum product by the fiber class is nilpotent, $\mathbb{M}^4 = 0$. 
More generally, any word of length $\ge 4$ in the matrix $\mathbb{M}$ and its derivatives $\mathbb{M}^{(j)}:= (q\partial_q)^jB$ must be zero.
\end{cor}
The above corollary is already immediate by noting that $B$ is an endomorphism of cohomological degree $2$ acting on a vector space with bounded degrees.

\section{The A-model Fontaine--Laffaille modules}
The goal of this section is to provide background on logarithmic Fontaine--Laffaile modules and use genus $0$ Gopakumar--Vafa integrality to prove Theorem \ref{thm:integrality}.

\subsection{Logarithmic Fontaine--Laffaille modules}\label{subsection:FL definitions}
 In this section, we give the precise definition of logarithmic Fontaine-Laffaille module which we need. Fix a ground ring $\mathbb{K}$. In the discussion below, we freely use notation from \S \ref{ssec:qhCY3}.
 

 In coordinates, the natural log structure on $\operatorname{Spf}(\Lambda_{\mathbb{K}})$ corresponds to the natural divisorial log structure along $\lbrace q_r=0 \rbrace \subset \operatorname{Spf}(\Lambda_{\mathbb{K}})$. We consider the module of $q_r$-adically continuous log differentials $\hat{\Omega}_{\Lambda_\mathbb{K}}$. There is a natural isomorphism: 
\begin{align} \label{eq:logarithmicforms} \Lambda_\mathbb{K} \otimes \bar{H}_2(X,\mathbb{Z}) \cong \hat{\Omega}_{\Lambda_\mathbb{K}} , \quad \beta \mapsto d\operatorname{log}(q^{\beta})  \end{align}

In this formal setting, an integrable logarithmic connection on a free finite rank $\Lambda_\mathbb{K}$-module $\calH$ will be a map \begin{align} \nabla: \calH \to \calH \otimes_{\Lambda_\mathbb{K}}\hat{\Omega}_{\Lambda_\mathbb{K}}  \end{align} which satisfies the usual Leibnitz rule and integrability condition $\nabla \circ \nabla=0.$ In terms of the coordinates $q_1,\cdots,q_r$ and a choice of basis for $\calH$, a logarithmic connection is encoded in a matrix valued one-form: 
\begin{align} \label{eq:connectioncoordinates} \sum_k M_k(q_1,\cdots,q_r)\frac{dq_k}{q_k},\quad M_k \in \operatorname{Mat}_{s\times s} (\Lambda_{\mathbb{K}}),\end{align}
where $s$ is the rank of $\calH$ as a module over $\Lambda_{\mathbb{K}}.$
Suppose $\calH$ carries a finite descending filtration: \begin{align}\label{eq:filtration} \calH=\operatorname{Fil}^w \supset \cdots \operatorname{Fil}^{i-1} \supset \operatorname{Fil}^i \supset \operatorname{Fil}^{i+1}  \cdots \end{align} 
\begin{defn} We say that $(\calH,\operatorname{Fil}^i,\nabla)$  satisfies Griffiths transversality if: \begin{align}\nabla(\operatorname{Fil}^i) \subset \operatorname{Fil}^{i-1}.\end{align}\end{defn}
Now fix a prime $p>2$ and set $\mathbb{K}=\mathbb{Z}_p$. The ``standard" lift of Frobenius \begin{align} \mathbb{Z}_p[H_2(X,\mathbb{Z})_{\geq 0}] \to \mathbb{Z}_p[H_2(X,\mathbb{Z})_{\geq 0}], \quad q^{\beta} \to q^{p\beta} \end{align} passes to completions and induces a map \begin{align} \label{eq:baseFmap} F: \Lambda_{\mathbb{Z}_p} \to \Lambda_{\mathbb{Z}_p}, \quad q^{\beta} \to q^{p\beta}. \end{align} 
The pullback module $F^*\calH$ is given by: \begin{align} F^*\calH := \calH \otimes_{\Lambda_{\mathbb{Z}_p}} \Lambda_{\mathbb{Z}_p}, \end{align} where $\Lambda_{\mathbb{Z}_p}$ acts on itself via \eqref{eq:baseFmap}. Let $j: \calH \to F^*\calH$ be the natural inclusion. The pull-back connection $F^*\nabla$ on $F^*\calH$ is the connection determined by the diagram:

\begin{center}
\begin{tikzcd}[row sep=large, column sep=huge]
\calH \arrow[r, "\nabla"] \arrow[dr, "F^*\nabla \circ j "'] & \calH \otimes \hat{\Omega}_{\Lambda_{\mathbb{Z}_p}} \arrow[d, "j \otimes dF"] \\
& F^*\calH \otimes \hat{\Omega}_{\Lambda_{\mathbb{Z}_p}},
\end{tikzcd}
\end{center} 
In terms of the matrix presentation \eqref{eq:connectioncoordinates}, the pull-back connection $F^*\nabla$ corresponds to:  
\begin{align} \label{eq:pullbackcoordinates} \sum_k pM_k(q_1^p,\cdots,q_r^p)\frac{dq_k}{q_k},\quad M_k \in \operatorname{Mat}_{s\times s} (\Lambda_{\mathbb{Z}_p}).\end{align}

\begin{defn} We say that a filtered module $\calH$ is filtered free if $\calH$, each $\operatorname{Fil}^i$, and each quotient $\calH/\operatorname{Fil}^i$ are all finite free modules.    \end{defn}
  The following Definition comes from \cite{Faltings1989} (see \cite[\S 2]{logFLpaper} for a detailed exposition of the log case):

\begin{defn}\label{defn:fontaine} A logarithmic Fontaine-Laffaille module is a quadruple $(\calH, \nabla, \mathrm{Fil}, \Phi)$ where

\begin{enumerate}
    \item \( (\calH, \mathrm{Fil}) \) is a filtered free $\Lambda_{\mathbb{Z}_p}$-module of finite rank equipped with a logarithmic connection \( \nabla \) satisfying Griffiths transversality.
     \item The (relative) Frobenius is an $\Lambda_{\mathbb{Z}_p}$-linear morphism 
    \[
    \Phi : F^{*} \calH \longrightarrow \calH \] which is horizontal with respect to the
    connections \( F^{*}\nabla \) on \( F^{*} \calH \) and \( \nabla \) on \( \calH \). Namely, the following diagram is commutative:
    $$
\begin{tikzcd}[row sep=large, column sep=large]
F^*\calH \arrow[r, "\Phi"] \arrow[d, "F^*\nabla"'] & \calH \arrow[d, "\nabla"] \\
F^*\calH \otimes \hat{\Omega}_{\Lambda_{\mathbb{Z}_p}} \arrow[r, "\Phi \otimes \mathrm{Id}"'] & \calH \otimes \hat{\Omega}_{\Lambda_{\mathbb{Z}_p}}
\end{tikzcd}
$$ 
    
    \item This data must satisfy the strong \( p \)-divisible property:
    \[
    \Phi(F^{*} \mathrm{Fil}^i) \subset p^i \calH, \quad \sum_{i=0}^{w} \frac{1}{p^i} \Phi(F^{*} \mathrm{Fil}^i) = \calH.
    \]
   
\end{enumerate} \end{defn} 

In coordinates, the condition that $\Phi$ is horizontal becomes the following system of $p$-adic differential equations: for any $1 \leq k \leq r$, we have
\begin{align} \label{eq:differentialequations} q_k\partial_{q_k}(\Phi(q_1, \cdots, q_r)) + M_k(q_1, \cdots, q_r)\Phi(q_1, \cdots, q_r)-p\Phi(q_1, \cdots, q_r)M_k(q_1^p, \cdots, q_r^p)=0. \end{align}

We use this ``elementary" interpretation to make an observation that will be used in \S \ref{sect:FLonQH}.  Let's consider a larger coefficient ring  \begin{align} \label{eq:largerring} L= \mathbb{Q}_p(q_1,\cdots,q_{r-1})[[q_r]]  \end{align} 
and look for solutions $\Phi(q)$ to \eqref{eq:differentialequations} with coefficients in $L$. The leading order term of such a solution will be its reduction modulo $q_r$: \begin{align} \label{eq:initialterm} \Phi_0 = \Phi(q)_{|{q_r}=0} \in \operatorname{Mat}_{s\times s}(\mathbb{Q}_p(q_1,\cdots,q_{r-1})). \end{align}
\begin{lemma} \label{lem:initialtermdeterminesFrob} Fix a matrix $\Phi_0 \in \operatorname{Mat}_{s\times s}(\mathbb{Q}_p(q_1,\cdots,q_{r-1})).$ There is at most one solution $\Phi(q)$ to  \eqref{eq:differentialequations} with coefficients in $L$ whose leading order term is $\Phi_0.$ \end{lemma}
\begin{proof} The argument from \cite[Lemma 3.4(i)]{bai2025p} shows that the solution $\Phi(q) \in \operatorname{Mat}_{s\times s}(L)$ to the differential equation \eqref{eq:differentialequations} with $k=r$ and initial term \eqref{eq:initialterm} is unique if it exists. Thus, the solution to the entire system is unique if it exists. (The standing conventions in that paper consider power series solutions in a different $p$-adic field from $\mathbb{Q}_p(q_1,\cdots,q_{r-1})$, but this does not impact the argument). \end{proof}

\subsection{Logarithmic $p$-torsion Fontaine--Laffaille modules}\label{ssec:ptorsion-FL}
We will also make use of the following $p$-torsion version of a Fontaine-Laffaille structure, which now lives over the base $\operatorname{Spf}(\Lambda_{\mathbb{F}_p}).$

\begin{defn}\label{defn:ptorsion_fontaine} A $p$-torsion (logarithmic) Fontaine-Laffaille module will be a tuple $$(\bar{\mathcal{H}}, \bar{\nabla}, \bar{\mathrm{Fil}}^i, \lbrace \phi_i\rbrace)$$ where
\begin{enumerate}
    \item $(\bar{\mathcal{H}}, \bar{\nabla}, \bar{\mathrm{Fil}}^i)$ is again a filtered free module of finite rank, now over $\Lambda_{\mathbb{F}_p}$, equipped with a logarithmic connection $\bar{\nabla}$ satisfying Griffiths transversality.
 \item $\phi_i : F^*\bar{\mathrm{Fil}}^i \to \bar{\mathcal{H}}$ is a $\Lambda_{\mathbb{F}_p}$-linear map such that
 \begin{align}\label{eq:vanishingcondition} \phi_i|_{F^*\bar{\mathrm{Fil}}^{i+1}} = 0. \end{align}
   The  $\phi_i$ maps are parallel, which in this context means the following diagram commutes:

\begin{center}
\begin{tikzcd}[row sep=large, column sep=huge]
\bar{\operatorname{Fil}}^{i} \arrow[r, "\phi_i"] \arrow[d, "\nabla_i"'] & \bar{\mathcal{H}} \arrow[d, "\nabla"] \\
\bar{\operatorname{Fil}}^{i-1} \otimes_{\Lambda} \widehat{\Omega}_{\Lambda_{\mathbb{F}_p}} \arrow[r, "\phi_{i-1} \otimes \frac{dF}{p}"'] & \bar{\mathcal{H}} \otimes_{\Lambda} \widehat{\Omega}_{\Lambda_{\mathbb{F}_p}}
\end{tikzcd}
\end{center}.
    \item The strong $p$-divisibility property becomes
    \begin{equation}\label{eqn:assemble}
    \sum_{i} \phi_i(F^*\bar{\operatorname{Fil}}^i) = \bar{\mathcal{H}}.
    \end{equation}
\end{enumerate}
\end{defn}

The conditions \eqref{eq:vanishingcondition} and \eqref{eqn:assemble} implies that the maps $\phi_i$ assemble to give an isomorphism: \begin{align} \label{eq:inversecartiergeneral} \mathcal{C}^{-1}: \bigoplus_i \frac{F^*\bar{\operatorname{Fil}}^i}{F^*\bar{\operatorname{Fil}}^{i+1}} \cong \bar{\mathcal{H}}. \end{align}

Given a Fontaine-Laffaille structure in the sense of Definition \ref{defn:fontaine}, its reduction $\bar{\calH}=\calH/p\calH$, naturally carries the structure of a $p$-torsion Fontaine-Laffaille module: \begin{itemize} \item The filtration and connection $\bar{\operatorname{Fil}}^i$ and $\bar{\nabla}$ are the reductions modulo $p$ of $\operatorname{Fil}^i$ and $\nabla.$ \vskip 5 pt \item The maps $\phi_i$ are given by the divided Frobenius operators: \begin{align} \phi_i = \frac{1}{p^i} \Phi_{|\operatorname{Fil}^i}: \bar{\operatorname{Fil}}^i \to \bar{\calH}, \end{align}
where $\frac{1}{p^i} \Phi_{|\operatorname{Fil}^i}$ is understood modulo $p.$ \end{itemize}

\subsection{The Fontaine--Laffaille structure on $QH^*$}\label{sect:FLonQH}
In this section, we construct a logarithmic Fontaine-Laffaille structure on the quantum connection. The construction is inspired by calculations in \cite{KSV06} (in particular the discussion around \cite[Lemma 3]{KSV06}). To begin, we let $\Gamma_p(z)$ denote the  Morita $p$-adic Gamma function. The continuous function $\Gamma_p(z): \mathbb{Z}_p \to \mathbb{Z}_p^\times$ is determined by
\begin{equation}
\Gamma_p(z+1) = 
\begin{cases}
 -z \cdot\Gamma_p(z) & \text{if } z \in \mathbb{Z}_p^{\times}, \\
-\Gamma_p(z) & \text{if } z \in p\mathbb{Z}_p.
\end{cases}
\end{equation}
We refer to \cite[Chapter 11]{cohen-vol2} for more background on $\Gamma_p(z)$. We use the $p$-adic $\Gamma$ function to define a characteristic class:
\begin{equation}
\Gamma_p(E) \in H^{\mathrm{even}}(X;\mathbb{Q}_p)[u,u^{-1}]
\end{equation}
for any complex vector bundle $E \rightarrow X$. In terms of Chern roots $c(E) = 1 + c_1(E) + c_2(E) + \cdots = \prod_{i=1}^n (1+r_i)$, 
\begin{equation} \label{eq:gamma-characteristic-class}
\begin{aligned}
\Gamma_p(E) = \prod_{i=1}^n \Gamma_p(u^{-1}\cdot r_i) &
= 1 + u^{-1}\cdot \Gamma_p'(0)c_1(E) + u^{-2}\cdot \Gamma_p'(0)^2 (c_1(E)^2/2) + 
\\[-.5em] 
& 
+ u^{-3}\cdot \Gamma_p'''(0) \big(c_3(E)/2 - c_1(E)c_2(E)/2 + c_1(E)^3/6 \big) \\ 
& 
+ u^{-3}\cdot \Gamma_p'(0)^3 \big(-\!c_3(E)/2 + c_1(E)c_2(E)/2 \big) + \cdots
\end{aligned}
\end{equation}
It follows immediately from this expression that if $E=TX$ for $X$ a Calabi-Yau 3-fold, then $\Gamma_p(TX)$ reduces to \eqref{eq:Gammap}.

Below, we use the notation from \cref{ssec:qhCY3}. Recall that in the basis $e_1,\cdots, e_r \in H^2(X;\mathbb{Z})$, there are Yukawa couplings \begin{align} Y_{ijk}(q)=\sum_{\beta,\omega(\beta)\geq 0} \langle e_i,e_j,e_k\rangle _{0,3,\beta}\ q^{\beta}.  \end{align}




We also have the genus zero Gromov-Witten potential $G(q)$ given by: \begin{align} \label{eq:GWpotential} G(q)= \sum_{\beta \neq 0} \langle \rangle_\beta \ q^{\beta}, \end{align} where $\langle \rangle_\beta$ denote zero-pointed genus zero Gromov-Witten invariants. 

\begin{defn} Define the $p$-adically integral Gromov-Witten potential $\eta(q)$ to be \begin{align} \label{eq:removepdivisors} \eta(q)= \frac{1}{p^3}G(q^p)-G(q) \in \Lambda_{\mathbb{Q}}, \end{align}  \end{defn}

 The name ``$p$-adically integral Gromov-Witten potential" is justified by the following, which is a consequence of Ionel-Parker's GV integrality theorem \cref{thm:gv-integrality}:

\begin{lemma} \label{lem:Kontsevichlemma}
Under the embedding $\Lambda_\mathbb{Q} \subset \Lambda_{\mathbb{Q}_p}$, the function $\eta(q)$ lies in $\Lambda_{\mathbb{Z}_p}.$
\end{lemma}
\begin{proof} Let $\beta \in H_2(X;\mathbb{Z})$ be a class with $\omega(\beta)>0$ and consider \begin{align} \mathrm{Li}_3(q^{\beta})=\sum_{d \geq 1} \frac{q^{d\beta}}{d^3} \in \Lambda_{\mathbb{Q}_p}. \end{align} Then \begin{align} \label{eq:nop} \mathrm{Li}_3(q^{\beta})-\frac{1}{p^3}\mathrm{Li}_3(q^{p\beta}) = \sum_{d \geq 1} \frac{q^{d\beta}}{d^3}- \sum_{d \geq 1} \frac{q^{pd\beta}}{p^3d^3} =\sum_{p \nmid d} \frac{q^{d\beta}}{d^3}. \end{align} It follows that \begin{align} \frac{1}{p^3}\mathrm{Li}_3(q^{p\beta})-\mathrm{Li}_3(q^{\beta}) \in \Lambda_{\mathbb{Z}_p}. \end{align}

By the Ionel--Parker GV integrality theorem, we can write $G(q)$ as a sum:
\begin{align} G(q) = \sum_{\beta \neq 0} n_\beta \mathrm{Li}_3(q^\beta).  \end{align}
where $n_\beta \in \mathbb{Z}$ are the instanton numbers. Consequently, \begin{align} \eta(q)= \sum_{\beta  \neq 0} n_\beta \left(\frac{1}{p^3}\mathrm{Li}_3(q^{p\beta})-\mathrm{Li}_3(q^{\beta})\right) \end{align} and the result follows.  
\end{proof}
\begin{rem} In view of \eqref{eq:nop}, the function $\eta(q)$ should be thought of as the function obtained by removing the contribution of multiple cover curves whose covering degree $d$ is divisible by $p$ from $G(q)$.  \end{rem}

We will use $\eta(q)$ to construct a Frobenius intertwiner $\Phi$ whose matrix coefficients lie in $\Lambda_{\mathbb{Z}_p}$. Recall from Lemma \ref{lem:yukawa-as-third-derivative} that
\begin{align} \delta_i \delta_j \delta_k G(q)=Y_{ijk}(q)-Y_{ijk}(0), \end{align} where $\delta_i=q_i\partial_{q_{i}}$ and similarly for $\delta_j,\delta_k$. By construction, the function \eqref{eq:removepdivisors} satisfies: \begin{align} \label{eq:characterizing equation} \delta_i \delta_j \delta_k \eta(q)=Y_{ijk}(q^p)-Y_{ijk}(q). \end{align}
For any $\zeta \in \mathbb{Q}_p$, define \begin{align} \eta_\zeta(q)=\zeta+\eta(q). \end{align}

Let us denote the degree $2k$ piece of $QH^{*}(X,\Lambda_{\mathbb{Z}_p})[u^\pm]$, by $\calH^{2k}.$ We first construct the Frobenius intertwiner on the degree $0$ piece, $\calH^0.$

\begin{lemma} \label{lem:integralityusingKSV} 

Fix $\zeta \in \mathbb{Q}_p$ with p-adic valuation $\operatorname{val}_p(\zeta)>-3.$ Fix as well the preferred homogeneous basis for $\calH^0$ given by $$1 \in H^0(X), u^{-1} e_i \in H^2(X), u^{-2}L_i \in H^4(X), u^{-3}[\mathrm{pt}] \in H^6(X).$$  Define a map $\Phi$ on this basis by the formula: 
\begin{align}
\Phi (1) &= p^3 1 +p^3 u^{-2} \sum_j\delta_j(\eta_\zeta (q)) L_j -2p^3 u^{-3} \eta_\zeta (q) [\mathrm{pt}], \\
\Phi (u^{-1}e_i) &= p^2 (u^{-1}e_i) + p^2 u^{-2}  \sum_j \delta_{i}\delta_{j}(\eta_\zeta(q))L_j -p^2 u^{-3}\delta_i(\eta_\zeta (q)) [\mathrm{pt}],\nonumber \\
\Phi (u^{-2}L_i) &= p \cdot u^{-2}L_i, \\
\Phi (u^{-3}[\mathrm{pt}])&=u^{-3}[\mathrm{pt}]. \nonumber  
\end{align}
The above map defines a horizontal Frobenius intertwiner: 
\begin{align}\label{eq:PhiH0} \Phi :  F^*\calH^0 \to \calH^0.\end{align} \end{lemma} 
\begin{proof} We will use the shorthand $M_k=u^{-1}\mathbb{M}_{e_k}$, where $\mathbb{M}_{e_k}$ are the matrices for the quantum product from \eqref{eq:matricesquantumproduct}. We wish to check that this $\Phi$ satisfies the matrix differential equations \eqref{eq:differentialequations} for being a Frobenius intertwiner. Due to the form of the connection and intertwiner, these equations are especially simple to verify when the inputs are $u^{-3}[pt]$, $u^{-2}L_i$. We therefore consider two cases: \vskip 5 pt
\emph{ Input is $1$:} We have $\delta_k \Phi (1) =  p^3u^{-2} \sum_j\delta_k\delta_j(\eta_\zeta(q)) L_j - 2p^3 u^{-3} \delta_k(\eta_\zeta(q)) [\mathrm{pt}]$, and   
$$M_k(q)\Phi(1)= p^3u^{-1} e_k+p^3 u^{-3}\delta_k(\eta_{\zeta}(q))[\mathrm{pt}],$$ 
$$ p\Phi(q) M_k(q^p)(1)= p^3 u^{-1}e_k+ p^3 u^{-2}\sum_j \delta_{k} \delta_j(\eta_{\zeta}(q))L_j - p^3 u^{-3}\delta_k(\eta_{\zeta}(q)) [\mathrm{pt}]. $$

 \vskip 5 pt
\emph{Input is $u^{-1}e_i$:} 
We have $\delta_k \Phi (u^{-1}e_i)= p^2 u^{-2} \sum_j \delta_{i} \delta_j \delta_{k}(\eta_{\zeta}(q))L_j -p^2 u^{-3}\delta_{i} \delta_k (\eta_{\zeta}(q)) [\mathrm{pt}]$ 

$$M_k(q)\Phi(u^{-1}e_i)= p^2 u^{-2}\sum_j Y_{ijk}(q)L_j +p^2 u^{-3}\delta_{i} \delta_{k}(\eta_{\zeta}(q))[\mathrm{pt}]$$

$$p\Phi(q)\circ M_k(q^p)(u^{-1}e_i)= p^2u^{-2}\sum_j Y_{ijk}(q^p)L_j.$$
\end{proof}

\begin{proof}[Proof of Theorem \ref{thm:integrality}]
     For $\Phi$ as in Lemma \ref{lem:integralityusingKSV}, if we set $q=0$, we see that its reduction modulo the ideal $\Lambda_{+}$, $\Phi_0 $, satisfies
    \begin{equation}
        \begin{aligned}
            \Phi_0(1) &= p^31 -2p^3u^{-3}\zeta[pt], \\
            \Phi_0(u^{-1}[e_i]) &= p^2(u^{-1}[e_i]), \\
            \Phi_0(u^{-2}[L_i]) &= pu^{-2}[L_i], \\
            \Phi_0(u^{-3}[pt]) &= u^{-3}[pt].
        \end{aligned}
    \end{equation}
    We set $\zeta = \frac12 \zeta_p(3)\chi(X)$. By Lemma \ref{lem:pintegrality}, this satisfies the necessary condition on $\operatorname{val}_p(\zeta)$. Note that multiplication by $u^k$ gives identifications: \begin{align}\calH^{0} \cong \calH^{2k},\quad F^*\calH^{0} \cong F^*\calH^{2k}.
    \end{align} 
   We extend $\Phi$ to $\calH^{2k}$ by requiring that \begin{align} \Phi(u^k\cdot x)=p^ku^k \cdot \Phi(x),\quad x \in \calH^{0}. \end{align}
   This is again a horizontal Frobenius (on each piece $\calH^{2k}$ being just a rescaling of the Frobenius on $\calH^{0}$) and now has leading order term  \eqref{eq:b-frobenius}. By construction, it also satisfies: \begin{align} \Phi(F^*\calH^{2k}) \subset p^k \calH^{2k} \end{align} as required. Lastly, uniqueness is guaranteed by Lemma \ref{lem:initialtermdeterminesFrob}.
\end{proof}

\begin{proof}[Proof of Corollary \ref{cor:FLclassical}] To begin, note that we have assumed that $H^{even}(X;\mathbb{Z}_p)$ contains no torsion, which implies that $\mathcal{H}^0$ is filtered-free as a module over $\Lambda_{\mathbb{Z}_p}.$ The constant $\zeta = \frac12 \zeta_p(3)\chi(X)$ is $p$-adically integral for $p>3$ by Lemma \ref{lem:pintegrality}. This implies that $\Phi$ in \eqref{eq:PhiH0} satisfies the strong divisibility property as required.  \end{proof}


We conclude this section with two remarks. 

\begin{rem} For certain monotone symplectic manifolds (e.g. $\mathbb{C}P^n$), the Frobenius intertwiner constructed using the $p$-adic $\Gamma$ class in \cite{bai2025p} is the unique overconvergent Frobenius intertwiner for the quantum connection. By contrast, in the Calabi-Yau case, integrality considerations alone do not seem to complete pin-down the Frobenius intertwiner;  the value  $\zeta = \frac12 \zeta_p(3)\chi(X)$ above could be replaced by any $p$-adic integer. (In fact, this value leads to the assumption $p>3$ in Corollary \ref{cor:FLclassical}.) However, as noted in the introduction, this value appears to be the correct one from the point of view of mirror symmetry and non-commutative geometry.  \end{rem}

\begin{rem} As in \cite[Lemma 3.4]{bai2025p}, we can express the Frobenius intertwiner in  Theorem \ref{thm:integrality} in terms of Givental's fundamental solution $\Pi(q)$ for the small quantum connection (\cite{cox-katz}): 
\begin{equation}\label{eq:frobenius_formula}
\Phi(q) = \Pi(q)\Phi_0\Pi^{-1}(q^p),
\end{equation}
 To see this, note that the right-hand side of this expression gives a solution to \eqref{eq:differentialequations} with coefficients in  the ring $L$ from \eqref{eq:largerring} with leading order term $\Phi_0.$ By Lemma \ref{lem:initialtermdeterminesFrob}, we must have equality in \eqref{eq:frobenius_formula} \end{rem}

\section{Structures on the mod $p$ quantum connection}

In this section, we discuss the structures carried by the \emph{mod $p$ reduction} of the loop-equivariant quantum cohomology. Unlike their $p$-adic counterparts, these objects are now modules over characteristic $p$ fields, and the datum of the horizontal Frobenius intertwiner on the $p$-adic quantum connection induces the new structure of \emph{inverse Cartier maps} on the A-model $p$-torsion Fontaine--Laffaille module.

The mod $p$ quantum connection naturally induces the associated $p$-curvature, and we state the analogue of Katz' observation \cite{Kat70} that the $p$-curvature of the quantum connection is related, up to conjugating by the inverse Cartier, to the Frobenius twist of the quantum products. 

On the other hand, the $p$-curvature of the quantum connection has a geometric incarnation in terms of quantum Steenrod operations, cf. Theorem \ref{thm:main}. We provide background on quantum Steenrod operations and defer the proofs of the main theorems to the next Section.

\subsection{The mod $p$ quantum connection and $p$-curvature}

For $\mathbb{K} = \mathbb{F}_p$, the quantum connection can be extended to the $\mu_p$-equivariant quantum cohomology. Here we consider the subgroup $\mu_p \subseteq \mathbb{G}_m$, the group of complex $p$-th roots of unity, in our constructions.

\begin{lemma}\label[lemma]{lem:group-cohomology-mup}
    The group cohomology of $\mu_p$ with $\mathbb{F}_p$-coefficients
    \begin{equation}
        H^*(B\mu_p ; \mathbb{F}_p) \cong \mathbb{F}_p [u]\langle \theta \rangle
    \end{equation}
    is a graded polynomial algebra on two formal variables with degrees $|u|=2$ and $|\theta|=1$. If $p =2$, these generators are subject to the relation $u = \theta^2$, and if $p > 2$, $u\theta = \theta u $ and $\theta^2 = 0$.
\end{lemma}

Indeed, the map induced on equivariant cohomology by the inclusion $\mu_p \subseteq \mathbb{G}_m$ induced the inclusion
\begin{equation}
    H^*(B \mathbb{G}_m; \mathbb{F}_p) \cong \mathbb{F}_p [u] \to H^*(B\mu_p;\mathbb{F}_p) \cong \mathbb{F}_p[u]\langle \theta \rangle
\end{equation}
where the generator $u$ in $H^2(B\mathbb{G}_m;\mathbb{F}_p)$ is sent to the corresponding degree $2$ generator $u$ in $H^2(B\mu_p;\mathbb{F}_p)$.

\begin{defn}
    The $\mu_p$-equivariant quantum cohomology is the graded $\Lambda$-module defined as
    \begin{equation}
        QH^*(X;\Lambda_{\mathbb{F}_p})[u]\langle \theta \rangle  := H^*(X;\Lambda_{\mathbb{F}_p})[u]\langle \theta \rangle
    \end{equation}
    where $|u| = 2$ and $|\theta|=1$. We also consider the version where $u$ is inverted, \begin{align} 
QH^*(X;\Lambda_{\mathbb{F}_p})[u^\pm]\langle \theta \rangle:=H^{even}(X; \Lambda_{\mathbb{F}_p})[u,u^{-1}]\langle \theta \rangle . 
\end{align}
\end{defn}
Here, we again remark that there is no $\mu_p$-action on the target symplectic manifold $X$; the action of $\mu_p$ should be considered as a \emph{discretized loop rotation} on the source curve $C$ for $J$-holomorphic maps $u: C \to X$.

In characteristic $p$, the $p$-th iterate $\partial_b^p$ of a derivation $\partial_b$ is again a derivation. In our case, it acts on $q^\beta \in \Lambda$ the same way as $\partial_b$. It is natural to ask whether or not the association $\partial_b \mapsto \nabla_b$ preserves such $p$-th power operations.

\begin{defn}
    The \emph{$p$-curvature} of the quantum connection is the endomorphism 
    $$\psi_b: QH^*(X;\Lambda_{\mathbb{F}_p})[u^\pm ]\langle \theta \rangle \to  QH^*(X;\Lambda_{\mathbb{F}_p})[u^\pm ] \langle \theta \rangle$$ defined by
    \begin{equation}
        \psi_b := \nabla_b^p - \nabla_b.
    \end{equation}
    where the right hand side is extended $\theta$-linearly.
\end{defn}
A remarkable fact is that the $p$-curvature is a \emph{$\Lambda$-linear} endomorphism of $QH^*(X;\Lambda_{\mathbb{F}_p})[u^\pm ]\langle \theta \rangle$, even though the individual $\nabla_b$ is certainly not. The general proof of linearity can be found in \cite[Section 5]{Kat70}. 

Note also that
\begin{equation}
    u^p \psi_b  = (u \nabla_b)^p - u^{p-1} (u\nabla_b) : QH^*(X;\Lambda_{\mathbb{F}_p})[u ]\langle \theta \rangle \to QH^*(X;\Lambda_{\mathbb{F}_p})[u ]\langle \theta \rangle
\end{equation}
may be defined without inverting $u$; it is an operation of degree $p|u| = 2p$.

\subsection{Cartier isomorphism and Katz' formula}

Fix $p>3$. As described in \cref{ssec:ptorsion-FL}, the Fontaine--Laffaile structure on the degree zero piece $\calH^0$ of $QH^*(X;\Lambda_{\mathbb{Z}_p})[u^\pm]$ prescribed by the Frobenius intertwiner
\begin{equation}
    \Phi: F^*\calH^0 \to \calH^0
\end{equation}
now induces a $p$-torsion Fontaine--Laffaille structure on the mod $p$ reduction:
\begin{equation}\bar{\calH}^0:=\calH^0/p\calH^0.  \end{equation}


We may extend the inverse Cartier operator $\mathcal{C}^{-1}$ from \eqref{eq:inversecartiergeneral} to a map
\begin{equation}
    \mathcal{C}^{-1}: F^* QH^* (X;\Lambda_{\mathbb{F}_p})[u^\pm] \langle \theta \rangle \to QH^* (X;\Lambda_{\mathbb{F}_p})[u^\pm] \langle \theta \rangle
\end{equation}
by extending $\mathbb{F}_p[u^{\pm}]\langle \theta \rangle$-linearly.

\begin{lemma}\label{lem:Cartier-computation}
In the preferred homogeneous basis given by $$1 \in H^0(X), u^{-1}e_i \in H^2(X), u^{-2}L_i \in H^4(X), u^{-3}[\mathrm{pt}] \in H^6(X),$$ the inverse Cartier operator is given by
\begin{align}
\mathcal{C}^{-1} (1) &= 1 + u^{-2}  \sum_i\delta_i (\eta_{\zeta}(q))   L_i - 2 u^{-3} \eta_{\zeta}(q)  [\mathrm{pt}], \\
\mathcal{C}^{-1} (u^{-1}e_i) &=  u^{-1}e_i +  u^{-2} \sum_j \delta_{i}\delta_{j}(\eta_{\zeta}(q)) L_j  - u^{-3}\delta_i(\eta_{\zeta}(q)) [\mathrm{pt}],\nonumber \\
\mathcal{C}^{-1} (u^{-2}L_i) &= u^{-2}L_i, \\
\mathcal{C}^{-1} (u^{-3}[\mathrm{pt}])&=u^{-3}[\mathrm{pt}] \nonumber  .
\end{align}
\end{lemma}
\begin{proof}
    The computation follows from the definition and the computation of the Frobenius intertwiner, \cref{lem:integralityusingKSV}.
\end{proof}

Katz's celebrated resolution of the $p$-curvature conjecture for Gauss--Manin connections involves the identification of the $p$-curvature endomorphisms (more precisely, a suitable associated graded thereof) with the action of the cup product by the Kodaira--Spencer class conjugated by $\mathcal{C}^{-1}$. 

Using the well-known interpretation of the quantum cup product as the mirror of the Kodaira--Spencer class, one may formulate the following $A$-side analogue of Katz's formula:

\begin{thm}\label{thm:Cartierconjugation-body}
    Let
    \begin{equation}
        \psi_b : QH^{*}(X;\Lambda_{\mathbb{F}_p})[u^\pm]\langle \theta \rangle \to QH^{*}(X;\Lambda_{\mathbb{F}_p})[u^\pm]\langle \theta \rangle
    \end{equation}
    be the $p$-curvature operator corresponding to $b \in H^2(X;\mathbb{K})$. Then the following diagram commutes:

    \begin{center}
        \begin{tikzcd}
             F^*QH^{*}(X;\Lambda_{\mathbb{F}_p})[u^\pm]\langle \theta \rangle \arrow[rr, "-u^{-1} F^*\mathbb{M}_b"] \dar["\mathcal{C}^{-1}"] & & F^*QH^{*}(X;\Lambda_{\mathbb{F}_p})[u^\pm]\langle \theta \rangle\dar["\mathcal{C}^{-1}"] \\
             QH^{*}(X;\Lambda_{\mathbb{F}_p})[u^\pm]\langle \theta \rangle \arrow[rr, "\psi_b"] &&  QH^{*}(X;\Lambda_{\mathbb{F}_p})[u^\pm]\langle \theta \rangle
        \end{tikzcd}.
    \end{center}
\end{thm}

The proof of this statement involves invoking the cluster decomposition of the $p$-curvature, which is introduced in the next section; we postpone the proof to \cref{ssec:katz-formula-proof}.

\subsection{Mod $p$-pseudocycles}\label{subsec:p-pseudocycles}
The $p$-curvature of the mod $p$ quantum connection can be given a geometric interpretation in terms of quantum power operations, as we will show in the next section. The definition of quantum power operations on mod $p$ quantum cohomology involves intersection theory of $p$-pseudocycles, which geometrically represent mod $p$ homology classes. The purpose of this subsection is to recall the notion of $p$-pseudocycles, defined in \cite[Appendix A]{Wil-sur}. We closely follow the exposition in \cite[Appendix A]{bai-pomerleano-xu}, and refer the reader there for details.

These $p$-pseudocycles are variants of the more familiar notion of (integral) pseudocycles, which geometrically represent classes in $H^*(X;\mathbb{Z})$ for a smooth oriented manifold $X$. The $p$-pseudocycles allows one to similarly represent classes in $H^*(X;\mathbb{F}_p)$ which do not necessarily arise from $H^*(X;\mathbb{Z}) \otimes \mathbb{F}_p$. 
Let $X$ be a smooth manifold.
\begin{defn}
    A smooth map $f: W \to X$ from a $k$-dimensional oriented manifold with boundary $W$ is called an (oriented $k$-dimensional) \emph{$p$-pseudocycle} if it satisfies the following properties.
    \begin{enumerate}
        \item $f(W)$ is precompact in $X$.
        \item The dimension of the $\Omega$-set of $f$,
        \begin{equation}
            \Omega_f := \bigcap_{K \subseteq W \mathrm{compact}} \ov{f(W\setminus K)},
        \end{equation}
        is at most $k-2$.
        
        \item There is a smooth map $g: V \to X$ and an oriented $p$-fold covering $\partial W \to V$ such that $f|_{\partial W}$ is the pullback of $g$ by the covering.
    \end{enumerate}
\end{defn}
Properties (1) and (2) define the usual notion of an (integral) pseudocycle when $\partial W = 0$. Property (3) allows an existence of codimension 1 boundary as long as it is a $p$-fold cover, and it allows $p$-cycles to represent classes in $H^*(X;\mathbb{F}_p)$ with non-vanishing Bockstein (i.e. those chains that are cycles in $\mathbb{F}_p$ because their boundaries are $p$-fold multiples). It is considered vacuous if $\partial W = 0$ (as an oriented manifold).

\begin{defn}
    Two (oriented $k$-dimensional) $p$-pseudocycles $f_0: W_0 \to X$ and $f_1 : W_1 \to X$ are \emph{$(p-)$cobordant} if there exists a smooth map $\hat{f}: \hat{W} \to X$ from a $(k+1)$-dimensional oriented manifold with boundary $\hat{W}$ such that
    \begin{enumerate}
        \item $\hat{f}(\hat{W})$ is precompact in $X$.
        \item $\dim \Omega_{\hat{f}} \le k-1$.
        \item There is an oriented diffeomorphism $\partial \hat{W} \cong -\mathrm{Int} W_0 \sqcup \mathrm{Int} W_1 \sqcup W_h$ such that $\hat{f}|_{\mathrm{Int} W_i} = f_i |_{\mathrm{Int} W_i}$ and $\hat{f}|_{W_h}$ is a pullback of some map $g : V \to X$ by an oriented $p$-fold covering $W_h \to V$.
    \end{enumerate}
\end{defn}
The set of bordism classes of $p$-pseudocycles are denoted by $\mathfrak{H}_k^{(p)}(X)$. Using disjoint union as addition and orientation reversal as inverses, $\mathfrak{H}_*^{(p)}(X) = \bigoplus_k \mathfrak{H}_k^{(p)}(X)$ forms a graded abelian group. It is a $\mathbb{F}_p$-module since $p$ times any $p$-pseudocycle is cobordant to the empty set by the trivial cylinder. 

If we denote $\mathfrak{H}_*(X)$ the abelian group of (integral) pseudocycle bordisms, there is an abelian group map $\mathfrak{H}_*(X) \to \mathfrak{H}_*^{(p)}(X)$ by considering a pseudocycle as a $p$-pseudocycle.

Later we use $p$-pseudocycles to implement mod $p$ Gromov--Witten invariants with incidence constraints, hence we need intersection theory of $p$-pseudocycles:
 
\begin{prop}
    There is a well-defined intersection pairing
    \begin{equation}
        \cap : \mathfrak{H}_*^{(p)}(X) \otimes \mathfrak{H}_*^{(p)}(X) \to \mathbb{F}_p.
    \end{equation}
\end{prop}
The intersection pairing is the (signed) count of intersection points when $p$-pseudocycles of complementary dimensions meet transversally. The count is only well-defined modulo $p$ on bordism classes due to codimension $1$ boundary that are $p$-fold covers. For details, we refer the reader to \cite[Appendix A]{bai-pomerleano-xu}.

\begin{prop}[{\cite[Theorem A.3]{bai-pomerleano-xu}}]
    There is a natural map
    \begin{equation}
        \Psi_*^{(p)}: H_*(X;\mathbb{F}_p) \to \mathfrak{H}_*^{(p)}(X)
    \end{equation}
    satisfying the following conditions.
    \begin{enumerate}
        \item The following diagram commutes:
        \begin{center}
            \begin{tikzcd}
                H_*(X;\mathbb{Z}) \rar["\Psi_*"] \dar & \mathfrak{H}_*(X) \dar \\ H_*(X;\mathbb{F}_p) \rar["\Psi_*^{(p)}"] & \mathfrak{H}_*^{(p)}(X).
            \end{tikzcd}
        \end{center}
        \item If a homology class $A \in H_*(X;\mathbb{F}_p)$ is represented by a smooth cycle $f: W \to X$ that is a $p$-pseudocycle, then $\Psi_*^{(p)}(a)$ is represented by $f$.
        \item The map $\Psi_*^{(p)}$ intertwines the Poincar\'e intersection pairing on $H_*(X;\mathbb{F}_p)$ with the intersection pairing on $\mathfrak{H}_*^{(p)}(X)$.
    \end{enumerate}
\end{prop}

In particular, a $\mathbb{F}_p$-homology class naturally gives rise to a (bordism class of) $p$-pseudocycles. As before, we introduce the following notation:

\begin{defn}
    Let $X$ be a closed oriented manifold. Let $a \in H^*(X;\mathbb{F}_p)$, and denote its Poincar\'e dual homology class by $\mathrm{PD}[a] \in H_*(X;\mathbb{F}_p)$. Then
    \begin{equation}
        \Psi(a) := \Psi_*^{(p)}(\mathrm{PD}[a]) \in \mathfrak{H}_*^{(p)}(X)
    \end{equation}
    is the (bordism class of) a $p$-pseudocycle representing the cohomology class $a \in H^*(X;\mathbb{F}_p)$.
\end{defn}
By a slight abuse of notation, we may denote by $\Psi(a)$ some $p$-pseudocycle representing the bordism class $\Psi(a)$. Assuming that $H^*(X;\mathbb{Z})$ is $p$-torsion-free, then there is a natural diagram
\begin{center}
    
\begin{tikzcd}
    \bar{H}^*(X;\mathbb{Z}) \otimes \mathbb{F}_p \dar["\mathrm{PD} \otimes \mathbb{F}_p"]  \rar["\sim"] & H^*(X;\mathbb{F}_p) \dar["\mathrm{PD}"] \\ \bar{H}_*(X;\mathbb{Z}) \otimes \mathbb{F}_p \rar \dar & H_*(X;\mathbb{F}_p) \dar \\  \mathfrak{H}_*(X) \otimes \mathbb{F}_p \rar  & \mathfrak{H}_*^{(p)}(X)
\end{tikzcd}

\end{center}
so that the notation is consistent with the previous definition of $\Psi(a)$ for a class $a \in \bar{H}^*(X;\mathbb{Z})$.

\subsection{Quantum Steenrod operations}\label{subsec:qst}

In this subsection, we define the quantum Steenrod operations using $p$-pseudocycles to implement the incidence constraints. Due to technical difficulties with strata with $\mathbb{Z}/p$-isotropy in the moduli space of $J$-holomorphic maps, the moduli spaces are defined using inhomogeneous perturbations. Other expositions of the construction can be found in \cite{seidel-wilkins, Lee23a, Lee23b}.

\subsubsection{Model of equivariant cohomology}
Constructions of quantum Steenrod operations involve the rotational symmetry group $\mu_p$ of the source curve $\mathbb{P}^1$. The model of the classifying space $B\mu_p$ is as follows. We take $E\mu_p = S^\infty \subseteq \mathbb{C}^\infty$ as the unit sphere on which $\mu_p$ acts by simultaneous multiplication $\sigma$ of all complex coordinates.  We use the cell decomposition of $S^\infty$ \cite[Section 2a]{seidel-wilkins} with exactly $p$ many $i$-cells $\Delta^i, \sigma \Delta^i, \dots, \sigma^{p-1} \Delta^i$ for each dimension $i \ge 0$, forming a free $\mu_p$-orbit. The $i$-cell $\Delta^i$ has a natural compactification to $\ov{\Delta}_i$ as a manifold with stratified boundary, whose top boundary strata are given by 
\begin{align}
    \partial \ov{\Delta}_{i} = \begin{cases} \bigcup_{j=0}^{p-1} \sigma^j \Delta^{i-1} & i \mbox{ even} \\ \sigma \Delta_{i-1} - \Delta_{i-1} & i \mbox{ odd} \end{cases}.
\end{align}
The images of $\Delta^i$ under the quotient $E\mu_p \to B\mu_p = E\mu_p / \mu_p$ induces a CW-structure on $B\mu_p$ with unique $i$-cell for every $i \ge 0$. With $\mathbb{F}_p$-coefficients, the quotient images of $\ov{\Delta}^i$ become cycles and they additively generate $H_*(B\mu_p;\mathbb{F}_p)$.

The additive generators $\theta \in H^1(B\mu_p; \mathbb{F}_p)$ and $u \in H^2(B\mu_p;\mathbb{F}_p)$ from \cref{lem:group-cohomology-mup} correspond to those classes satisfying
\begin{equation}
    \langle \theta, \Delta^1 \rangle =1, \quad \langle u, \Delta^2 \rangle = -1.
\end{equation}

We denote the preferred additive generator of $H^*(B\mu_p;\mathbb{F}_p)$ in degree $i$ with the following notation:
\begin{equation}\label{eqn:eq-param}
    (u, \theta)^i = \begin{cases} u^{i/2} & i \mbox{ even } \\ u^{(i-1)/2}\theta & i \mbox{ odd } \end{cases}.
\end{equation}

\subsubsection{Parametrized moduli spaces}
For the target symplectic manifold $(X,\omega)$, fix a compatible almost complex structure $J$. An equivariant inhomogeneous term is a choice of a $J$-complex antilinear bundle homomorphism
\begin{equation}
    \nu^{eq} \in C^\infty ( E\mu_p  \times_{\mu_p} \mathbb{P}^1 \times X ; \Omega^{0,1}_{\mathbb{P}^1} \boxtimes TX).
\end{equation}
where $\mu_p$ acts on $\mathbb{P}^1$ by rotation action (denoted $\sigma_{\mathbb{P}^1}$). Equivalently, $\nu^{eq}$ can be considered as a collection of inhomogeneous terms $\nu^{eq}_w$ (in the usual sense) parametrized by $w \in E\mu_p$ satisfying the equivariance condition
\begin{equation}
    \nu_{ \sigma_{E} (  w), z, x} = \nu_{w, \sigma_{\mathbb{P}^1}(z), x} \circ D \sigma_z : T\mathbb{P}^1_z \to TX_x.
\end{equation}
Assume that $\nu^{eq}$ is supported away from the neighborhoods of distinguished marked points $z_0 = 0, z_j = \zeta^j, z_\infty = \infty$ on $\mathbb{P}^1$.

Denote by $\mathcal{M}^{eq}_\beta$ the moduli space of solutions to the parametrized problem
\begin{equation}
    \mathcal{M}^{eq}_\beta = \mathcal{M}_{0, p+2}^{eq}(X,J,\nu, \beta) := \left\{ (u:\mathbb{P}^1 \to X , w \in E\mu_p )  :  u_*[\mathbb{P}^1] =\beta, \ (\overline{\partial}_J u)_z = \nu^{eq}_{w, z, u(z)} \right\}.
\end{equation}

The moduli space admits the stable map compactification (fiberwise under the projection $\mathcal{M}^{eq}_A \to E\mu_p$), and carries an evaluation map
\begin{align}
    \mathrm{ev}^{eq}: \  &\ov{\mathcal{M}}^{eq}_\beta \to X \times X \times (E\mu_p \times X^p)  \\
    &(u,w) \mapsto (u(0) ; u(\infty);  (w, u(z_1), \dots, u(z_p)) )
\end{align}
which is $\mu_p$-equivariant with respect to the $\mu_p$-action is given by $(u,w) \mapsto (\sigma_E(w), u \circ \sigma_{\mathbb{P}^1})$ on $\mathcal{M}_\beta^{eq}$ and cyclic permutation on $E\mu_p \times X^p$. Provided the boundedness of the moduli space $\mathcal{M}^{eq}_\beta$, by Gromov compactness the evaluation map is covered by $E\mu_p$-parametrized versions of simple stable maps \cite[Chapter 6]{MS12}.

\subsubsection{Quantum Steenrod operators}\label{sssec:quantum-Steenrod-definition}
Fix $i \ge 0$ and denote by $\ov{\Delta}_i \subseteq E\mu_p$ the natural compactification of $\Delta_i$ as a manifold with boundary. By a slight abuse of notation, we also denote the inclusion of the quotient image of $\ov{\Delta}_i$ under $E\mu_p \to B \mu_p$ into $B\mu_p$ as $\ov{\Delta}_i \hookrightarrow B\mu_p$ (this is a closed manifold if $i$ is odd). Consider the moduli space obtained from the fiber product,
\begin{equation}
    \ov{\mathcal{M}}_\beta^{eq, i} := \begin{cases} \ov{\mathcal{M}}_\beta^{eq} \times_{E\mu_p} \ov{\Delta}^i & i \mbox{ even} \\ \ov{\mathcal{M}}_\beta^{eq}/\mu_p \times_{B\mu_p} \ov{\Delta}^i & i \mbox{ odd} \end{cases}
\end{equation}
which carries an evaluation map
\begin{align}
    \mathrm{ev}^{eq, i}: \  &\ov{\mathcal{M}}^{eq, i}_\beta \to X \times X \times X^p  \\
    &(u,w) \mapsto (u(0) ; u(\infty);  (u(z_1), \dots, u(z_p)) ).
\end{align}

\begin{lemma}\label[lemma]{lem:qst-pseudocycle}
    For a generic choice of $J$ and $\nu^{eq}$, the map $\mathrm{ev}^{eq, i}$ restricted to the open locus $\mathcal{M}^{eq, i}_\beta \to X \times X \times X^p$ defines a $p$-pseudocycle of dimension $i + \dim_{\mathbb{R}}X + 2c_1(\beta) = i + \dim_{\mathbb{R}} X $.
\end{lemma}
\begin{proof}
    For $i$ odd, the open locus fibers over the closed submanifold $\ov{\Delta}_i \subseteq B\mu_p$. The Omega-limit set of the evaluation map consists of the image of ($B\mu_p$-parametrized versions of) the simple stable maps, which is codimension at least $2$. Therefore, the evaluation map defines an integral pseudocycle. For $i$ even, $\ov{\Delta}_i$ is a manifold with boundary whose codimension $1$ boundary $\partial \ov{\Delta}_i = \mu_p \cdot \left( \bigcup_{k < i} \Delta^i \right)$ is given as a free $\mu_p$-orbit, and therefore admits a $p$-fold covering $\partial \ov{\Delta}_i \to \partial \ov{\Delta}_i /\mu_p$. The Omega-limit set of the evaluation map still consists of the image of ($E\mu_p$-parametrized versions of) simple stable maps, which is of codimension at least $2$.
\end{proof}

We denote the corresponding $p$-pseudocycle as $\mathrm{ev}_* [\ov{\mathcal{M}}^{eq, i}_\beta]$, which is independent of the choice of $J$ and $\nu^{eq}$ up to bordism by standard cobordism arguments.

\begin{defn}
    Let $a_1, a_2 , b\in H^*(X;\mathbb{F}_p)$. The  \emph{$i$-th quantum Steenrod correlator} is the number
    \begin{equation}
        \langle a_1, a_2, b^{\otimes p} \rangle_{\beta}^{eq, i} :=  \mathrm{ev}_*[\overline{\mathcal{M}}_\beta^{eq, i}] \cap \Psi(a_1 \otimes a_2 \otimes b^{\otimes p}) \in \mathbb{F}_p
    \end{equation}
    where $\cap$ denotes the intersection pairing of $p$-pseudocycles in $X \times X \times X^p$.
\end{defn}

The equivariant inhomogeneous term $\nu^{eq}$ can be ($C^\infty$-)generically chosen so that the intersection with the $p$-pseudocycle representatives of $\Psi(a_1 \otimes a_2 \otimes b^{\otimes p}) \in \mathfrak{H}_*^{(p)} (X \times X \times X^p)$ is transverse. Equivalently, the transversality requirement can be achieved by choosing generic $p$-pseudocycle representatives.

\begin{defn}
    Fix $b \in H^*(X;\mathbb{F}_p)$. The \emph{quantum Steenrod operator} $$Q\Sigma_b: QH^*(X;\Lambda_{\mathbb{F}_p})[u]\langle \theta \rangle \to QH^*(X;\Lambda_{\mathbb{F}_p})[u]\langle \theta \rangle$$ for $b$ is the $\Lambda[u]\langle \theta \rangle$-linear operator defined by the structure constants
    \begin{equation}
        (Q\Sigma_b (a_1), a_2) =  \sum_{i \ge 0} \sum_\beta (-1)^\dagger \langle a_1, a_2, b^{\otimes p} \rangle_\beta^{eq, i} q^\beta (u, \theta)^i \in \Lambda[u]\langle \theta \rangle
    \end{equation}
    for signs
    \begin{equation}
        \dagger = \begin{cases} |b||a_1| & i \mbox{ even} \\ |b||a_1| + |b| + |a_1| & i \mbox{ odd} \end{cases}
    \end{equation}
    and the Poincar\'e pairing $(\cdot, \cdot)$.
\end{defn}
\begin{rem}\label[rem]{rem:signs}
    The signs $(-1)^\dagger$ are largely invisible, since in this paper we will focus on equator insertions $b \in H^2(X)$ of even degree, and the series rarely has a nonzero term for $i$ odd, cf. \cref{rem:purely-quantum}.
\end{rem}

\begin{rem}\label[rem]{rem:CY-degree}

(Simplification in the Calabi--Yau setting). Fix a class $b \in H^2(X;\mathbb{F}_p)$. Note that
\begin{equation}
    Q\Sigma_b : QH^*(X;\Lambda_{\mathbb{F}_p})[u]\langle \theta \rangle \to QH^*(X;\Lambda_{\mathbb{F}_p})[u]\langle \theta \rangle
\end{equation}
is a degree $p|b| = 2p$ endomorphism of graded modules. Hence, the highest degree of the equivariant parameters $(u, \theta)^i$ that can appear in a structure constant $(Q\Sigma_b(a_1), a_2)\in \Lambda[u]\langle \theta \rangle $ is bounded above by
\begin{equation}
    p|b|+|a_1| + |a_2| - \dim_{\mathbb{R}}X - 2c_1(\beta) \le 2p + \dim_{\mathbb{R}} X.
\end{equation}
In particular, to define the quantum Steenrod operator $Q\Sigma_b$ it suffices to consider a finite-dimensional approximation of $E\mu_p = S^\infty$, for large $E\mu_p^{\mathrm{fin}} := S^{2N+1} \subseteq S^\infty$ and $\mu_p$-equivariant inhomogeneous term parametrized over $E\mu_p^{\mathrm{fin}}$. In particular, the inhomogeneous term $\nu^{eq}$ can be chosen to be parametrized over a \emph{compact} manifold.
\end{rem}

\section{$p$-curvature and power operations}

In this section, we prove for Calabi--Yau threefolds the equivalence between quantum Steenrod operaitons and the $p$-curvature. The two key steps involved are (i) the application of cluster formalism, which reduces the statement to the case of local $\mathbb{P}^1$ geometry, and (ii) a direct verification in the case of local $\mathbb{P}^1$. As a corollary, we can verify Katz' formula.

Unlike the results \cref{thm:integrality} and \cref{thm:Cartierconjugation}, \cref{thm:main} holds without the torsion-free assumption on $H^{*}(X;\mathbb{Z})$, and also holds if one defines quantum cohomology additively to be based on the whole cohomology, not only the even part:
$$QH^*(X;\Lambda_{\mathbb{F}_p}) [u^\pm] \langle \theta \rangle := H^*(X;\Lambda_{\mathbb{F}_p}) [u, u^{-1}]\langle \theta \rangle.$$ 

In particular, we present our proof in this chapter throughout using $p$-pseudocycles to geometrically represent classes in $H^*(X;\mathbb{F}_p)$ that generate the quantum cohomology.

\subsection{The local computation}\label{ssec:local-p1-computation}
Here, we revisit the computation in \cite[Section 4]{Lee23a} for the quantum Steenrod operation of $Y$ and verify that it is equivalent to our key local computation that the quantum Steenrod operation is equal to the $p$-curvature of the quantum connection. This is the local version of the main theorem in the case of the local $\mathbb{P}^1$; for our notations we refer back to \cref{ssec:local-p1-geometry}.

\begin{thm}[{\cite[Corollary 4.16]{Lee23a}}]\label{thm:localP1-qst-is-pcurv}
    Let $Y = \mathrm{Tot}(\mathcal{O}(-1) \oplus \mathcal{O}(-1) \to \mathbb{P}^1)$, and consider a divisor class $b \in H^2(Y;\mathbb{F}_p).$  Then
    \begin{equation}
        Q\Sigma_b = u^p \psi_b  \in \mathrm{End}(H^*(Y) \oplus H^*_c(Y)).
    \end{equation}
\end{thm}

We follow the notation from \cref{ssec:qconn-localP1}. Here, we consider the connection given by $\nabla_e =  q \partial_q + u^{-1} \mathbb{M}$ and its $p$-curvature
\begin{equation}
    \psi_e := \nabla_e^p - \nabla_b.
\end{equation}
Recall that $u$ is a degree $2$ variable and $\mathbb{M}$ is a degree $2$ operator so that $\nabla_b$ is a derivation of $H^*(Y;\Lambda)$ of degree $0$. Denote by $\mathbb{M}^{(j)} := (q\partial_q)^j \mathbb{M}$ the (logarithmic) derivatives of the matrix $\mathbb{M}$.

First we compute the leading term of the $p$-curvature, which is a rather general statement.
\begin{lemma}\label[lemma]{lem:pcurv-leading-term}
    The operator $u^p \psi_b$ is a linear endomorphism of degree $2p$, In its expansion as a polynomial in $u$, we have the leading term
    \begin{equation}
        u^p \psi_e = u^{p-1} (\mathbb{M}^{(p-1) } - \mathbb{M}) + \left( \mbox{terms of lower order in } u \right).
    \end{equation}
\end{lemma}
\begin{proof}
    The linearity of $p$-curvature (despite its expression in terms of $\nabla_e$ which is not linear in the Novikov variables) is standard, as can be found in e.g. \cite[Section 5]{Kat70}. From the definition $\psi_e:= \nabla_e^p - \nabla_e$ and the fact that $\nabla_e$ is of degree $0$, it is clear that the $u^p \psi_e $ is of degree $2p$, and hence can be expressed as a matrix-coefficient polynomial in $u$ of degree bounded by $p$:
    \begin{equation}
        u^p\psi_e = u^p C_0 + u^{p-1}C_1 + \left( \mbox{terms of lower order in } u \right).
    \end{equation}
    Again from definition, it is clear that $C_0 = 0$. It remains to show that $C_1 = \mathbb{M}^{(p-1)} - \mathbb{M}$. It suffices to show that the $u^{p-1}$-term in the expansion of $(u\nabla_b)^p = (u q\partial_q + \mathbb{M})^p$ is equal to $\mathbb{M}^{(p-1)}$. The $u^{p-1}$-term in the expansion is in turn given by
    \begin{equation}
       u^{p-1}  \sum_{k=1}^{p} (q\partial_q)^{p-k} \mathbb{M} (q\partial_q)^{k-1}.
    \end{equation}
    The Leibniz rule $[q\partial_q, \mathbb{M}^{(k)}] = \mathbb{M}^{(k+1)}$ implies that each term in the sum can be written as
    \begin{equation}
         (q\partial_q)^{p-k} \mathbb{M} (q\partial_q)^{k-1} = \left( \sum_{i=0}^{p-k} \binom{p-k}{i} \mathbb{M}^{(i)} (q\partial_q)^{p-k-i} \right)  (q\partial_q)^{k-1} = \sum_{i=0}^{p-k}\binom{p-k}{i} \mathbb{M}^{(i)} (q\partial_q)^{p-1-i}.
    \end{equation}
    Now consider
    \begin{align}
   C_1 +\mathbb{M}  &= \sum_{k=1}^{p} (q\partial_q)^{p-k} \mathbb{M} (q\partial_q)^{k-1} \\
        &= \sum_{k=1}^{p}   \sum_{i=0}^{p-k} \binom{p-k}{i} \mathbb{M}^{(i)} (q\partial_q)^{p-1-i} 
        \\
&= \sum_{i=0}^{p-1} \sum_{k=1}^{p-i}  \binom{p-k}{i} \mathbb{M}^{(i)} (q\partial_q)^{p-1-i} \\
&= \sum_{i=0}^{p-1} \binom{p}{i+1} \mathbb{M}^{(i)} (q\partial_q)^{p-1-i} \\
&= \mathbb{M}^{(p-1)}.
    \end{align}
    This is the desired result. The fourth equality follows from $\sum_{k=1}^{p-i} \binom{p-k}{i} = \binom{p}{i+1}$ (Pascal's triangle).
\end{proof}

We may now give the proof of \cref{thm:localP1-qst-is-pcurv}, which is a re-interpretation of \cite[Corollary 4.16]{Lee23a}.

\begin{proof}[Proof of \cref{thm:localP1-qst-is-pcurv}]
First consider the case where $e =[\mathrm{fiber}] \in H^2(Y)$ is the class Poincar\'e dual to the fiber of $Y \to \mathbb{P}^1$ (the generator of $H^2(Y;\mathbb{P}^1)$). We consider the quantum connection
\begin{equation}
    \nabla_e :=  q \partial_q + u^{-1} \mathbb{M} =  q\partial_q + u^{-1} \begin{pmatrix} 0 & 0 & 0 & 0 \\ 1 & 0 & 0 & 0 \\ 0 & \frac{q}{1-q} & 0 & 0 \\ 0 & 0& 1& 0 \end{pmatrix}
\end{equation}
in the direction of the fiber class $e \in H^2(Y)$, and the corresponding $p$-curvature
\begin{equation}
    \psi_b:= \nabla_b^p - \nabla_b.
\end{equation}
By \cref{lem:pcurv-leading-term}, the leading term of the $p$-curvature is given by
\begin{equation}
    u^p\psi_e = \sum_{k=1}^{p} u^{p-k} C_k = u^{p-1} \left( \mathbb{M}^{(p-1)} - \mathbb{M} \right) + u^{p-2} C_2 + u^{p-3} C_3 + \cdots.
\end{equation}
From \cref{cor:localP1-qconn-nilp}, we know that $C_j = 0$ for all $j \ge 4$. From covariant constancy $[\nabla_e, \psi_e] = 0$, it follows that
\begin{equation}
    (u q\partial_q ) u^p \psi_e = u^p C_1^{(1)} + u^{p-1} C_2^{(1)} + u^{p-2} C_3^{(1)}
\end{equation}
is equal to
\begin{equation}
    [u^p\psi_e, \mathbb{M}] = u^{p-1} [C_1, \mathbb{M}] + u^{p-2}[C_2, \mathbb{M}] + u^{p-3}[C_3, \mathbb{M}].
\end{equation}
This yields a system of equations which determine $C_2$ and $C_3$ by comparing the coefficients of the $u$-terms. The end result of the computation is that

\begin{equation}\label{eqn:localP1-pcurv-formula}
    u^p\psi_e = \begin{pmatrix}
        0 & 0 & 0 & 0 \\
        - u^{p-1} & 0 & 0 & 0 \\
        u^{p-2}\sum_{d=1}^\infty d^{p-2}q^d & -u^{p-1}\sum_{d=1}^\infty  q^{pd} & 0 & 0  \\ -u^{p-3} \sum_{d=1}^\infty 2 d^{p-3} q^d & - u^{p-2}\sum_{d=1}^\infty d^{p-2}q^d & -u^{p-1} & 0 
    \end{pmatrix}
\end{equation}
Here, to simplify $\mathbb{M}^{(p-1)}- \mathbb{M}$, we used the fact that $d^{p-1} \equiv 1$ mod $p$ for all $d \not \equiv 0$. Moreover, note that all entries are in given in terms of iterated logarithmic derivatives of the function $f(q) = \frac{q}{1-q}$. In particular, they are all rational functions.

To conclude the proof, we simply remark that this is exactly the same result as the computation of $Q\Sigma_e$ from \cite[Corollary 4.16]{Lee23a} (in \emph{loc. cit.} the convention uses $-u$ in place of our $u$, leading to the difference in signs for $u^{p-2}$ term). 

For more general classes of the form $b = k[F] \in H^2(Y;\mathbb{F}_p)$ for $k \in \mathbb{Z}$, the claim follows from the special case above because both operators $Q\Sigma_b$ and $\psi_b$ are Frobenius-linear in the entry $b$.
\end{proof}

\begin{rem}
    The relation between the explicit formal power series in $q$ obtained as structure constants of the quantum Steenrod operation with the series $\eta(q)$ in the proof is as follows:
    \begin{align*}
        \sum_{d=1}^{\infty} d^{p-2}q^d &\equiv \sum_{p \nmid d} \frac{q^d}{d} = \mathrm{Li}_1 (q) - \frac{1}{p} \mathrm{Li}_1(q^p), \\
        \sum_{d=1}^\infty d^{p-3}q^d &\equiv \sum_{p \nmid d} \frac{q^d}{d^2} = \mathrm{Li}_2 (q) - \frac{1}{p^2} \mathrm{Li}_2(q^p).
    \end{align*}
    
\end{rem}

\begin{rem}\label{rem:enumerative-interpretation-qst}
    The computation shows that there are terms in $Q\Sigma_b$ counting Gromov--Witten invariants in degree $p\cdot d[\mathbb{P}^1] \in H_2(Y;\mathbb{Z})$; they arise in the $u^{p-1}$-term of $Q\Sigma_b$ in the form $f(q^p) = \frac{q^p}{1-q^p}$. The above proof that $Q\Sigma_b = \psi_b$ ``explains'' the source of these terms, as they arise as the leading term in the $p$-curvature $\psi_b = u^{p-1} (\mathbb{M}^{(p-1)} - \mathbb{M}) + \cdots$. In particular, these curve counts are exactly identified with the \emph{$p$-fold multiple covers} of the underlying to the simple curve $\mathbb{P}^1 \to \mathbb{P}^1 \subseteq Y$ mapping to the zero section of $Y$. Hence, we arrive at the enumerative interpretation that, for Calabi--Yau 3-folds, the ($q^{pd}$-terms of) quantum Steenrod operations count the multiple covers of the embedded $J$-holomorphic curves. It could be interesting to compare this result with the FOP counts (cf. \cite{Bai-Xu-FOP1}) from the strata with $\mathbb{Z}/p$-isotropy.
\end{rem}

\subsection{Invoking cluster formalism}\label{ssec:clusters}

In this section, we prove \cref{thm:main} by invoking the cluster formalism of \cite{IP-GV} when $X$ is a Calabi--Yau 3-fold of real dimension $6$. The key steps can be summarized as follows.

\begin{itemize}
    \item Fix an energy bound $E$. There is a \emph{cluster decomposition} of the moduli space of genus $0$ $J$-holomorphic curves in $X$ into a finite collection of open and closed subsets.
    \item The structure constants of the quantum Steenrod and the $p$-curvature operators inherit a cluster decomposition. We address the technical details arising when working with inhomogeneous terms and incidence constraints for marked points on the genus $0$ curves.
    \item For \emph{elementary clusters}, the structure constants agree by the computation in the local $\mathbb{P}^1$-case.
    \item By the cluster isotopy theorem, adapted to our situation, the theorem is reduced to the case of elementary clusters \emph{up to} error in high energy. We conclude by inducting on $E$.
\end{itemize}

For the discussions prior the proof of Theorem \ref{thm:main}, we will be focused on the computation and structural properties of the enumerative invariants from \emph{effective} curves (the counts of non-constant $J$-holomorphic maps), that is the purely quantum parts of the operators. The classical terms, that is the contribution of constant maps, are considered separately in the proof of Theorem \ref{thm:main}.

\subsection{Cluster decomposition of the moduli space}
We first review the cluster formalism to describe the structure of the moduli space of genus $0$ $J$-holomorphic curves in $X$. 

\begin{defn} Fix a real number $E > 0$. The space of \emph{$E$-admissible almost complex structures} $\mathcal{J}(X)^E$ is defined to be the space (equipped with $C^\infty$-topology) of $\omega$-compatible almost complex structures on $X$ such that
\begin{enumerate}
    \item All genus simple genus $0$ $J$-holomorphic maps $u: \mathbb{P}^1 \to X$ with $E(u) := \omega[u] \leq E$ are embedded and transversely cut out.
    \item Consequently, either an embedded $J$-curve is isolated, or a multiple cover of it arises as the Gromov limit of a sequence of embedded curves with higher energy.
    \item For any two simple genus $0$ $J$-holomorphic maps $u_1, u_2: \mathbb{P}^1 \to X$ with $E(u_1), E(u_2) \leq E$ such that $u_1$ is not of the form $u_2 \circ \phi$ for some reparametrization $\phi \in \mathrm{Aut}(\mathbb{P}^1)$, their images satisfy $u_1(\mathbb{P}^1) \cap u_2(\mathbb{P}^1) = \emptyset$.
\end{enumerate}
\end{defn}
The space $\mathcal{J}(X)^E$ is dense in the space of all $\omega$-compatible almost complex structures $\mathcal{J} := \mathcal{J}(X,\omega)$, see \cite[Corollary 1.4]{IP-GV}.

Fix $J \in \mathcal{J}(X)^E$, and consider the moduli space of $J$-holomorphic maps from genus $0$ curves $\mathcal{M}(X):= \mathcal{M}_0(X, J)$, and its Gromov compactification $\ov{\mathcal{M}}(X)$. We can consider the ``underlying curve'' map
\begin{equation}
    c : \overline{\mathcal{M}}(X) \to \mathcal{K}
\end{equation}
where $\mathcal{K} := \mathcal{K}(X)$ is the space of all non-empty compact subsets of $X$. After fixing a background metric on $X$ and an induced distance function $d$, the set $\mathcal{K}(X)$ is given the Hausdorff metric topology defined by
\begin{equation}
    d_H(A, B) = \sup_{a \in A} \inf_{b \in B} d(a,b) + \sup_{b \in B} \inf_{a \in A} d(a, b), \quad A, B \in \mathcal{K}(X).
\end{equation}
The map $c$ is continuous, and proper when restricted to the space $\ov{\mathcal{M}}(X)^E$ of maps of energy $\le E$. The proofs can be found in \cite[Section 1]{IP-GV} or \cite[Theorem 2.11]{doan2021gopakumar}. Note that moduli space of maps from genus $0$ curves with marked points can also be equipped with the map $c$ by pre-composing it with the continuous forgetful maps $\ov{\mathcal{M}}_{0,n}(X, J) \to \ov{\mathcal{M}}_{0}(X,J)$. 

\begin{defn}
    The image under $c$ of embeddings in $\overline{\mathcal{M}}(X), \ov{\mathcal{M}}(X)^E$ are denoted $\mathcal{C}(X), \mathcal{C}(X)^E \subseteq \mathcal{K}(X)$, respectively. An element of $\mathcal{C}(X)$ is called an \emph{embedded $J$-holomorphic curve} in $X$. 
\end{defn}

Note that for $J \in \mathcal{J}(X)^E$, all simple maps are embedded, and the image of any map in $\ov{\mathcal{M}}(X, J)^E$ under $c$ factors through $\mathcal{C}(X)^E$.

Denote by $\mathcal{C}(X)^E_m$ the space of embedded $J$-holomorphic curves of genus $0$ with energy $\leq E$ such that its underlying homology class is $k$ times a primitive homology class, where $k \leq m$. Then we have a filtration
\begin{equation}
    \mathcal{C}(X)^E_1 \subset
    \mathcal{C}(X)^E_2 \subset \cdots \subset \mathcal{C}(X)^E_m \subset \cdots
\end{equation}
which terminates for sufficiently large $m$ by energy bound: $\mathcal{C}(X)^E = \mathcal{C}(X)^E_m$. Another observation is that the complement of an open neighborhood of $\mathcal{C}(X)^E_{m-1}$ in $\mathcal{C}(X)^E_m$ consists of only finitely many embedded $J$-holomorphic curves, see \cite[Lemma 1.6]{IP-GV}.

We recall the definition of clusters following \cite{IP-GV, doan2021gopakumar}. We focus on the case of genus $0$ $J$-holomorphic curves.
\begin{defn}\label[defn]{defn:clusters}
    For $E > 0$, an \emph{$E$-cluster} for $X$ is a triple $(\mathcal{U}, J, C)$ such that:
    \begin{enumerate}
        \item $J$ is an $\omega$-compatible almost complex structure and $C \subset X$ is a genus $0$ embedded $J$-holomorphic curve, where we identify $C$ with the image of a $J$-holomorphic map $u: S^2 \to X$.
        \item $\mathcal{U}$ is an open subset of the space of compact subsets $\mathcal{K}$ of $X$, with topology induced from the Hausdorff distance under a (equivalently, any) Riemannian metric.
        \item $C \in \mathcal{U}$, and it is the only embedded $J$-holomorphic genus $0$ curve in $\mathcal{U}$ of homology class $\beta = [C]$; moreover, any genus $0$ stable $J$-holomorphic map $\varphi: \Sigma \to X$ such that $\mathrm{im}(\varphi) \in \mathcal{U}$ satisfies $\varphi_*([\Sigma]) = k \cdot \beta$ for some $k \in \mathbb{Z}_{\geq 1}$. The curve $C$ is called the \emph{core} of the cluster.
        \item Any embedded genus $0$ $J$-holomorphic curve $C'  \in \partial \mathcal{U}$ must have energy $> E$.
    \end{enumerate}
\end{defn}

Primary examples of $E$-clusters come from the set of compact subsets contained in an $\epsilon$-tubular neighborhoods of embedded genus $0$ $J$-holomorphic curves $C$ for $J \in \mathcal{J}(X)^E$ for a generic $\epsilon$. They are used to decompose generating series of Gromov--Witten type invariants in to local pieces, in the sense of prescribing the images of the corresponding $J$-holomorphic maps. To be more precise, for instance, denote by $\ov{\mathcal{M}}_{0,k}(X;J)^E$ the moduli space of genus $0$ $J$-holomorphic stable maps with $k$ fixed marked points to $X$ with energy at most $E$. Then for $J \in \mathcal{J}(X)^E$, we have a continuous map
\begin{equation}
    \ov{\mathcal{M}}_{0,k}(X,J)^E \to \mathcal{C}(X)^E \to \mathcal{K}.
\end{equation}

Then we can cover the image of $\mathcal{C}(X)^E$ in $\mathcal{K}$ by finitely many $E$-clusters $\{ (\mathcal{U}_\ell, J, C_\ell) \}$, such that the preimage of $\mathcal{U}_\ell$ under the above map gives rise to an open and closed subset
\begin{equation}
    \ov{\mathcal{M}}_{0,k}(X,J)^E_{\mathcal{U}_\ell}
\end{equation}
of $\ov{\mathcal{M}}_{0,k}(X,J)^E$ under the Gromov topology. In \cite{IP-GV} and \cite{doan2021gopakumar}, one crucially uses the fact that the virtual fundamental classes of the moduli spaces of $J$-holomorphic maps satisfy the decomposition
\begin{equation}
    [\ov{\mathcal{M}}_{0,k}(X,J)^E]_{\mathrm{vir}} = \sum_\ell [\ov{\mathcal{M}}_{0,k}(X,J)^E_{\mathcal{U}_\ell}]_{\mathrm{vir}},
\end{equation}
giving rise to a corresponding decomposition of the generating series of Gromov--Witten invariants in terms of clusters. We summarize the discussion above as

\begin{prop}[{\cite[Lemma 1.6, Lemma 2.3, Proposition 2.4]{IP-GV}}]
    Given $E>0$ and $J \in \mathcal{J}(X)^E$, there exists a cluster decomposition
    \begin{equation}
        \mathcal{C}(X)^E \subseteq \bigcup_{\ell} \mathcal{U}_\ell
    \end{equation}
    into finitely many $E$-clusters $\{(\mathcal{U}_\ell, J, C_\ell)\}$ so that each $\mathcal{C}(X)^E_{\mathcal{U}_\ell} := \mathcal{C}(X)^E \cap \mathcal{U}_\ell$ is an open and closed subset of $\mathcal{C}(X)^E$. Moreover, the subsets $\mathcal{U}_\ell$ are disjoint for different $\ell$, and $\mathcal{C}(X)^E \cap \partial \mathcal{U}_\ell = \emptyset$ for all $\ell$.
\end{prop}

The clusters $\mathcal{U}_\ell$ can be chosen to be $\epsilon$-neighborhoods in the Hausdorff topology of $\mathcal{K}$ of the core curve $C_\ell$ for generic $\epsilon$ (\cite[Lemma 2.3]{IP-GV}, \cite[Proposition 3.17]{doan2021gopakumar}), and we will only use the clusters constructed in this way throughout.

\begin{cor}
    Given a finite cluster decomposition of the embedded $J$-holomorphic curves $\mathcal{C}(X)^E = \bigcup_{\ell} \mathcal{C}(X)^E_{\mathcal{U}_\ell}$, the preimages
    \begin{equation}
        \ov{\mathcal{M}}_{0,k}(X,J)_{\mathcal{U}_\ell}^E := c^{-1} \left( \mathcal{C}(X)^E_{\mathcal{U}_\ell} \right) \subseteq \ov{\mathcal{M}}_{0,k}(X,J)^E
    \end{equation}
    form a finite decomposition of the moduli space of $J$-holomorphic maps of energy $\le E$ into disjoint open and closed subsets,
    \begin{equation}
        \ov{\mathcal{M}}_{0,k}(X,J)^E = \bigcup_{\ell} \ov{\mathcal{M}}_{0,k}(X,J)_{\mathcal{U}_\ell}^E.
    \end{equation}
\end{cor}

The key claim \cite[Lemma 2.1]{IP-GV} is that this induced decomposition is in fact locally constant in $J$, which uses the fact that there is no embedded genus $0$ $J$-holomorphic curve on the boundary of the cluster $\mathcal{U}$, property (4) of \cref{defn:clusters}.

In this paper, as we would like to only use classical approaches to define symplectic enumerative invariants as in \cite{MS12}, we need a variant of the cluster decomposition incorporating inhomogeneous perturbations and other transversality requirements.

\subsubsection{Inhomogeneous perturbations}
We consider inhomogeneous perturbations $\nu \in \Gamma(\mathbb{P}^1 \times X, \Omega^{0,1}_{\mathbb{P}^1} \otimes TX)$, from which we define a $(J, \nu)$-holomorphic map from $\mathbb{P}^1$ to $X$ to be a smooth map $u: \mathbb{P}^1 \to X$ such that
\begin{equation}
    \ov{\partial}_{J} u(z) = \frac12 (du + J \circ du \circ j)(z) = \nu(z, u(z)), \quad \quad \quad \forall z \in \mathbb{P}^1,
\end{equation}
where $j$ is the canonical integrable complex structure of $\mathbb{P}^1$.

\begin{lemma}\label[lemma]{lemma:compactness}
    Fix $J \in \mathcal{J}(X)^E$ and $k \in \mathbb{Z}_{\geq 1}$. For any $\epsilon > 0$, there exists $\delta > 0$ such that for any $\| \nu \|_{C^k} \leq \delta$, any $(J, \nu)$-holomorphic sphere $u$ with energy at most $E$ must have image contained in an $\epsilon$-neighborhood of a simple $J$-holomorphic curve of energy at most $E$ under the Hausdorff topology.
\end{lemma}

\begin{proof}
    This is Gromov compactness: consider a $(J, \nu)$-holomorphic map $u: \mathbb{P}^1 \to X$ as a $J_\nu$-holomorphic section of $\mathbb{P}^1 \times X \to \mathbb{P}^1$ for a $\nu$-dependent almost complex structure
    \begin{equation}
        J_\nu = \begin{pmatrix}
            j & 0 \\ \nu \circ j - J \circ \nu & J
        \end{pmatrix}
    \end{equation}
    on $\mathbb{P}^1 \times X$ (Gromov trick), and note that as $\nu \to 0$ the $J_\nu$-holomorphic sphere converges to a stable map into $X$ whose image curve agrees with the image of the underlying simple map. Then, the statement is deduced from a standard argument-by-contradiction.
\end{proof}

For sufficiently small $\nu$, this allows us to consider a cluster decomposition of moduli spaces of $(J,\nu)$-holomorphic maps. Let
\begin{equation}
    \mathcal{M}_{0, k}(X, J, \nu) := \left\{ u : \mathbb{P}^1 \to X: (\ov{\partial}_J u)_z = \nu_{z, u(z)} \right\}
\end{equation}
be the moduli spaces of solutions to the problem with inhomogeneous perturbation where $\mathbb{P}^1$ carries $k$ fixed marked points and denote by $\ov{\mathcal{M}}_{0, k}(X, J, \nu)$ the stable map compactification.

\begin{lemma}\label[lemma]{lemma:nu-cluster-decomposition}
    Fix $J \in \mathcal{J}(X)^E$. Given a finite cluster decomposition of the embedded $J$-holomorphic curves $\mathcal{C}(X)^E = \bigcup_{\ell} \mathcal{C}(X)^E_{\mathcal{U}_\ell}$, for sufficiently $C^k$-small $\nu$ the preimages
    \begin{equation}
        \ov{\mathcal{M}}_{0,k}(X,J, \nu)_{\mathcal{U}_\ell}^E := c^{-1} \left( \mathcal{U}_\ell \right) \subseteq \ov{\mathcal{M}}_{0,k}(X,J, \nu)^E
    \end{equation}
    form a finite decomposition of the moduli space of $(J,\nu)$-holomorphic maps of energy $\le E$ into disjoint open and closed subsets.
\end{lemma}
\begin{proof}
    Since $\mathcal{U}_\ell$ is open in $\mathcal{K}$ and $c$ is continuous, the preimage $\ov{\mathcal{M}}_{0,k}(X,J, \nu)_{\mathcal{U}_\ell}^E = c^{-1}(\mathcal{U}_\ell)$ is open. It remains to show that $\ov{\mathcal{M}}_{0,k}(X,J, \nu)_{\mathcal{U}_\ell}^E $ is also closed. Note that in the case of genuine $J$-holomorphic maps, we used the fact that $\mathcal{C}(X)^E_{\mathcal{U}_\ell}$ is a closed subset of $\mathcal{C}(X)^E$ (that is, there is no embedded $J$-holomorphic curve of energy $\leq E$ on the boundary $\partial \mathcal{U}_\ell$). 

    By property (4) in the definition of the cluster (\cref{defn:clusters}) for each cluster $\mathcal{U}_\ell$ there exists a strictly positive $\epsilon_\ell = \inf_{C \in \mathcal{C}(X)^{E}_{\mathcal{U}_\ell}} d (C, \partial \mathcal{U}_\ell) > 0$. Choose $\epsilon>0$ such that $\epsilon < \frac{1}{2} \min_\ell \epsilon_\ell$, using that the cluster decomposition $\mathcal{C}(X)^E = \bigcup_\ell \mathcal{C}(X)^E_{\mathcal{U}_\ell}$ is finite. By \cref{lemma:compactness}, choose a small $\nu$ such that $c(\ov{\mathcal{M}}_{0,k}(X,J, \nu)^E)$ is contained in an $\epsilon$-neighborhood of $\mathcal{C}(X)^E$. If $u \in \ov{\mathcal{M}}_{0,k}(X,J, \nu)^E$, then there exists $C \in \mathcal{C}(X)^E_{\mathcal{U}_\ell}$ such that $d(C,c(u)) < \epsilon$. Then $ d(c(u), \partial \mathcal{U}_\ell) \ge d(C, \partial \mathcal{U}_\ell) - d(C, c(u)) \ge \epsilon_\ell - \epsilon > 0$. Hence, $c^{-1}(\ov{\mathcal{U}_\ell}) = c^{-1}(\mathcal{U}_\ell)$. Since $\ov{\mathcal{U}_\ell}$ is closed we conclude that $c^{-1}(\mathcal{U}_\ell)$ is also closed as desired.
\end{proof}

Note that the theorem also holds in a parametrized form for a family of inhomogeneous terms $\nu_w$ for $w \in S$ for compact $S$ (with the same proof), that is given a uniform bound on $\sup_{w \in S} \|\nu_w\|_{C^k}$ there is an induced cluster decomposition of the parametrized moduli spaces of pairs $(u, w)$ such that $\ov{\partial}_J u = \nu_w$.

\subsubsection{Incidence constraints}

Recall that our moduli spaces $\ov{\mathcal{M}}_{0,k}(X,J, \nu)^E$ consist of maps from a domain with a distinguished component (the parametrized copy of $\mathbb{P}^1$), which in the open locus $\mathcal{M}_{0,k}(X,J, \nu)^E$ carries $k$ \emph{fixed} marked points. We may extend the cluster decomposition to the moduli spaces where we impose incidence constraints on these marked points by pesudocycles.

Fix cohomology classes $a_1, \dots, a_k \in H^*(X;\mathbb{F}_p)$ and consider their $p$-pseudocycle representatives as $\Psi(a_i) \in \mathfrak{H}_*^{(p)}(X)$, cf. \cref{ssec:pseudocycles}. Let us denote the fixed representatives by maps $\Psi(a_i): A_i \to X$.

Consider the evaluation map
\begin{equation}
    \mathrm{ev} : \ov{\mathcal{M}}_{0,k}(X,J, \nu)^E \to X^k
\end{equation}
and its fiber product with the pseudocycle maps
\begin{center}
    \begin{tikzcd}
        & A_1 \times \cdots \times A_k \dar["\prod_i \Psi(a_i)"] \\ \ov{\mathcal{M}}_{0,k}(X,J, \nu)^E \rar["\mathrm{ev}"] &  X^k.
    \end{tikzcd}
\end{center}
We denote by $\ov{\mathcal{M}}_{0,k}(X,J, \nu)^E_{a_1, \dots, a_k}$ this fiber product. The following observation is simple but it is essential to define ($E$-truncated) correlators for quantum operations from the moduli space $\ov{\mathcal{M}}_{0,k}(X,J, \nu)^E$.

\begin{lemma}\label[lemma]{lemma:ic-transversality}
    For $J \in \mathcal{J}(X)^E$, consider a finite cluster decomposition $\mathcal{C}(X)^E = \bigcup_\ell \mathcal{C}(X)^E_{\mathcal{U}_\ell}$. Denote the core curve of the cluster $\mathcal{C}(X)^E_{\mathcal{U}_\ell}$ by $C_\ell \in \mathcal{K}(X)$. Suppose $\sum_{i=1}^k |a_i|  = \dim_{\mathbb{R}} \ov{\mathcal{M}}_{0,k}(X,J, \nu)^E$. Then the $p$-pseudocycle representatives $\Psi(a_i)$ and the inhomogeneous term $\nu$ can be chosen generically so that \begin{enumerate}
        \item Each $\Psi(a_i)$ intersects all $C_\ell$ transversely,
        \item $\mathrm{ev}$ restricted to the open locus $\mathcal{M}_{0,k}(X,J, \nu)^E$ defines a (integral) pseudocycle, which intersects $\prod_i \Psi(a_i)$ transversely (at finitely many points).
        \item $\prod_i \Psi(a_i)$ intersects with the image of $\ov{\mathcal{M}}_{0,k}(X,J, \nu)^E$ along $\mathrm{ev}$ only along the open locus $\mathcal{M}_{0,k}(X,J, \nu)^E$ and in the interior of $A_1 \times \cdots \times A_k$.
    \end{enumerate}
    Moreover, the signed count of $\ov{\mathcal{M}}_{0,k}(X,J, \nu)^E_{a_1, \dots, a_k}$ is well-defined (independent of $J \in \mathcal{J}(X)^E$, $\nu$, and choices of representatives of $\Psi(a_i)$) modulo $p$.
\end{lemma}
\begin{proof}
    For $J \in \mathcal{J}(X)^E$, the core curves $C_\ell$ are embedded submanifolds of $X$, hence (1) can be arranged. (2) and (3) are standard (transversality of evaluation maps, \cite[Chapter 6]{MS12}). The signed count of the fiber product $\ov{\mathcal{M}}_{0,k}(X,J, \nu)^E \cap (A_1 \times \cdots \times A_k)$, assuming transversality conditions (2) and (3), agree with the intersection pairing of $p$-pseudocycles applied to the (mod $p$-reduction of) the evaluation pseudocycle and the incidence $p$-pseudocycles $\Psi(a_i)$, hence is well-defined in $\mathbb{F}_p$.
\end{proof}

The cluster decomposition also holds for these moduli spaces of maps satisfying the incidence constraints:

\begin{lemma}\label[lemma]{lemma:ic-cluster-decomposition}
    Given a finite cluster decomposition of the embedded $J$-holomorphic curves $\mathcal{C}(X)^E = \bigcup_{\ell} \mathcal{C}(X)^E_{\mathcal{U}_\ell}$, for sufficiently $C^k$-small $\nu$ the preimages
    \begin{equation}
        \left(\ov{\mathcal{M}}_{0,k}(X,J, \nu)_{a_1, \dots, a_k}^E \right)_{\mathcal{U}_\ell} := c^{-1} \left( \mathcal{U}_\ell \right) \subseteq \ov{\mathcal{M}}_{0,k}(X,J, \nu)^E_{a_1, \dots, a_k}
    \end{equation}
    form a finite decomposition of the moduli space of $(J,\nu)$-holomorphic maps of energy $\le E$ satisfying incidence constraints into disjoint open and closed subsets.
\end{lemma}
\begin{proof}
    By \cref{lemma:nu-cluster-decomposition}, the cluster decomposition holds for $\ov{\mathcal{M}}_{0,k}(X,J, \nu)^E  = \bigcup_{\ell} \ov{\mathcal{M}}_{0,k}(X,J, \nu)^E _{\mathcal{U}_\ell}$ for sufficiently small $\nu$. The map $\ov{\mathcal{M}}_{0,k}(X,J, \nu)^E_{a_1, \dots, a_k} \to \ov{\mathcal{M}}_{0,k}(X,J, \nu)^E $ from the fiber product is continuous, so the cluster decomposition also holds for the moduli space satisfying incidence constraints.
\end{proof}

\subsection{Cluster decomposition of the structure constants}
Using the cluster formalism, here we show that the structure constants of the operators $Q\Sigma_b$ and $\psi_b$, that is the matrix coefficients of these operators after choosing a fixed basis of cohomology classes, inherit a cluster decomposition from the cluster decomposition of $J$-holomorphic curves.

\subsubsection{$p$-curvature}

Consider the $p$-curvature operator
\begin{equation}
    u^p \psi_b : QH^*(X;\Lambda_{\mathbb{F}_p})[u]\langle \theta \rangle \to QH^*(X;\Lambda_{\mathbb{F}_p})[u]\langle \theta \rangle.
\end{equation}

For two fixed cohomology classes $a_1, a_2 \in QH^*(X;\Lambda_{\mathbb{F}_p})[u]\langle \theta \rangle$, we denote the structure constants of the $p$-curvature operator by
\begin{equation}
  \left( u^p \psi_b(a_1), a_2 \right) =:  \left( u^p \psi_b|_{q^\beta = 0} (a_1), a_2 \right) + \sum_{i \ge 0} c_i \langle  u^p  \psi_b; a_1, a_2 \rangle (u, \theta)^i \in \Lambda[u]\langle \theta \rangle
\end{equation}
where $(\cdot, \cdot)$ is the Poincar\'e pairing (extended to $QH^*(X;\Lambda_{\mathbb{F}_p})[u]\langle \theta \rangle$ linearly on $q^\beta$ and $u, \theta$) and $c_i \langle  u^p \psi_b; a_1, a_2 \rangle  \in \Lambda$ is the coefficient of the $i$th equivariant parameter $(u, \theta)^i$ (cf. \cref{eqn:eq-param}). The first term $u^p \psi_b|_{q^\beta = 0} $ corresponds to the classical ($q^\beta = 0$) term of the operator. We denote the truncation of these structure constants by
\begin{equation}
    c_i \langle  u^p \psi_b; a_1, a_2 \rangle^E  \in \Lambda_{\le E}
\end{equation}
the image of $c_i \langle u^p  \psi_b; a_1, a_2 \rangle  \in \Lambda$ under the quotient $\Lambda \to \Lambda_{\le E}$ of the Novikov ring by the ideal $(q^\beta : \omega(\beta) > E)$. 

The goal of this subsection is to prove the following cluster decomposition:

\begin{prop}\label[prop]{prop:pcurv-cluster-decomposition}
    Let $J \in \mathcal{J}(X)^E$, and $\mathcal{C}(X)^E = \bigcup_\ell  \mathcal{C}(X)^E_{\mathcal{U}_\ell}$ a finite decomposition into $E$-clusters. Then there is an induced decomposition of the ($E$-truncation of the) structure constants,
    \begin{equation}
        c_i \langle  u^p \psi_b; a_1, a_2 \rangle  ^E = \sum_{\ell} c_i \langle  u^p \psi_b; a_1, a_2 \rangle ^E_{\mathcal{U}_\ell} \in \Lambda_{\le E},
    \end{equation}
    induced naturally in the sense that $c_i \langle u^p  \psi_b; a_1, a_2 \rangle ^E_{\mathcal{U}_\ell} $ are defined out of counts of moduli spaces of $J$-holomorphic maps whose underlying curves are supported in $\mathcal{U}_\ell$.
\end{prop}

We first prove the decomposition for the quantum product operator.  Let $b \in H^2(X;\mathbb{F}_p)$ and consider the quantum product
    \begin{equation}
        \mathbb{M}_b := b \ \star_q : QH^*(X;\Lambda_{\mathbb{F}_p})[u]\langle \theta \rangle \to QH^*(X;\Lambda_{\mathbb{F}_p})[u]\langle \theta \rangle
    \end{equation}
and its Novikov derivatives
\begin{equation}
    \mathbb{M}_b^{(j)} := \partial_b^j \mathbb{M}_b : QH^*(X;\Lambda_{\mathbb{F}_p})[u]\langle \theta \rangle \to QH^*(X;\Lambda_{\mathbb{F}_p})[u]\langle \theta \rangle.
\end{equation}
Let $\Lambda \to \Lambda_{\le E}$ be the $E$-truncation of the Novikov ring and denote the induced endomorphisms from $\mathbb{M}_b$, $\mathbb{M}_b^{(j)}$ on
\begin{equation}
    QH^*(X;\Lambda_{\mathbb{F}_p})[u]\langle \theta \rangle^E := H^*(X;\Lambda_{\le E})[u]\langle \theta \rangle
\end{equation}
with superscripts $E$ as $\mathbb{M}_b^E$, $\mathbb{M}_b^{(j), E}$. 

\begin{lemma}\label[lemma]{lemma:quantum-product-cluster-decomposition}
    Let $J \in \mathcal{J}(X)^E$, and $\mathcal{C}(X)^E = \bigcup_\ell  \mathcal{C}(X)^E_{\mathcal{U}_\ell}$ a finite decomposition into $E$-clusters. Then there is a cluster decomposition of operators
    \begin{equation}
        \mathbb{M}_b^E = \mathbb{M}_b|_{q^\beta = 0} +  \sum_\ell \mathbb{M}_{b, \mathcal{U}_\ell}^E, \quad         \mathbb{M}_b^{(j),E} = \sum_\ell \mathbb{M}_{b, \mathcal{U}_\ell}^{(j),E} 
    \end{equation}
    where the structure constants of $\mathbb{M}_{b, \mathcal{U}_\ell}^E$ are given by counts of moduli spaces of $J$-holomorphic maps satisfying incidence constraints whose underlying curves are supported in $\mathcal{U}_\ell$.
\end{lemma}
\begin{proof}
    Recall from \cref{lemma:quantum-product-correspondence} that the operator $\mathbb{M}_b$ is induced by the correspondence
    \begin{equation}
        \mathrm{ev}_0 \times \mathrm{ev}_\infty : \ov{\mathcal{M}}(X, J, \beta)_b \to X \times X
    \end{equation}
    in the sense that the evaluation map above restricted to $\mathcal{M}(X,J,\beta)_b \to X \times X$ defines a pseudocycle $\mathrm{ev}_*[\ov{\mathcal{M}}_{\beta,b}] \in \mathcal{H}_*(X \times X)$ and the $q^\beta$-term of the structure constant $(\mathbb{M}_b(a_1), a_2) \in \Lambda$ is given by the pairing of $\mathrm{ev}_*[\ov{\mathcal{M}}_{\beta,b}]$ with $\Psi(a_1 \otimes a_2)$. Indeed, for $J \in \mathcal{J}(X)^E$, the simple maps are regular. For $\beta \in \mathrm{im}(\pi_2(X) \to H_2(X;\mathbb{Z}))$ such that $\omega(\beta) > E$, note moreover from $J \in \mathcal{J}(X)^E$ that embedded images of two simple maps which are not reparametrizations of each other do not intersect. These claims together imply that indeed $\mathcal{M}(X,J,\beta)_b^E \to X \times X$ defines an integral pseudocycle (see \cite[Section 6.2, 6.3]{MS12}).

    Now, our claim reduces to \cref{lemma:ic-cluster-decomposition} to the case $k=3$ and $\nu = 0$: for $a_1, a_2 \in H^*(X;\mathbb{F}_p)$, we declare
    \begin{equation}
        \left( \mathbb{M}^E_{b, \mathcal{U}_\ell} (a_1), a_2 \right) := \sum_\beta \left(\ov{\mathcal{M}}_{0,3}(X,J, \beta)_{a_1, a_2, b}^E \right)_{\mathcal{U}_\ell}  q^\beta \in \Lambda_{\le E},
    \end{equation}
    that is $\mathbb{M}_{b, \mathcal{U}_\ell}^E$ is the operator defined by the correspondences (weighted by $q^\beta$ for different components)
    \begin{equation}
        \mathrm{ev}_0 \times \mathrm{ev}_\infty : \left( \ov{\mathcal{M}}_{0,3}(X,J, \beta)^E_b \right)_{\mathcal{U}_\ell} \to X \times X.
    \end{equation}
    By construction, there is a decomposition $\mathbb{M}_b^E = \sum_\ell \mathbb{M}^E_{b,\mathcal{U}_\ell}$ with the desired properties. The decomposition of $\mathbb{M}_b^{(j), E}$ follows from a purely algebraic reason by declaring
     \begin{equation}
        \left( \mathbb{M}^{(j),E}_{b, \mathcal{U}_\ell} (a_1), a_2 \right) := \sum_\beta \left(\ov{\mathcal{M}}_{0,3}(X,J, \beta)_{a_1, a_2, b}^E \right)_{\mathcal{U}_\ell}  (b \cdot \beta)^j q^\beta \in \Lambda_{\le E}.
    \end{equation}
\end{proof}

Using the correspondence description of the operators, we can further deduce the following. Assume the cluster decomposition of the operators $\mathbb{M}_b^E$ and $\mathbb{M}_b^{(j), E}$ as above.

\begin{lemma}\label[lemma]{lemma:quantum-product-iteration}
    For $\ell_1 \neq \ell_2$ and any $j_1, j_2 \in \mathbb{Z}_{\ge 0}$, the composition of cluster components always vanish:
    \begin{equation}
        \mathbb{M}_{b, \mathcal{U}_{\ell_1}}^{(j_1), E} \circ         \mathbb{M}_{b, \mathcal{U}_{\ell_2}}^{(j_2), E} = 0.
    \end{equation}
    The same vanishing holds when classical cup products are inserted:
    \begin{equation}
        \mathbb{M}_{b, \mathcal{U}_{\ell_1}}^{(j_1), E} \circ \left(\mathbb{M}_b|_{q^\beta=0}\right)^N   \circ     \mathbb{M}_{b, \mathcal{U}_{\ell_2}}^{(j_2), E} = 0 \quad (N \ge 1).
    \end{equation}
\end{lemma}
\begin{proof}
    By construction, the operator $\mathbb{M}_{b, \mathcal{U}_\ell}^{(j), E}$ is defined by the correspondence on $H^*(X;\Lambda_{\le E})$ induced by
    $\mathrm{ev}_0 \times \mathrm{ev}_\infty : \left( \ov{\mathcal{M}}_{0,3}(X,J)^E_b \right)_{\mathcal{U}_\ell} \to X \times X$,
    where the correspondence arising from different components indexed by $\beta \in \mathrm{im}(\pi_2(X) \to H_2(X;\mathbb{Z}))$ are weighted by $(b \cdot \beta)^j q^\beta \in \Lambda_{\le E}$. But for $\ell_1 \neq \ell_2$ the moduli spaces
    \begin{equation}
        \left( \ov{\mathcal{M}}_{0,3}(X,J)^E_b \right)_{\mathcal{U}_{\ell_1}}, \left( \ov{\mathcal{M}}_{0,3}(X,J)^E_b \right)_{\mathcal{U}_{\ell_2}}  
    \end{equation}
    are disjoint. In particular, they tautologically intersect transversely so the correspondences can be composed, and the composition must be equal to $0$. The second claim follows by the same proof since the classical multiplication is induced by the diagonal correspondence.
\end{proof}

\begin{proof}[Proof of \cref{prop:pcurv-cluster-decomposition}]
    The $p$-curvature operator is defined as the following expression in terms of the quantum product $\mathbb{M}_b$:
\begin{equation}
    u^p \psi_b = (u\nabla_b)^p - u^{p-1}(u\nabla_b) = (u \partial_b + \mathbb{M}_b)^p - u^{p-1}(u \partial_b + \mathbb{M}_b).
\end{equation}
Even though the expression \emph{a priori} involves $\partial_b$, the resulting endomorphism $\psi_b$ is linear. Using $[u \partial_b, \mathbb{M}_b^{(j)}] = u \mathbb{M}_b^{(j+1)}$, we see that $u^p \psi_b = \sum_{i \ge 0} f_i (\mathbb{M}_b, \mathbb{M}_b^{(1)}, \dots) u^i$ itself is a \emph{polynomial} expression in $u$ with coefficients given by $f_i$ that are linear combinations of  the compositions of operators $\mathbb{M}_b, \mathbb{M}_b^{(1)}, \mathbb{M}_b^{(2)}, \cdots$.

In particular, the structure constants in concern have the form
\begin{align}
    c_i \langle u^p \psi_b;a_1, a_2\rangle &= \left( f_i (\mathbb{M}_b, \mathbb{M}_b^{(1)}, \dots)(a_1), a_2 \right) \in \Lambda, \\ \quad 
    c_i \langle u^p  \psi_b;a_1, a_2\rangle^E &= \left( f_i (\mathbb{M}_b^E, \mathbb{M}_b^{(1), E}, \dots) (a_1), a_2 \right) \in \Lambda_{\le E}.
\end{align}
From \cref{lemma:quantum-product-cluster-decomposition} we may expand each of the operators $\mathbb{M}_b^E = \mathbb{M}_b|_{q^\beta = 0} + \sum_\ell \mathbb{M}_{b, \mathcal{U}_\ell}^E$,  $\mathbb{M}_b^{(j), E} = \sum_\ell \mathbb{M}_{b, \mathcal{U}_\ell}^{(j), E} $ ($j \ge 1$) in their cluster decompositions, and \cref{lemma:quantum-product-iteration} implies that
\begin{equation}
    f_i \left(\mathbb{M}_b|_{q^\beta = 0} + \sum_\ell \mathbb{M}_{b, \mathcal{U}_\ell}^{E}, \sum_\ell \mathbb{M}_{b, \mathcal{U}_\ell}^{(1), E}, \dots \right) = \sum_\ell g_i \left( \mathbb{M}_{b, \mathcal{U}_\ell}^{E}, \mathbb{M}_{b, \mathcal{U}_\ell}^{(1), E}, \dots \right)
\end{equation}
for some other polynomials $g_i$. Therefore we may define
\begin{equation}
    c_i \langle  u^p  \psi_b; a_1, a_2 \rangle ^E_{\mathcal{U}_\ell}  := \left( g_i (\mathbb{M}_{b, \mathcal{U}_\ell}^E, \mathbb{M}_{b, \mathcal{U}_\ell}^{(1), E}, \dots) (a_1), a_2 \right) \in \Lambda_{\le E}
\end{equation}
for the desired decomposition.
\end{proof}

\subsubsection{Quantum Steenrod operators}
We continue with the parallel discussion for the quantum Steenrod operators
\begin{equation}
    Q\Sigma_b : QH^*(X;\Lambda_{\mathbb{F}_p})[u]\langle \theta \rangle \to QH^*(X;\Lambda_{\mathbb{F}_p})[u]\langle \theta \rangle.
\end{equation}

For two fixed cohomology classes $a_1, a_2 \in QH^*(X;\Lambda_{\mathbb{F}_p})[u]\langle \theta \rangle$, we denote its structure constants by
\begin{equation}
  \left( Q\Sigma_b(a_1), a_2 \right) =:  \left( Q\Sigma_b|_{q^\beta = 0}(a_1), a_2 \right) +  \sum_{i \ge 0} c_i \langle  Q\Sigma_b; a_1, a_2 \rangle (u, \theta)^i \in \Lambda[u]\langle \theta \rangle
\end{equation}
and their $E$-truncations by
\begin{equation}
    c_i \langle  Q\Sigma_b; a_1, a_2 \rangle^E  \in \Lambda_{\le E}.
\end{equation}

Using cluster decomposition of moduli spaces with inhomogeneous perturbations and incidence constraints, we may conclude that:

\begin{prop}\label[prop]{prop:qst-cluster-decomposition}
   
    Let $J \in \mathcal{J}(X)^E$, and $\mathcal{C}(X)^E = \bigcup_\ell  \mathcal{C}(X)^E_{\mathcal{U}_\ell}$ a finite decomposition into $E$-clusters. Then there is an induced decomposition of the ($E$-truncation of the) structure constants,
    \begin{equation}
        c_i \langle  Q\Sigma_b; a_1, a_2 \rangle  ^E = \sum_{\ell} c_i \langle  Q\Sigma_b; a_1, a_2 \rangle ^E_{\mathcal{U}_\ell} \in \Lambda_{\le E},
    \end{equation}
    induced naturally in the sense that $c_i \langle  Q\Sigma_b; a_1, a_2 \rangle ^E_{\mathcal{U}_\ell} $ are defined out of counts of moduli spaces of $(J, \nu)$-holomorphic maps whose underlying curves are supported in $\mathcal{U}_\ell$.
\end{prop}

\begin{proof}
    Recall from \cref{sssec:quantum-Steenrod-definition}, \cref{lem:qst-pseudocycle} that the quantum Steenrod operations are defined from the $p$-pseudocycles from evaluation maps  $\mathrm{ev}^{eq, i}: \ov{\mathcal{M}}^{eq, i}_\beta \to X \times X \times X^p$.
    on the parametrized moduli spaces equipped with $\mu_p$-equivariant inhomogeneous terms. In fact, by \cref{rem:CY-degree}, we may work with a finite-dimensional approximation $E\mu_p^{\mathrm{fin}} := S^{2N+1} \subseteq S^\infty$ for $N \gg 0$ and have the $\mu_p$-equivariant inhomogeneous perturbations to be parametrized over $E\mu_p^{\mathrm{fin}}$:
    \begin{equation}
        \nu^{eq} \in \Gamma ( E\mu_p^{\mathrm{fin}} \times_{\mu_p} \mathbb{P}^1 \times X , \Omega_{\mathbb{P}^1}^{0,1} \boxtimes TX).
    \end{equation}
    We define the corresponding equivariant moduli spaces as
    \begin{equation}
        \mathcal{M}^{eq}_\beta = \mathcal{M}^{eq}(X, J, \nu, \beta) := \left \{ (u: \mathbb{P}^1 \to X, w \in E\mu_p^{\mathrm{fin}}) :\  u_*[\mathbb{P}^1] = \beta, \ \ov{\partial}_J u = \nu^{eq}_w\right\}
    \end{equation}
    together with their stable map compactifications $\ov{\mathcal{M}}^{eq}_{\beta}$ (fiberwise under $\mathcal{M}^{eq}_\beta \to E\mu_p^{\mathrm{fin}}$). The construction of the quantum Steenrod operators applies without change by using these moduli spaces and the induced moduli spaces $\mathcal{M}^{eq, i}_\beta$ (obtained as fiber products with $\ov{\Delta_i}$). In particular, the evaluation map $\mathrm{ev}^{eq, i} : \ov{\mathcal{M}}_\beta^{eq, i} \to X \times X \times X^p$ restricted to the open locus defines a $p$-pseudocycle.

    Now consider the energy bound $E$. Since the inhomogeneous terms $\nu^{eq}_w$ are parametrized over a compact manifold $w \in E\mu_p^{\mathrm{fin}} = S^{2N+1}$, a parametric form of \cref{lemma:nu-cluster-decomposition} applies to the equivariant moduli spaces $\ov{\mathcal{M}}_\beta^{eq, i, E}$ of maps with energy $\omega(\beta) \le E$, yielding the decomposition $\ov{\mathcal{M}}_\beta^{eq, i, E} = \bigcup_\ell (\ov{\mathcal{M}}_\beta^{eq, i, E})_{\mathcal{U}_\ell} = \bigcup_{\ell} c^{-1}(\mathcal{U}_\ell)$ for disjoint open and closed subsets $(\ov{\mathcal{M}}_\beta^{eq, i})_{\mathcal{U}_\ell} $. 
    
    Fix $p$-pseudocycle representatives of $a_1, a_2, b \in H^*(X;\mathbb{F}_p)$ given by $\Psi(a_i) : A_i \to X$ and $\Psi(b) : B \to X$. Applying \cref{lemma:ic-transversality} and \cref{lemma:ic-cluster-decomposition} to the equivariant moduli spaces $\ov{\mathcal{M}}_\beta^{eq, i}$ with incidence constraints
\begin{center}
    \begin{tikzcd}
        & A_1 \times A_2  \times B^{\times p} \dar["{\Psi(a_1 \otimes a_2 \otimes b^{\otimes p}})"] \\ \ov{\mathcal{M}}^{eq, i, E} \rar["\mathrm{ev}"] &  X \times X \times X^p,
    \end{tikzcd}
\end{center}
we see that the ($i$th) quantum Steenrod correlator for $\omega(A) \le E$ admits an induced decomposition
\begin{align}
    \langle a_1, a_2, b^{\otimes p}\rangle_\beta^{eq, i} &:= \mathrm{ev}_* [\ov{\mathcal{M}}^{eq, i}_\beta] \cap \Psi(a_1 \otimes a_2 \otimes b^{\otimes p}) \\
    &=\# (\ov{\mathcal{M}}^{eq,i}_\beta)^E_{a_1, a_2, b, \dots, b} \\
    &= \# \sum_\ell \left( (\ov{\mathcal{M}}^{eq,i}_\beta)^E_{a_1, a_2, b, \dots, b} \right)_{\mathcal{U}_\ell}.
\end{align}
It remains to declare
\begin{equation}
    c_i \langle Q\Sigma_b; a_1, a_2 \rangle^E_{\mathcal{U}_\ell} = \sum_{\beta, \omega(\beta) \le E} (-1)^\dagger \#  \left( (\ov{\mathcal{M}}^{eq,i}_\beta)^E_{a_1, a_2, b, \dots, b} \right)_{\mathcal{U}_\ell} q^\beta \in \Lambda_{\le E}
\end{equation}
to obtain the desired decomposition.
\end{proof}

\subsection{Contributions from elementary clusters}

In the previous section, we showed that both the $p$-curvature endomorphism $\psi_b$ and the quantum Steenrod operator $Q\Sigma_b$ (more precisely, $E$-truncations thereof) admit a cluster decomposition. Here we show that the cluster components of the structure constants agree for a particular type of a cluster that can be identified with the local $\mathbb{P}^1$ geometry. 

Let $(X,\omega)$ be a Calabi--Yau manifold of real dimension $6$.

\begin{defn}[{\cite[Definition 3.1]{IP-GV}, \cite[Definition 3.42]{doan2021gopakumar}}]
    A cluster $(\mathcal{U}, J, C)$ is called \emph{elementary} if
    \begin{enumerate}
        \item The core curve $C \cong \mathbb{P}^1$ with its $J$-holomorphic embedding $u: C \to X$ has a normal bundle $u^*TX/TC =:N_C \cong \mathcal{O}(-1) \oplus \mathcal{O}(-1)$ as almost complex manifolds. In particular, the linearization of $u : C \to X$ uniquely descends to a \emph{normal operator} $D_N:\Gamma(NC) \to \Gamma(\Omega^{0,1}_C \otimes NC)$.
        \item The only non-trivial $J$-holomorphic maps into an $\epsilon$-neighborhood of $C$ are multiple covers of the embedding $C \hookrightarrow X$.
        \item For any such multiple cover $\rho: \widetilde{C} \to C$, the pullback $\rho^* D_N$ is injective.
    \end{enumerate}
\end{defn}
The normal operator $D_N : \Gamma(NC) \to \Gamma(\Omega^{0,1}_C \otimes NC)$ is the restriction of the linearization $D_u$ to the normal directions, see \cite[(4.15)]{IP-GV} or \cite[Section 3.2]{Lee23a}. Property (3) is called the \emph{super-rigidity} of the core curve $C$; it means that there are no (first-order) normal perturbations of the curve, even for any of its multiple covers. It moreover implies that the core $C$ must be transversely cut out.

Our situation is simpler than that of \cite{IP-GV, doan2021gopakumar} in the sense that we only consider rational curves:

\begin{lemma}\label{lem:elementary-cluster-existence}
    There exists an elementary cluster $(\mathcal{U}, J, C)$ whose core curve is $C \cong \mathbb{P}^1$, such that $\mathcal{U} \subseteq X$ is isomorphic as complex manifolds to the total space $\mathcal{O}(-1) \oplus \mathcal{O}(-1) \to C$. In particular, $J$ can be chosen to be integrable near $C$.
\end{lemma}
\begin{proof}
    Fix $J$ to be an $\omega$-compatible almost complex structure $J$ such that all simple maps are transversely cut out, and let $u: C \subseteq X$ be a genus $0$ embedded $J$-holomorphic curve. Then $TC$ is a complex subbundle of $(u^*TX, J)$ and the normal bundle $N_C = u^*TX/TC$ is a complex vector bundle on $C \cong \mathbb{P}^1$, hence isomorphic to $\mathcal{O}(a) \oplus \mathcal{O}(b)$ as complex vector bundles. We deform the almost complex structure near $C$ so that it agrees with the natural holomorphic structure on $N_C$ transported from the total space of $\mathcal{O}(a) \oplus \mathcal{O}(b)$. Now note that $a+b = -2$ since $\det(\mathcal{O}(a) \oplus \mathcal{O}(b)) \cong K_C \cong \mathcal{O}(-2)$. Moreover, $H^0(C, N_C) = 0$ since $C$ is regular with index $0$ and therefore should not have any $J$-holomorphic deformations; hence, $a, b < 0$. This implies that $a=b=-1$.
\end{proof}

In light of \cref{lem:elementary-cluster-existence}, we may consider an elementary cluster $(\mathcal{U}, J, C)$ as a codimension $0$ $J$-holomorphic embedding of the $\epsilon$-disk bundle in the total space $Y = \mathrm{Tot}(\mathcal{O}(-1) \oplus \mathcal{O}(-1) \to \mathbb{P}^1)$ of the local $\mathbb{P}^1$-geometry with a tubular neighborhood $N_C \subseteq X$ of the core curve $C$, such that the zero section of $Y$ is identified with the core curve $C$. We will freely use this identification below.

\begin{prop}\label{prop:elementary-cluster-qst-is-pcurv}
    Fix $b \in H^2(X;\mathbb{F}_p)$. Then the cluster components of the structure constants
    \begin{equation}
         c_i \langle Q\Sigma_b ; a_1, a_2 \rangle^E_{\mathcal{U}_\ell} = c_i \langle u^p \psi_b ; a_1, a_2 \rangle^E_{\mathcal{U}_\ell} \in \Lambda_{\le E}
    \end{equation}
    agree when $(\mathcal{U}_\ell, J, C_\ell)$ is an elementary cluster.
\end{prop}

We would like to reduce \cref{prop:elementary-cluster-qst-is-pcurv} to the previous claim from \cref{thm:localP1-qst-is-pcurv} about the equivalence
\begin{equation}
    Q\Sigma_e = u^p \psi_e \in \mathrm{End}(H^*(Y) \oplus H^*_c(Y))
\end{equation}
for the local $\mathbb{P}^1$ 3-fold $Y$ with $e \in H^2(Y)$ given by the Poincar\'e dual to the fiber $Y \to \mathbb{P}^1$. We fix the insertion $b \in  H^2(X;\mathbb{F}_p)$ throughout. 

\begin{lemma}\label[lemma]{lem:codim-3}
    Let $(\mathcal{U}_\ell, J, C_\ell)$  be an elementary cluster. Let $a_1, a_2 \in H^*(X;\mathbb{F}_p)$ be cohomology classes, and assume that one of the degrees $|a_1|, |a_2|$ is $\ge 3$.  Then the corresponding cluster components of the structure constants
    \begin{equation}
         c_i \langle Q\Sigma_b ; a_1, a_2 \rangle^E_{\mathcal{U}_\ell} = c_i \langle u^p \psi_b ; a_1, a_2 \rangle^E_{\mathcal{U}_\ell} \in \Lambda_{\le E}
    \end{equation}
    vanish.
\end{lemma}
\begin{proof}
    Without loss of generality, we consider the class $a_1 \in H^*(X;\mathbb{F}_p)$ and assume that  $|a_1| \ge 3$. Recall that the operator $Q\Sigma_b$ have structure constants given by the intersection counts of evaluation images from moduli spaces of (perturbed) $J$-holomorphic maps and $p$-pseudocycles $\Psi(a_i)$, $\Psi(b)$ representing the cohomology classes $a_i \in H^*(X;\mathbb{F}_p)$, $b \in H^*(X;\mathbb{F}_p)$. The cluster component $ c_i \langle Q\Sigma_b ; a_1, a_2 \rangle^E_{\mathcal{U}_\ell}$ of the quantum Steenrod operator only involves maps whose underlying curve is contained in the cluster $\mathcal{U}_\ell$, hence is $\epsilon$-close to the core curve $C_\ell$ under the Hausdorff topology. Since $|a_1| \ge 3$, the representing pseudocycle $\Psi(a_1)$ is of codimension $\ge 3$ in $X$. By transversality, we may assume that the pseudocycles $\Psi(a_1)$ does not intersect $C_\ell$. Therefore, may choose the perturbation data $\nu^{eq}$ for quantum Steenrod moduli spaces small enough that $\Psi(a_i)$ does not intersect any of the curves from the quantum Steenrod moduli spaces. Hence $c_i \langle Q\Sigma_b ; a_1, a_2 \rangle_{\mathcal{U}_\ell}^E = 0$. By the same argument, we may show that the structure constants of $\mathbb{M}_b$ vanish when $a_1$ is inserted. From the proof of \cref{prop:pcurv-cluster-decomposition} we see that this implies that the structure constants $c_i \langle u^p \psi_b; a_1, a_2 \rangle_{\mathcal{U}_\ell}^E$ vanish.
    
\end{proof}

\begin{proof}[Proof of \cref{prop:elementary-cluster-qst-is-pcurv}]
    By \cref{lem:codim-3}, we may assume that the degrees of the insertions $a_1, a_2$ satisfy $|a_1|, |a_2| \le 2$. We consider $p$-pseudocycles $\Psi(a_1)$, $\Psi(a_2)$, and $\Psi(b)$ representing the cohomology classes $a_1, a_2, b \in H^*(X;\mathbb{F}_p)$. By \cref{lemma:ic-transversality}, we may further assume that these $p$-pseudocycles intersect the core curve $C_\ell$ of the elementary cluster transversely, and in particular the boundary of $p$-pseudocycles are supported away from $C_\ell$.

    For the neighborhood $Y \cong N_{C_\ell} \subseteq X$ of the core curve $C_\ell$, there is a  pullback map in cohomology $H^*(X;\mathbb{F}_p) \to H^*(N_C;\mathbb{F}_p) \cong H^*(Y;\mathbb{F}_p) = H^0(Y;\mathbb{F}_p) \oplus H^2(Y;\mathbb{F}_p)$. We may restrict to a small neighborhood $Y^\epsilon \subseteq Y$ of the zero section $\mathbb{P}^1 \to Y$, which is homotopy equivalent to $Y$, to assume that the image of the pseudocycles $\Psi(a_i), \Psi(b)$ in $Y^\epsilon$ are locally finite cycles in $Y^\epsilon$ representing classes in $H^*(Y^\epsilon;\mathbb{F}_p) \cong H^*(Y;\mathbb{F}_p)$. Let us denote the image of a cohomology class $a \in H^*(X;\mathbb{F}_p)$ under this pullback map by $a' \in H^*(Y;\mathbb{F}_p)$.
    
    Under this pullback map, $b \in H^2(X;\mathbb{F}_p)$ gets sent to $b' = k \cdot e \in H^2(Y;\mathbb{F}_p)$ for some $k \in \mathbb{Z}$ and $e = [\mathrm{fiber}]$ is the class Poincar\'e dual to the fiber of $Y \to \mathbb{P}^1$. 

    It remains to note that the $(J,\nu)$-holomorphic maps to $X$ whose underlying curves are supported in the elementary cluster $\mathcal{U}_\ell$ can be identified with $(J, \nu)$-holomorphic maps to $Y$. Therefore, the structure constants $c_i \langle Q\Sigma_b; a_1, a_2 \rangle^E_{\mathcal{U}_\ell}$ and $c_i \langle u^p \psi_b; a_1, a_2 \rangle^E_{\mathcal{U}_\ell}$ can be identified with ($E$-truncations) of the corresponding structure constants $c_i\langle Q\Sigma_{b'}; a_1', a_2' \rangle^E$ and $c_i\langle u^p  \psi_{b'}; a_1', a_2' \rangle^E$ in $Y$. But \cref{thm:localP1-qst-is-pcurv} implies that the operations $Q\Sigma_{b'}$ and $u^p \psi_{b'}$ agree, so that 
    \begin{equation}
        c_i\langle Q\Sigma_{b'}; a_1', a_2' \rangle^E = c_i\langle u^p \psi_{b'}; a_1', a_2' \rangle^E \in \Lambda_{\le E}
    \end{equation}
    (under our degree constraints $|a_1|, |a_2| \le 2$, these correspond to the matrix coefficients in the ``lower left'' $2 \times 2$ block from \cref{eqn:localP1-pcurv-formula}). 
\end{proof}

\subsection{Cluster isotopy theorem}
The main goal of this subsection is to prove the following statement.

\begin{prop}\label{prop:isotopy}
    Fix $b \in H^2(X;\mathbb{F}_p)$ and cohomology classes $a_1, a_2 \in QH^*(X;\Lambda_{\mathbb{F}_p})[u]\langle \theta \rangle$. For two $E$-clusters $(\mathcal{U}_0, J_0, C)$ and $(\mathcal{U}_1, J_1, C)$ such that $J_0, J_1 \in \mathcal{J}(X)^E$ with identical core curves, there exists a finite set of $E$-clusters $\{ (\mathcal{U}_k, J_k, C_k)\}$ such that 
    \begin{equation}
    \begin{aligned}
    c_i \langle Q\Sigma_b ; a_1, a_2 \rangle^E_{\mathcal{U}_0}  &= \pm c_i \langle Q\Sigma_b ; a_1, a_2 \rangle^E_{\mathcal{U}_1} + \sum_k c_i \langle Q\Sigma_b ; a_1, a_2 \rangle^E_{\mathcal{U}_k}
    \\
    c_i \langle  u^p \psi_b; a_1, a_2 \rangle^E_{\mathcal{U}_0}  &= \pm c_i \langle u^p \psi_b; a_1, a_2 \rangle^E_{\mathcal{U}_1} + \sum_k c_i \langle u^p  \psi_b; a_1, a_2 \rangle^E_{\mathcal{U}_k}\\
    \end{aligned}
    \end{equation}
    where the core curves $C_i$ are embedded and transversely cut out, and its homology class satisfies $[C_i] = d_i [C]$ for some $d_i \geq 2$.
\end{prop}

To prove Proposition \ref{prop:isotopy}, we follow Ionel--Parker's cluster isotopy method \cite[Section 7]{IP-GV}, which was also used in \cite[Section 3.3]{doan2021gopakumar}. We will focus on explaining how to adapt their argument to our setting and refer the reader to the original reference of the technical details.

To this end, we introduce some notations. To simplify the notations, given two $E$-clusters $(\mathcal{U}_0, J_0, C_0)$ and $(\mathcal{U}_1, J_1, C_1)$ with $[C_0] = [C_1]$, we write
\begin{equation}
    c_i \langle Q\Sigma_b ; a_1, a_2 \rangle^E_{\mathcal{U}_0} \approx c_i \langle Q\Sigma_b ; a_1, a_2 \rangle^E_{\mathcal{U}_1}
\end{equation}
if
\begin{equation}
    c_i \langle Q\Sigma_b ; a_1, a_2 \rangle^E_{\mathcal{U}_0}  = c_i \langle Q\Sigma_b ; a_1, a_2 \rangle^E_{\mathcal{U}_1} + \sum_k c_i \langle Q\Sigma_b ; a_1, a_2 \rangle^E_{\mathcal{U}_k}
\end{equation}
where $\{ (\mathcal{U}_k, J_k, C_k)\}$ consists of $E$-clusters such that $[C_i] = d_i [C]$ for some $d_i \geq 2$. We write
\begin{equation}
    c_i \langle  u^p \psi_b; a_1, a_2 \rangle^E_{\mathcal{U}_0}  \approx  c_i \langle  u^p \psi_b; a_1, a_2 \rangle^E_{\mathcal{U}_1}
\end{equation}
for a similar meaning for $p$-curvature.

Denote by $\mathcal{J}(X)$ the space of $\omega$-compatible almost complex structures on $(X, \omega)$, equipped with $C^\infty$-topology. Write
\begin{equation}
    \mathcal{M}(X)^{\mathrm{emb}} := \{ (J, [u]) \ | \ J \in \mathcal{J}(X), \ u: S^2 \to X \ \text{embdded $J$-holomorphic sphere}\},
\end{equation}
i.e., the universal moduli space of \emph{embedded} genus $0$ $J$-holomorphic curves, where $[u]$ stands for the classes represented by $u$ up to reparametrization. $\mathcal{M}(X)^{\mathrm{emb}}$ is an infinite-dimensional smooth manifold. We equip $\mathcal{M}(X)^{\mathrm{emb}}$ with the product of $C^\infty$-topology and Gromov topology so that the projection map $\mathrm{pr}_{\mathcal{J}}: \mathcal{M}(X)^{\mathrm{emb}} \to \mathcal{J}(X)$ is continuous. For any $\beta \in H_2(X;\mathbb{Z})$, write $\mathcal{M}(X)^{\mathrm{emb}}_\beta$ the component of $\mathcal{M}(X)^{\mathrm{emb}}$ whose underlying curve represents the class $\beta$.

Given any stable $J$-holomorphic map $u: \Sigma \to X$, we denote the deformation operator of $u$ by $D_{u,J}$, whose index computes the virtual dimension of the moduli space. Consider the subspace
\begin{equation}
    \mathcal{W}^k := \{ (J, [u]) \ | \ \dim (\ker(D_{u,J})) = \dim (\mathrm{coker}(D_{u,J})) = k\} \subset \mathcal{M}(X)^{\mathrm{emb}}
\end{equation}
and their union
\begin{equation}
    \mathcal{W} = \bigcup_{k \geq 1} \mathcal{W}^k.
\end{equation}
By \cite[Proposition 5.3]{IP-GV}, $\mathcal{W}^1 \subset \mathcal{M}(X)^{\mathrm{emb}}$ is a codimension $1$ smooth submanifold. By \cite[Lemma 5.6]{IP-GV}, the subset $\mathcal{A} \subset \mathcal{W}^1$ specified by the condition that $\mathrm{pr}_{\mathcal{J}}$ fails to be an immersion, is a codimension $1$ submanifold of $\mathcal{W}^1$.

The first technical lemma towards establishing Proposition \ref{prop:isotopy} is the situation where the two $E$-clusters can be connected without wall-crossing in $\mathcal{J}(X)$.

\begin{lemma}\label{lemma:tech-1}
    Let $(J_t, [u_t])_{0 \leq t \leq 1}$ be a path in $\mathcal{M}(X)^{\mathrm{emb}}_\beta$ in $\mathcal{M}(X)^{\mathrm{emb}} \setminus \mathcal{W}$. Let $C_0$ and $C_1$ be the images of $u_0$ and $u_1$ respectively. If $(\mathcal{U}_0, J_0, C_0)$ and $(\mathcal{U}_1, J_1, C_1)$ are $E$-clusters, then we have
    \begin{equation}
        c_i \langle Q\Sigma_b ; a_1, a_2 \rangle^E_{\mathcal{U}_0} \approx c_i \langle Q\Sigma_b ; a_1, a_2 \rangle^E_{\mathcal{U}_1}, \quad \quad \quad c_i \langle  u^p \psi_b; a_1, a_2 \rangle^E_{\mathcal{U}_0}  \approx c_i \langle  u^p \psi_b; a_1, a_2 \rangle^E_{\mathcal{U}_1}.
    \end{equation}
\end{lemma}
\begin{proof}
    The first input here is the cluster stability result, cf. \cite[Lemma 2.1]{IP-GV} and \cite[Proposition 3.20]{doan2021gopakumar}, which is similar to the Gromov compactness argument in our discussion of small inhomogeneous perturbations. The conclusion of the referenced results implies that for any $t \in [0,1]$, we can find an $E$-cluster $(\mathcal{U}_t, J_t, C_t)$ such that for any $s$ lying in a sufficiently small open neighborhood of $t$, the triple $(\mathcal{U}_t, J_s, C_s)$ is also an $E$-cluster. The contribution of these two $E$-clusters to the truncated generating series is the same, which can also be achieved in our setting with inhomogeneous perturbations if the perturbation is sufficiently small and $s$ is sufficiently close to $t$. The second input here is the cluster refinement result, cf. \cite[Corollary 2.5]{IP-GV} and \cite[Proposition 3.19]{doan2021gopakumar}, which implies that for all such $s$, the difference between $(\mathcal{U}_t, J_s, C_s)$ and $(\mathcal{U}_s, J_s, C_s)$ can be decomposed into a finite union of $E$-clusters whose underlying core curve has homology class $d[C_s]$ for some $d \geq 2$. In particular, using the cluster decomposition statements \cref{prop:pcurv-cluster-decomposition} and \cref{prop:qst-cluster-decomposition}, we conclude that for any $t \in [0,1]$ and $s$ sufficiently close to $t$, we can find $E$-clusters $(\mathcal{U}_t, J_t, C_t)$ and $(\mathcal{U}_s, J_s, C_s)$ such that
    \begin{equation}
        c_i \langle Q\Sigma_b ; a_1, a_2 \rangle^E_{\mathcal{U}_t} \approx c_i \langle Q\Sigma_b ; a_1, a_2 \rangle^E_{\mathcal{U}_s}, \quad \quad \quad c_i \langle  u^p  \psi_b; a_1, a_2 \rangle^E_{\mathcal{U}_t}  \approx c_i \langle u^p  \psi_b; a_1, a_2 \rangle^E_{\mathcal{U}_s}.
    \end{equation}
    Using the compactness and connectedness of $[0,1]$, by covering $[0,1]$ with finitely many intervals with above property, we conclude the Lemma.
\end{proof}

By \cite[Lemma 6.5]{IP-GV}, any path in $\mathcal{J}(X)^E$ whose end points do not lie in the image of $\mathcal{W}$ under $\mathrm{pr}_{\mathcal{J}}$ can be deformed to a path in $\mathcal{J}(X)$ which intersects the wall $\mathcal{W}$ at finitely many points in $\mathcal{W}^1 \setminus \mathcal{A}$ transversely while keeping the end points fixed. Moreover, such path can be chosen so that for the correponding $J$, any simple curve of energy $\leq E$ is embedded and any two $J$-holomorphic simple curves not different by a reparametrization have disjoint images (the set of such $J$ is denoted by $\mathcal{J}^E_{isol}$ in \emph{loc.cit.}). We need to deal with wall-crossings along paths of this kind.

\begin{lemma}\label{lem:tech-2}
    Let $(J_t)_{-1 \leq t \leq 1}$ be a $C^1$-path in $\mathcal{J}(X)$ that is contained in $\mathcal{J}^E_{isol}$. Suppose $\mathrm{pr}_{\mathcal{J}}: \mathcal{M}(X)^{\mathrm{emb}}_A \to \mathcal{J}(X)$ intersects $(J_t)_{-1 \leq t \leq 1}$ transversely at $(J_0, [u_0]) \in \mathcal{W}^1 \setminus \mathcal{A}$. Then there exists $\delta > 0$, $\sigma \in {\pm 1}$ and a neighborhood $U$ of $(J_0, [u_0])$ in $\mathcal{M}(X)^{\mathrm{emb}}$ such that
    \begin{equation}
    \mathrm{pr}_{\mathcal{J}}^{-1}(J_t) \cap U = 
    \begin{cases}
    \{ C_t^{\pm}\}, &\text{for } 0 < |t| < \delta, \mathrm{sign}(t) = \sigma,\\
    \emptyset, &\text{for } 0 < |t| < \delta, \mathrm{sign}(t) = -\sigma.
    \end{cases}
    \end{equation}
    Moreover, if $0 < |t| < \delta$ and $\mathrm{sign}(t)=\sigma$, given two $E$-clusters $(\mathcal{U}_t^-, J_t, C_t^-)$ and $(\mathcal{U}_t^+, J_t, C_t^+)$, we have
    \begin{equation}
        c_i \langle Q\Sigma_b ; a_1, a_2 \rangle^E_{\mathcal{U}_t^+} \approx - c_i \langle Q\Sigma_b ; a_1, a_2 \rangle^E_{\mathcal{U}_t^-}, \quad \quad \quad c_i \langle  u^p \psi_b; a_1, a_2 \rangle^E_{\mathcal{U}_t^+}  \approx - c_i \langle u^p  \psi_b; a_1, a_2 \rangle^E_{\mathcal{U}_t^-}.
    \end{equation}
\end{lemma}
\begin{proof}
    The first part of the statement is the delicate Kuranishi model analysis presented in \cite[Corollary 6.3]{IP-GV}. As for the consequence on enumerative invariants, it follows from the following two observations. First, if $C_0$ denotes the underlying embedded $J$-holomorphic curve of $u_0$, we can find an $E$-cluster $(\mathcal{U}_0, J_0, C_0)$ for $C_0$ such that for any $s$ sufficiently close to $0$, the generating series
    \begin{equation}
        c_i \langle Q\Sigma_b ; a_1, a_2 \rangle^E_{\mathcal{U}_s}, \quad \quad \quad c_i \langle u^p \psi_b; a_1, a_2 \rangle^E_{\mathcal{U}_s},
    \end{equation}
    which are defined with respect to the almost complex structure $J_s$ by looking at ($\nu$-perturbed) $J_s$-holomorphic curves mapped into $\mathcal{U}_0$, are well-defined: cf. the cluster stability result discussed in the proof of Lemma \ref{lemma:tech-1}. Additionally, the resulting series are independent of $s$, which is a consequence of the deformation invariance. Second, for $\mathrm{sign}(t) = \sigma$, if $|s| < \delta$ the cluster decomposition applied to $\mathcal{U}_s$ implies that
    \begin{equation}
    \begin{aligned}
        c_i \langle Q\Sigma_b ; a_1, a_2 \rangle^E_{\mathcal{U}_s} &\approx c_i \langle Q\Sigma_b ; a_1, a_2 \rangle^E_{\mathcal{U}_t^+} + c_i \langle Q\Sigma_b ; a_1, a_2 \rangle^E_{\mathcal{U}_t^-}, \\ 
        c_i \langle u^p \psi_b; a_1, a_2 \rangle^E_{\mathcal{U}_s} &\approx c_i \langle u^p \psi_b; a_1, a_2 \rangle^E_{\mathcal{U}_t^+}  +c_i  \langle u^p \psi_b; a_1, a_2 \rangle^E_{\mathcal{U}_t^-},
    \end{aligned}
    \end{equation}
    by \cref{prop:pcurv-cluster-decomposition} and \cref{prop:qst-cluster-decomposition}. Therefore, the computations that $c_i \langle Q\Sigma_b ; a_1, a_2 \rangle^E_{\mathcal{U}_s} \approx 0$ and $c_i \langle u^p \psi_b; a_1, a_2 \rangle^E_{\mathcal{U}_s} \approx 0$ for $\mathrm{sign}(s) = - \sigma$ show the Lemma holds.
\end{proof}

Finally, we need the following lemma which is concerned with isotopies of almost complex structures keeping the core curve pseudoholomorphic. This is needed for comparing cluster contributions with the ones from the standard (elementary) cluster. To set up the stage, we introduce some further notations. Given an embedded closed $2$-dimensional submanifold $C \subset X$, write $\mathcal{J}_C \subset \mathcal{J}(X)$ is subset consisting of all $J$ for which $C$ is $J$-holomorphic. Recall the definition of $\mathcal{J}^E_{isol}$ proceeding Lemma \ref{lem:tech-2}.

\begin{lemma}\label{lemma:tech-3}
    Let $(J_t)_{-1 \leq t \leq 1}$ be a $C^1$-path in $\mathcal{J}_C \cap \mathcal{J}(X)$ that is contained in $\mathcal{J}^E_{isol}$. Suppose $\mathrm{pr}_{\mathcal{J}}: \mathcal{M}(X)^{\mathrm{emb}}_A \to \mathcal{J}(X)$ intersects $(J_t)_{-1 \leq t \leq 1}$ transversely at $(J_0, [u_0]) \in \mathcal{W}^1 \setminus \mathcal{A}$ where $C$ is the underlying curve associated to $u_0$. Then there exists a $\delta > 0$ such that if $(\mathcal{U}_+, J_{\delta}, C)$ and $(\mathcal{U}_-, J_{-\delta}, C)$ are $E$-clusters, then
    \begin{equation}
        c_i \langle Q\Sigma_b ; a_1, a_2 \rangle^E_{\mathcal{U}_+} \approx - c_i \langle Q\Sigma_b ; a_1, a_2 \rangle^E_{\mathcal{U}_-}, \quad \quad \quad c_i \langle u^p \psi_b; a_1, a_2 \rangle^E_{\mathcal{U}_+}  \approx - c_i \langle u^p \psi_b; a_1, a_2 \rangle^E_{\mathcal{U}_-}.
    \end{equation}
\end{lemma}
\begin{proof}
    Again, the most important technical input is the Kuranishi model discussed in \cite[Corollary 6.3]{IP-GV}, more specifically, Equation (6.6) in \emph{loc.cit.}. Given the cluster decomposition results \cref{prop:pcurv-cluster-decomposition} and \cref{prop:qst-cluster-decomposition} and deformation invariance of our invariants, using arguments similar to those in Lemma \ref{lemma:tech-1} and Lemma \ref{lem:tech-2}, one can replicate the proof of \cite[Lemma 7.4]{IP-GV} to deduce the desired result, where we need to use Lemma \ref{lem:tech-2} in the process.
\end{proof}

\begin{proof}[Proof of Proposition \ref{prop:isotopy}]
    We consider the space $\mathcal{J}_C$, which is path-connected. Using \cite[Lemma 6.7]{IP-GV}, we can find a path in $\mathcal{J}_C$ connecting $J_0$ and $J_1$ such that the path $(J_t, C)_{0 \leq t \leq 1}$ in $\mathcal{M}(X)^{\mathrm{emb}}_{[C]}$ intersects $\mathcal{W}$ transversely at finitely many points in $\mathcal{W}^1 \setminus \mathcal{A}$ at times $\{t_i\}$, such that the subset $\mathrm{pr}_{\mathcal{J}}^{-1}([0,1] \setminus \{t_i\})$ is a smooth $1$-dimensional manifold intersecting the wall $\mathcal{W}$ transversely at finitely many points in $\mathcal{W}^1 \setminus \mathcal{A}$. Then such a path necessarily lies in $\mathcal{J}^E_{isol}$. Therefore, the Proposition follows from combining Lemma \ref{lemma:tech-1} and Lemma \ref{lemma:tech-3}, applied to the complement of $\{t_i\}$ and the wall-crossing respectively.
\end{proof}



\subsection{Final proof}
We explain how to deduce \cref{thm:main} using the technical results above.




\begin{proof}[Proof of \cref{thm:main}]
    Let $b \in \mathrm{im}(H^2(X;\mathbb{Z}) \to H^2(X;\mathbb{F}_p))$. Then the total Steenrod power of $b \in H^2(X;\mathbb{F}_p)$ is given by
    \begin{equation}
        \mathrm{St}(b) = b^p - u^{p-1}b \in H^{2p}(X;\mathbb{F}_p)[u]\langle \theta \rangle,
    \end{equation}
    since the Bockstein $\beta$ vanishes on $b$. Since $\mathrm{St}(b) \smile = Q\Sigma_b |_{q^\beta =0}$ under the projection $(-)|_{q^\beta=0} : \Lambda \to \Lambda_{\le 0}$ by the general property of the quantum Steenrod operations \cite[Proposition 2.21]{Lee23a}, and $( b^p - u^{p-1} b) \smile = u^p \psi_b |_{q^\beta =0}$, we conclude that the classical terms of the two operations agree,  $Q\Sigma_b |_{q^\beta =0} = u^p  \psi_b |_{q^\beta =0}$. It suffices to consider the non-classical terms (i.e. terms with $q^\beta$ with $\omega(\beta) > 0$) in both operations.

    Fix an energy bound $E> 0$, an almost complex structure $J \in \mathcal{J}(X)^E$ and a finite cluster decomposition $\mathcal{C}(X)^E = \bigcup_{\ell} \mathcal{C}(X)^E_{\mathcal{U}_\ell}$. It suffices to show that for a fixed cluster $\mathcal{U}_\ell$, there is an equality
    \begin{equation}
        c_i \langle Q\Sigma_b ; a_1, a_2 \rangle^E_{\mathcal{U}_{\ell}} =     c_i \langle u^p \psi_b ; a_1, a_2 \rangle^E_{\mathcal{U}_{\ell}} \in \Lambda_{\le E}
    \end{equation}
    for all $i$; then we can take $E \to \infty$. 

    By cluster decomposition, the monomial terms in $q^\beta$ in both structure constants $c_i \langle Q\Sigma_b ; a_1, a_2 \rangle^E_{\mathcal{U}_{\ell}} =     c_i \langle u^p \psi_b ; a_1, a_2 \rangle^E_{\mathcal{U}_{\ell}}$ have the form $\beta = k [C_\ell]$ for some $k \ge 1$.

    Now fix $J' \in \mathcal{J}(X)^E$ such that the core curve $C_\ell$ is $J'$-holomorphic and moreover the cluster $(\mathcal{U}_{\mathrm{elem}}, J', C_\ell)$ is an elementary cluster; this can be achieved by taking the elementary cluster around $C_\ell$ and perturbing it away from $C_\ell$ to achieve the transversality condition $J' \in \mathcal{J}(X)^E$ (cf. \cite[Theorem 7.1]{IP-GV}. By the cluster isotopy theorem (Proposition \ref{prop:isotopy}), we have 

    \begin{equation}\label{eqn:cluster-isotopy}
    \begin{aligned}
    c_i \langle Q\Sigma_b ; a_1, a_2 \rangle^E_{\mathcal{U}_\ell}  &= \pm c_i \langle Q\Sigma_b ; a_1, a_2 \rangle^E_{\mathcal{U}_{\mathrm{elem}}} + \sum_k c_i \langle Q\Sigma_b ; a_1, a_2 \rangle^E_{\mathcal{U}_k}
    \\
    c_i \langle u^p \psi_b; a_1, a_2 \rangle^E_{\mathcal{U}_\ell}  &= \pm c_i \langle u^p \psi_b; a_1, a_2 \rangle^E_{\mathcal{U}_{\mathrm{elem}}} + \sum_k c_i \langle u^p \psi_b; a_1, a_2 \rangle^E_{\mathcal{U}_k}\\
    \end{aligned}
    \end{equation}
    where $(\mathcal{U}_k, J_k, C_k)$ are $E$-clusters with $[C_k] = d_k [C_\ell] \in H_2(X;\mathbb{Z})$ for $d_k \ge 2$. By the local computation (Proposition \ref{prop:elementary-cluster-qst-is-pcurv}) for elementary clusters, we know that 
    \begin{equation}
        c_i \langle Q\Sigma_b ; a_1, a_2 \rangle^E_{\mathcal{U}_{\mathrm{elem}}} = c_i \langle u^p \psi_b ; a_1, a_2 \rangle^E_{\mathcal{U}_{\mathrm{elem}}}.
    \end{equation}
    For each of the clusters $(\mathcal{U}_k, J_k, C_k)$ in the error terms $c_i \langle Q\Sigma_b ; a_1, a_2 \rangle^E_{\mathcal{U}_k}$ and $c_i \langle u^p \psi_b ; a_1, a_2 \rangle^E_{\mathcal{U}_{k}}$, we may now apply the cluster isotopy theorem again to deform these clusters to elementary clusters at the cost of introducing new clusters with core curves with homology classes $d[C_k]$ for $d>2$. This deformation process also fixes the sign ambiguity in \eqref{eqn:cluster-isotopy}, as both signs are determined by the orientation sign of the underlying embedded curve (cf. \cite[Equation (7.3)]{IP-GV}). After finitely many iterations, the homology classes of the core curves in the error terms have energy $\omega(N[C_\ell]) > E$ for large $N \gg 0$, implying that the corresponding $E$-truncations of the structure constants in $\Lambda_{\le E}$ vanish. Therefore
        \begin{equation}
        c_i \langle Q\Sigma_b ; a_1, a_2 \rangle^E =     c_i \langle u^p \psi_b ; a_1, a_2 \rangle^E \in \Lambda_{\le E}
    \end{equation}
    for a fixed $E$, as desired.
\end{proof}

\begin{rem}\label[rem]{rem:purely-quantum}
    Note that the part of the proof that uses the cluster formalism shows that the \emph{non-classical terms} of the two operations agrees. A priori, the classical term (the coefficient of $q^0$) may differ if $b \in H^2(X;\mathbb{F}_p)$ is a torsion class which does not lift to $H^2(X;\mathbb{Z})$, since
    \begin{equation}
        \mathrm{St}(b) = b^p - u^{p-1} b + u^{p-2}\theta (\beta b)
    \end{equation}
    has the term from the Bockstein $\beta b \in H^3(X;\mathbb{F}_p)$ which does not contribute to the (classical part of) the $p$-curvature. If $H^3(X;\mathbb{Z})$ has no torsion, this condition will be satisfied automatically.

\end{rem}

\subsection{Proof of Katz's formula}\label{ssec:katz-formula-proof}

As a corollary of the cluster decomposition we proved for the $p$-curvature, we may now also provide a proof of the Katz' formula.

\begin{proof}[Proof of \cref{thm:Cartierconjugation-body}]

Recall from \cref{lem:Cartier-computation} that the matrix of inverse Cartier operator $\mathcal{C}^{-1}$ in the preferred homogeneous basis is computed, and the quantum product $\mathbb{M}_{e_k}$ is also computed in \cref{lem:qproduct-computation}. Structure constants of both operators are determined in terms of Yukawa couplings, and therefore it remains to compute the $p$-curvature endomorphism in the preferred basis explicitly in terms of the Yukawa couplings.

We work in the preferred homogeneous basis given by $1 \in H^0(X)$, $e_i \in H^2(X;\mathbb{Z})$, $L_j \in \bar{H}^4(X;\mathbb{Z})$, and $[\mathrm{pt}] \in H^6(X;\mathbb{Z})$ and compute the $p$-curvature endomorphism $ \psi_{e_k}$ in the direction of $e_k \in H^2(X;\mathbb{Z})$.

Under the Poincar\'e pairing, the $p$-curvature operator is given in terms of the structure constants
\begin{equation}
  \left(u^p  \psi_{e_k} (a_1) , a_2 \right) =:   \left(u^p  \psi_{e_k} |_{q^\beta = 0} (a_1) , a_2 \right)  + \sum_{i \ge 0} c_i \langle u^p  \psi_{e_k}; a_1, a_2 \rangle (u,\theta)^i \in \Lambda_{\mathbb{F}_p}[u]\langle \theta \rangle
\end{equation}
where $(\cdot, \cdot)$ is the Poincar\'e pairing; by the cluster decomposition theorem (\cref{prop:pcurv-cluster-decomposition}), we know that the ($E$-truncations of) the structure constants admit a cluster decomposition,
\begin{equation}
    c_i \langle u^p  \psi_{e_k}; a_1, a_2 \rangle^E  = \sum_{\ell} c_i \langle u^p \psi_{e_k} ; a_1, a_2 \rangle ^E_{\mathcal{U}_\ell}.
\end{equation}
By the cluster isotopy theorem (\cref{prop:isotopy}), we may assume that each of the cluster $\mathcal{U}_\ell$ is elementary. If $\mathcal{U}_\ell$ is elementary, the corresponding contribution $c_i \langle u^p \psi_{e_k} ; a_1, a_2 \rangle ^E_{U_\ell}$ is computed by taking the restriction of cohomology classes $H^*(X) \to H^*(Y)$ for $Y \cong N_{C_\ell}$ the neighborhood of the core curve $C_\ell$ (cf. the proof of \cref{prop:elementary-cluster-qst-is-pcurv}). Following the notation from the previous proof, we denote the image of a cohomology class $a \in H^*(X;\mathbb{F}_p)$ under this pullback map by $a' \in H^*(Y;\mathbb{F}_p)$. It follows that
\begin{equation}
    c_i \langle  u^p \psi_{e_k}; a_1, a_2 \rangle^E_{\mathcal{U}_\ell}  = c_i \langle u^p  \psi_{e_k'} ; a_1', a_2' \rangle^E
\end{equation}
where the right hand side is the corresponding structure constant computed in $Y$.  The latter is completely determined by the computation in \cref{eqn:localP1-pcurv-formula}. We now explain how this computation yields the full determination of $(\psi_{e_k}(a_1), a_2)$.

To illustrate the method of computation, we explain the case $a_1 = e_{j_1}, a_2 = e_{j_2} \in H^2(X;\mathbb{Z})$. Since $H^2(Y)$ is rank $1$ generated by $e$, there are coefficients $m_j \in \mathbb{F}_p$ such that $e_j' = m_j e$. Then
\begin{equation}
    c_i \langle u^p \psi_{e_k'}; e_{j_1}', e_{j_2}' \rangle^E = m_{j_1} m_{j_2} m_k^p \cdot  c_i \langle u^p \psi_e; e, e \rangle^E = \begin{cases} - m_{j_1} m_{j_2} m_k \sum_{d=1}^{\mathrm{const}(E)} q^{pd} & i = p-1 \\ 0 & \mbox{otherwise} \end{cases}
\end{equation}
from \cref{eqn:localP1-pcurv-formula}. Here we used $m_k^p = m_k \in \mathbb{F}_p$. On the other hand, the Yukawa coupling $Y_{{j_1}{j_2}k}(q)  = ( \mathbb{M}_{e_k}( e_{j_1}), e_{j_2} )$ can also be computed via cluster decomposition as they are structure constants of quantum product operators:
\begin{equation}
    ( \mathbb{M}_{e_k}( e_{j_1}), e_{j_2} )^E = \sum_\ell  ( \mathbb{M}_{e_k}( e_{j_1}), e_{j_2} )_{\mathcal{U}_\ell}^E 
\end{equation}
where, under $Y \cong N_{C_\ell}$ we have
\begin{equation}
    ( \mathbb{M}_{e_k}( e_{j_1}), e_{j_2} )_{\mathcal{U}_\ell}^E  = m_{j_1}m_{j_2}m_k ( \mathbb{M}_{e}( e), e )^E = m_{j_1}m_{j_2}m_k \sum_{d=1}^{\mathrm{const}(E)} q^d.
\end{equation}
By passing to the limit $E \to \infty$, it follows that
\begin{equation}
     c_i \langle u^p  \psi_{e_k}; e_{j_1}, e_{j_2} \rangle = \begin{cases} -(Y_{j_1 j_2 k} (q^p) - Y_{j_1j_2 k }(0)) & i = p-1 \\ 0 &\mbox{otherwise} \end{cases}.
\end{equation}

We may similarly compute the other entries  $\left(u^p  \psi_{e_k} (a_1) , a_2 \right)$. The final result of the computation is as follows. Recall that
\begin{equation}
    \eta_\zeta(q) = \zeta + \frac{1}{p^3} G(q^p) - G(q)  = \zeta + \sum_{A \neq 0} n_A \left( \frac{1}{p^3} \mathrm{Li}_3(q^{p\beta}) - \mathrm{Li}_3(q^\beta) \right);
\end{equation}
our structure constants are given in terms of logarithmic derivatives of $\eta_\zeta(q)$: 

\begin{align}\label{eqn:p-curv-formula}
\psi_{e_k} (1) &=  - u^{-1} e_k -  u^{-2} \sum_{j, k} \delta_j\delta_k(\eta_{\zeta}(q)) L_j +  2 u^{-3} \delta_k  \eta_{\zeta}(q) [\mathrm{pt}], \\
\psi_{e_k} (u^{-1}e_i) &=  - u^{-2}\sum_{i, j, k} Y_{ijk}(q^p) L_j   + u^{-3} \delta_i \delta_k (\eta_{\zeta}(q)) [\mathrm{pt}],\nonumber \\
\psi_{e_k} (u^{-2}L_i) &= -u^{-3} [\mathrm{pt}],  \nonumber \\
\psi_{e_k} (u^{-3} [\mathrm{pt}])&= 0 \nonumber  .
\end{align}

Here, we used the assumption $p > 3$ for the computation of the classical part of the operation,
\begin{equation}
    u^p \psi_{e_k}|_{q^\beta = 0} = (b^p - u^{p-1}b) \smile = -u^{p-1} \cdot b \smile.
\end{equation}

The desired result follows by explicit matrix multiplication.
\end{proof}

\begin{rem}
    An alternative way to prove \cref{thm:Cartierconjugation-body} would be to show that the Frobenius intertwiner, and hence the Cartier operator, satisfy the cluster decomposition. Then one can obtain a proof of the equivalence of the two operations by verifying the equivalence for the local $\mathbb{P}^1$, and bootstrapping the general case from it. This alternative approach is essentially equivalent to the argument given above, as our construction of the $p$-adically integral Frobenius intertwiner depends on Gopakumar--Vafa integrality (and hence the cluster formalism, implicitly).
\end{rem}


\appendix
\section{$P$-adic integrality of $\zeta_p(3)$}
In this appendix, we will prove that $\zeta_p(3) \in \mathbb{Z}_p$ for $p \ge 5$. In order to do so, we will use the following classical Clausen–von
Staudt congruence \cite[Theorem 9.5.14, Corollary 9.5.15]{cohen-vol2} concerning even Bernouli numbers $B_k$:   
\begin{thm} \label{thm:ClausenvonStaudt}
For any even $k \in \mathbb{Z}_{>0}$ we have
\[
B_k \equiv -\sum_{(p-1)|k} \frac{1}{p} \pmod{1},
\]
where it is understood that $p$ is a positive prime number. In particular, the denominator of $B_k$ is the product of the primes $p$ such that $(p-1) \mid k$, where each such prime occurs to exactly the first power.
\end{thm}
\begin{lemma}\label{lem:pintegrality}
Let $p$ be an odd prime.
\begin{enumerate}
    \item[(a)] If $p \ge 5$, then $\zeta_p(3) \in \mathbb{Z}_p$.
    \item[(b)] If $p = 3$, then $v_3(\zeta_3(3)) = -1$ (in particular, $\zeta_3(3) \notin \mathbb{Z}_3$).
\end{enumerate}
\end{lemma}
\begin{proof}
To prove (a), suppose $p \geq 5$. We recall (see e.g. \cite[pg. 73]{candelas}) that: \begin{align} \zeta_p(3)= \operatorname{lim}_{m\to \infty} B_{(p-1)p^m-2}/2, \end{align} where the $B_{(p-1)p^m-2}$ are understood as $p$-adic numbers. By Theorem \ref{thm:ClausenvonStaudt}, $B_{(p-1)p^m-2} \in \mathbb{Z}_p$ for all $m$ because $(p-1)\nmid ((p-1)p^m-2)$. By the completeness of $\mathbb{Z}_p$, we conclude $\zeta_p(3) \in \mathbb{Z}_p.$ \vskip 5 pt
For (b), set $p=3$. Then because $(p-1)|(p-1)p^m-2$ by Theorem \ref{thm:ClausenvonStaudt}, $B_{(p-1)p^m-2}$ has a simple factor of $p=3$ in the denominator. It follows that $\operatorname{val}_p(B_{(p-1)p^m-2})=-1$ and the same for $\zeta_p(3)$.
\end{proof}

\bibliographystyle{amsalpha}
\bibliography{ref}

\end{document}